\documentclass{amsart}
\usepackage[utf8]{inputenc}
\pdfoutput=1

\usepackage{graphicx, booktabs,enumerate,microtype}
\usepackage[dvipsnames]{xcolor}
\usepackage{amssymb,amsmath,amsthm, stmaryrd, mathtools} 
\usepackage{tikz, tikz-cd}
\usepackage{hyperref}

\newcommand{\lm}[1]{\texorpdfstring{#1}{replacedLaTeXcode}}

\usepackage{mathdots}

\hypersetup{colorlinks=true, linkcolor=RoyalBlue, citecolor=JungleGreen}

\usepackage[foot]{amsaddr}

\usepackage{mathrsfs, mathtools}
\newcommand{\ms}[1]{\mathscr{#1}}

\let\oldtocsection=\tocsection

\let\oldtocsubsection=\tocsubsection

\let\oldtocsubsubsection=\tocsubsubsection

\renewcommand{\tocsection}[2]{\hspace{0em}\oldtocsection{#1}{#2}}
\renewcommand{\tocsubsection}[2]{\hspace{1em}\oldtocsubsection{#1}{#2}}
\renewcommand{\tocsubsubsection}[2]{\hspace{2em}\oldtocsubsubsection{#1}{#2}}

\newcounter{prcounter}

\title[Cohomological Automorphic Reps on Unitary Groups]{Statistics of Cohomological Automorphic Representations on Unitary Groups via the Endoscopic Classification}

\author{Rahul Dalal$^1$}
\address{$^1$Department of Mathematics, University of Vienna, 
Vienna, Austria}
\email{rahul.dalal@univie.ac.at}

\author{Mathilde Gerbelli-Gauthier$^2$}
\address{$^2$Department of Mathematics, University of Toronto,
Toronto, ON}
\email{mgg@math.toronto.edu}

\date{\today}

\DeclareFontFamily{U}{mathx}{}
\DeclareFontShape{U}{mathx}{m}{n}{<-> mathx10}{}
\DeclareSymbolFont{mathx}{U}{mathx}{m}{n}
\DeclareMathAccent{\widecheck}{0}{mathx}{"71}

\newcommand{\lf}{\left}
\newcommand{\ri}{\right}

\newcommand{\f}{\frac} 
 
\newcommand{\into}{\hookrightarrow}

\newcommand{\iso}{\xrightarrow{\sim}}
\newcommand{\wh}{\widehat}
\DeclareMathOperator{\sgn}{sgn}

\DeclareMathOperator{\rank}{rank}

\DeclareMathOperator{\Out}{Out}

\DeclareMathOperator{\Gal}{Gal}

\DeclareMathOperator{\Hom}{Hom}

\DeclareMathOperator{\Res}{Res}
\DeclareMathOperator{\Ind}{Ind}
\DeclareMathOperator{\Sym}{Sym}
\DeclareMathOperator{\vol}{vol}

\DeclareMathOperator{\ad}{ad}
\DeclareMathOperator{\Ad}{Ad}
\DeclareMathOperator{\tr}{tr}
\DeclareMathOperator{\image}{im}

\newcommand{\id}{\mathrm{id}}

\DeclareMathOperator{\Lie}{\mathrm{Lie}}

\newcommand{\Ld}[1]{{}^L\!{#1}}

\newcommand{\m}[1]{\mathbf{#1}}
\newcommand{\mf}[1]{\mathfrak{#1}}
\newcommand{\mc}[1]{\mathcal{#1}}
\newcommand{\td}[1]{\tilde{#1}}
\newcommand{\wtd}[1]{\widetilde{#1}}

\newcommand{\C}{\mathbb C}
\newcommand{\R}{\mathbb R}
\newcommand{\Q}{\mathbb Q}

\newcommand{\Z}{\mathbb Z}

\newcommand{\F}{\mathbb F}
\newcommand{\A}{\mathbb A}

\newcommand{\SL}{\mathrm{SL}}
\newcommand{\Sp}{\mathrm{Sp}}

\newcommand{\GL}{\mathrm{GL}}
\newcommand{\gl}{\mathfrak{gl}}

\newcommand{\Frob}{\mathrm{Frob}}

\newcommand{\eps}{\epsilon}
\newcommand{\om}{\omega}
\newcommand{\lb}{\lambda}
\newcommand{\Om}{\Omega}

\newcommand{\bs}{\backslash}
\newcommand{\1}{\m 1}

\newcommand{\ab}{\mathrm{ab}}

\newcommand{\disc}{\mathrm{disc}}

\newcommand{\el}{\mathrm{ell}}

\newcommand{\ur}{\mathrm{ur}}
\newcommand{\ssm}{\mathrm{ss}}

\newcommand{\can}{\mathrm{can}}

\newcommand{\pl}{\mathrm{pl}}

\newcommand{\der}{\mathrm{der}}

\newcommand{\sm}{\mathrm{sim}}
\newcommand{\temp}{\mathrm{temp}}
\newcommand{\ST}{\mathrm{ST}}

\newcommand{\inv}{\mathrm{inv}}

\DeclareMathOperator{\Irr}{\mathrm{Irr}}

\usepackage{stackengine,scalerel}
\DeclareMathOperator*{\prodf}{%
  \ThisStyle{\mathop{\ensurestackMath{\stackinset{c}{0\LMpt}{c}{}{%
  \rotatebox[origin=lb]{-90}{$\SavedStyle\scalerel*{\oplus}{%
  i}$}}{\SavedStyle\prod}}}}}

\newcommand{\EP}{\mathrm{EP}}

\newcommand{\ten}{\otimes}
\newcommand{\ip}[2]{\langle #1, #2 \rangle}
\newcommand{\dom}{\backslash}
\newcommand{\BC}{\mathbb C}
\newcommand{\BR}{\mathbb R}
\newcommand{\BQ}{\mathbb Q}

\newcommand{\BA}{\mathbb A}

\newcommand{\mH}{\mathcal{H}}

\newcommand{\mS}{\mathcal{S}}

\newcommand{\fa}{\mathfrak{a}}
\newcommand{\fu}{\mathfrak{u}}

\newcommand{\fp}{\mathfrak{p}}
\newcommand{\fq}{\mathfrak{q}}
\newcommand{\fg}{\mathfrak{g}}
\newcommand{\fl}{\mathfrak{l}}

\newcommand{\fn}{\mathfrak{n}}
\newcommand{\fk}{\mathfrak{k}}
\newcommand{\ft}{\mathfrak{t}}

\newcommand{\fm}{\mathfrak{m}}

\newtheorem{thm}{Theorem}[subsection]
\newtheorem{prop}[thm]{Proposition}

\newtheorem{cor}[thm]{Corollary}
\newtheorem{lem}[thm]{Lemma}
\newtheorem{conj}[thm]{Conjecture}

\newtheorem*{goal}{Goal}

\theoremstyle{remark}
\newtheorem*{note}{Note}
\newtheorem*{ex}{Example}
\newtheorem*{rmk}{Remark}

\theoremstyle{definition}
\newtheorem{dfn}[thm]{Definition}

\numberwithin{equation}{subsection}
\numberwithin{table}{subsection}

\newcommand{\red}[1]{\textcolor{red}{#1}} 

\begin{document}

\begin{abstract}
Consider the family of automorphic representations on a unitary group with 
cohomological factor $\pi_0$ at infinity and given split 
level. We compute statistics of this family as the level goes to infinity. For unramified unitary groups and a large class of $\pi_0$, we use the endoscopic classification of representations to compute the exact leading term for 
counts of representations and averages of Satake parameters. The bounds on our error terms are 
similar to previous work by Shin-Templier who
studied the case of discrete series at infinity. We also prove new upper bounds for all cohomological representations. 

This has 
many corollaries: new exact asymptotics on the growth of cohomology in certain towers of locally symmetric spaces, an averaged Sato-Tate equidistribution law for spectral families with specific non-tempered cohomological components at infinity, and the Sarnak-Xue density hypothesis for cohomological representations at infinity on all unitary groups of rank $\geq 5$.

\end{abstract}

\maketitle

\tableofcontents


\section{Introduction}
\subsection{Context}

\subsubsection{Statistics of automorphic representations}
Let $G$ be a reductive group over a number field $F$. Automorphic representations of $G$ are very roughly irreducible subrepresentations of $L^2(G(F) \bs G(\A))$ under right multiplication by $G(\A)$. 
They encode information important to many 
applications---for example Galois representations through the Langlands program, the geometry of locally symmetric spaces through automorphic decompositions of their cohomology, and so-called 
expanders graphs used for computer algorithms through constructions akin to that of 
\cite{LPS88}.

Every automorphic representation $\pi$ has a tensor product decomposition:
\[
\pi = {\bigotimes_v}' \pi_v
\]
into unitary irreducible representations $\pi_v$ of $G(F_v)$ over places $v$ of $F$. Applications usually depend on a key question: among all possible tensor products of $\pi_v$, which ones are automorphic, 
i.e. are represented as functions in $L^2(G(F) \bs G(\A))$?

This paper broadly 
focuses on an easier version of this key question---that of computing \emph{statistics}. We will consider families of automorphic representations in the spectral sense of \cite{SST16}, i.e. satisfying local conditions on each $\pi_v$. Then we will estimate as well as possible the asymptotics of the (weighted) counts of automorphic representations satisfying these conditions as they ``go to infinity''. For example, when~$G = \GL_2$, we could look at the ``$m$th moment'' of the Hecke eigenvalue at some place $p$ averaged over weight-$k$, level-$N$ holomorphic modular forms as $N \to \infty$. 

\subsubsection{Our specific case}
More specifically, we are interested in the case where the component $\pi_\infty$ at infinity is restricted to a specific representation---this can be thought of as fixing the ``qualitative type'' of an automorphic representation; for example, ``holomorphic Siegel modular form of weight $\vec v$''. We organize such questions\footnote{There are of course many powerful automorphic statistics results that don't fit into this story---as a good ``most general'' representative \cite{FLM15} considers $\pi_\infty$ contained in any subset of the unitary dual with finite, non-zero volume. We are not giving a full literature review here.}
by an informal ranking of the complexity of an automorphic representation based on the component $\pi_\infty$.

For us, the simplest $\pi_\infty$ are the discrete series that can be realized as explicit subrepresentations of $L^2(G_\infty)$.
Statistics of discrete-at-infinity automorphic representations can be well-understood through a version of Arthur's trace formula in \cite{Art89}. The techniques of \cite{ST16}, together with \cite{Fer07} as used in first author's thesis \cite{Dal22}, build on this 
trace formula to provide asymptotic estimates and error terms in the general discrete-at-infinity case.

This paper goes beyond discrete series and studies automorphic representations with cohomological component at infinity---i.e. component at infinity that has some non-trivial relative Lie algebra cohomology. There is a critical new complication here: when a group has discrete series, non-discrete cohomological representations are always non-tempered. In particular, general cohomological automorphic representations can 
represent violations of the Ramanujan conjecture. They can therefore be very sparse in the automorphic spectrum and difficult to isolate. 

Recent results suggest that known cases of the endoscopic classification \cite{Art13, Mok15, KMSW14} 
are a good way to study cohomological automorphic representations. Work of Marshall, Shin and the second author have used it to provide good upper bounds for counts on unitary groups, see
\cite{Mar14, MS19, Ger20}.
In specific simpler cases, explicit counts have even been computed: 
\cite{CR15} 
and \cite{Tai17} consider level-$1$ representations on classical groups and \cite{RSY22} develops techniques that apply to low-level automorphic representations on $\Sp_4$. 

This work attempts to organize and synthesize the bounds of Marshal, Shin, and the second author together with an inductive analysis used by
Ta\"ibi~\cite{Tai17} into a proposal for a general method to understand statistics of cohomological automorphic representations. While our proposed method in general depends on some wide-open and difficult problems in local representation theory, we do implement it explicitly in some specific cases on unitary groups. This gives the most general understanding of cohomological asymptotic statistics on unitary groups to date. 

We emphasize in particular that, in many cases, we compute exact leading terms together with estimates on sub-leading terms. This gives us applications to the growth of cohomology of locally symmetric spaces, Sato-Tate equidistribution averaged over families of automorphic representations, and the Sarnak-Xue density conjecture. We motivate these in more detail:

\subsubsection{Application: growth of cohomology} \label{section intro coho} A motivating problem for these statistical computations is the growth of cohomology in towers of arithmetic groups. Let~$\Gamma \subset G_\infty = G(F \ten_{\BQ} \BR)$ be a (neat and cocompact for simplicity) arithmetic lattice. Then the group cohomology of $\Gamma$, or equivalently, the Betti cohomology of the locally symmetric space $\Gamma \dom G_\infty/K_{\infty}$ for $K_\infty$ a maximal compact subgroup, is computed in terms of automorphic forms via Matsushima's formula: 
\[ 
H^*(\Gamma, \BC) = \bigoplus_{\pi_\infty \text{ cohomological}} m(\pi_\infty, \Gamma) H^*(\fg,K_\infty; \pi_\infty), 
\] 
where $m(\pi_\infty, \Gamma) = \dim \Hom_{G_\infty}(\pi_\infty, L^2(\Gamma \dom G_\infty))$. Beyond the fact that locally symmetric spaces provide a rich class of examples of manifolds, the cohomology~$H^*(\Gamma, \BC)$ is also of interest because it carries an action of Hecke operators, making it a generalization of modular forms of weight $>1$. The Hecke eigensystems that arise should correspond to Galois representations via the Langlands program. 

Outside of some low-rank examples, dimensions of $H^*(\Gamma, \BC)$ are only known for specific lattices $\Gamma$, see for example \cite{ash2008cohomology, gunnells2021cohomology}. A fruitful approach is to consider these cohomology groups 
\emph{in towers}: one fixes a sequence~$\Gamma_n$ of (typically nested) finite-index subgroups, and studies the asymptotics of $\dim H^*(\Gamma_n, \BC)$ as $n \to \infty$. Without giving a systematic survey of this problem, we note there have been multiple approaches: topological constructions e.g. \cite{Mil76,  TshishikuStudenmund}, non-abelian Iwasawa theory methods as in \cite{calegari2009bounds}, and, beginning with the work of DeGeorge-Wallach \cite{degeorge1978limit}, what can be referred to as spectral approaches. 

Going back to our complexity ranking, it is known that cohomological discrete series only contribute to cohomology in degree $(1/2)\dim (G_\infty/K_\infty)$ (we assume~$G_\infty$ has anisotropic center for simplicity). Thus the exact asymptotics of \cite{degeorge1978limit}, and later \cite{clozel1986limit,savin1989limit}, 
show that when~$G_\infty$ has discrete series, the middle degree of cohomology grows like the volume of the corresponding symmetric space and the growth of the lower degrees is slower. This leaves open the question of more precise upper bounds for other degrees of cohomology, and it seems (see \cite[\S 1]{SX91}) that this motivating question was at the heart of the discussion which led to the formulation of the Sarnak-Xue conjecture discussed in \ref{section intro SX}. 

There has been progress on 
bounding growth of Betti numbers,  
from the lower bounds of \cite{Cos09, CM11, KS15growth, TshishikuStudenmund} to the vanishing results of \cite{CD06automorphic, Clo93Coho} to investigations of 
more general sequences of lattices \cite{AB+L2}. Starting with~\cite{SX91}, there has also been progress on upper bounds, e.g. the 
power saving on the trivial bound in \cite{calegari2009bounds}. 
Most influential for us is a series \cite{CM11, Mar14, Mar16} of works of Marshall and his collaborators, giving (in some cases sharp
) upper bounds on cohomology growth. Most recently, in \cite{MS19} Marshall-Shin obtained upper bounds for all degrees for lattices in $U(N,1)$. In this article, we prove new upper bounds for growth of Betti numbers of lattices in unitary groups, as well exact asymptotics that show e.g. that the bounds of Marshall-Shin are sharp in every other degree.

\subsubsection{Application: averaged Sato-Tate}
Let $\pi$ be an automorphic representation on a reductive group $G$. A Sato-Tate result for $\pi$ is a statement that the Satake parameters of the unramified components $\pi_v$ are equidistributed over~$v$ according to a Sato-Tate distribution $\mu_\pi$ determined by 
$\pi$. This should be thought of as a generalized, automorphic-side 
analogue of the classical Sato-Tate conjecture 
for the equidistribution over $p$ of the coefficients $a_p$ associated to point counts over~$\F_p$ of elliptic curves. See \cite[\S1.1]{ST16} for a full introduction to the problem---in particular,~\cite{Ser94} states a very general Galois-side version of the conjecture and~\cite{BLGG, BLGHT11} prove it for restrictions of scalars of $\GL_1$ and $\GL_2$. 

Unfortunately, even 
defining the Sato-Tate distribution $\mu_\pi$ 
for a single $\pi$ depends on more-or-less the full conjectures of Langlands functoriality. Extremely roughly, $\pi$ should correspond to another reductive group $H_\pi$ and $L$-map~$\varphi: \Ld H_\pi \into \Ld G$---the smallest such that the conjectural global $L$-parameter for~$\pi$ factors through $\varphi$. The law $\mu_\pi$ should then be thought of as pushforward of a measure~$\mu_{H_\pi}$ from the space of Satake parameters of $H_\pi$ to that of $G$. For general~$\pi$ on high-rank groups, Sato-Tate results therefore appear unapproachable. In fact, empirically measuring the Sato-Tate distribution for $\pi$ is arguably one of the key currently accessible pieces of evidence for the existence of this conjectural $H_\pi$ in the first place.

Following \cite{ST16}, we therefore instead study Sato-Tate laws averaged over some family $\mc F$.  Heuristically, representations $\pi \in \mc F$ can have many different $H_\pi$. However, for reasonable families, we expect the log of the count of forms associated to $\pi \in \mc F$ with $H_\pi = H_0$ to increase with the dimension of $H_0$. Therefore, we should expect most $\pi \in \mc F$ to have $H_\pi$ be some maximum value $H_{\mc F}^{\max}$. In particular, if we look at Satake parameters $\pi_v$ over all places $v$ \emph{and} all $\pi \in \mc F$, the distribution for $H_\pi = H_{\mc F}^{\max}$ should dominate.

Such ``averaged Sato-Tate laws'' end up being far easier to establish. For example,~\cite{ST16} studied families of automorphic representations on $G$ with discrete series at infinity and showed they satisfy averaged Sato-Tate laws coming from~$G$ itself. This corresponds to the  heuristic that most $\pi$ with discrete series at infinity should be ``primitive'': i.e. have $H_\pi = G$ (see e.g. Theorem 1.3 in \cite{KWY21} for a precise version of such a claim).

In Section \ref{sectionSatoTate}, we instead study families with certain specific (possibly non-tempered) cohomological  
components $\pi_0$ at infinity,  
which we refer to as ``odd GSK-maxed'' 
in \ref{GSKmE}. In this case, the resulting averaged Sato-Tate laws are not those from $G$ itself, but rather from  certain $H_{\pi_0}$ that we explicitly compute from $\pi_0$. The 
pairs $(\varphi, H_{\pi_0})$ are not in general endoscopic embeddings, but instead compositions of  
these with tensor-product maps analogous to $\GL_n \times \GL_m \to \GL_{nm}$. This can be taken as evidence for a speculative interpretation that is nevertheless clearly suggested by the form of the endoscopic classification: that the decomposition of an $A$-parameter in terms of cuspidal parameters is literally realizing the corresponding packet as a functorial transfer from a smaller, possibly \emph{non-endoscopic} group.

\subsubsection{Application: Sarnak-Xue density} \label{section intro SX}
The Sarnak-Xue density conjecture aims to quantify the failure of the na\"ive generalized Ramanujan conjecture---that all the local components of cuspidal automorphic representations are tempered. Na\"ive Ramanujan was the generalization and translation to representation-theoretic, adelic language of the classical Ramanujan conjecture bounding Hecke eigenvalues of the~$\Delta$ function---see \cite{Sar05} for a full introduction. It was found to be false even for split~$G$ through counterexamples constructed in~\cite{HPS79}. 

Luckily, many desired applications don't require the absence of non-tempered representations: only that there to be not too many. Sarnak and Xue in \cite{SX91} conjectured a precise meaning for ``too many'': the necessary upper bound on the asymptotic growth rate of representations $\pi$ with component $\pi_v$ non-tempered. 

Their original conjecture was stated in terms of classical, real locally symmetric spaces. 
Rephrasing it in more modern language and focusing on the case of cohomological representations, let $G/F$ be reductive and let $U_i$ be a sequence of open compact subgroups of $G^\infty := G(\A^\infty)$ decreasing to the identity. Let $\Gamma_i = G(F) \cap U_i$ and choose a cohomological representation $\pi_0$ of $G_\infty$. Then:
\begin{conj}[Cohomological Sarnak-Xue] \label{conj SX intro}
Let $m(\pi_0, \Gamma_i \bs G_\infty)$ be the multiplicity of $\pi_0$ in $L^2(\Gamma_i \bs G_\infty)$. Then for all $\eps > 0$,
\[
m(\pi_0, \Gamma_i \bs G_\infty) \ll_\eps \vol(\Gamma_i \bs G_\infty)^{\f2{p(\pi_0)} + \eps}
\]
where $p(\pi_0)$ is the infimum over $p>2$ such that the (spherically finite) matrix coefficients of $\pi_0$ are in $L^p(G_\infty)$. 
\end{conj}
As alluded to in \ref{section intro coho}, 
if $\pi_0$ is discrete series (in which case $p(\pi_0) = 2$) then 
\[
m(\pi_0, \Gamma_i \bs G_\infty) \asymp  \vol(\Gamma_i \bs G_\infty).
\] 
On the other hand, if $\pi_0$ is a character, (i.e. $p(\pi_0) = \infty$), then $m(\pi_0, \Gamma_i \bs G_\infty) \asymp 1$. Conjecture \ref{conj SX intro} is therefore a claim that an asymptotically negligible fraction of automorphic representations have non-tempered component $\pi_0$ at infinity and further that the quantitative strength of ``asymptotically negligible'' depends on the failure of temperedness $p(\pi_0)$, interpolating between the cases $p = 2$ and $p = \infty$.

Many analytic applications are discussed in \cite{GK20} including those from \cite{EP22} relating to constructions (so-called golden gates and Ramunjuan complexes) used in computer science. Many recent breakthroughs have proved the conjecture in specific cases: for example, \cite{Blo22} proved a version of the conjecture at infinity for Maass forms on $\GL_n$ using the Kuznetsov formula and \cite{FHMM20} proved many versions for Maass forms on products of $\SL_2 \R$ and $\SL_2 \C$ using Arthur's trace formula.  

Most important in our context is the work of Marshall and collaborators applying Arthur's classification to 
cohomological $\pi_0$ on unitary groups. The most general results are in \cite{MS19} and prove the cohomological Sarnak-Xue conjecture at split level for groups that are $U(N,1)$ at infinity. In this paper, through a more powerful global framework, we can more fully leverage Marshall-Shin's local bounds and prove 
the cohomological Sarnak-Xue conjecture 
at split level 
for \emph{all} unitary groups that don't have a $U(2,2)$ factor at infinity.


\subsection{Results}
To make this all precise, let $E/F$ be a CM extension of number fields (i.e. a totally imaginary quadratic extension of a totally real field). Let $G/F$ be a unitary group that splits over $E$, so that $G$ has discrete series at infinity. Denote the discrete automorphic spectrum of $G$ by $\mc{AR}_\disc(G)$. We prove two main results, both conditional on the endoscopic classification of representations 
in \cite{KMSW14}. 

First, assume $E/F$ is unramified at all finite places. Let $\pi_0$ be a in certain class of good cohomological representations $\pi_0$ of $G_\infty$ that satisfy a condition of being odd GSK-maxed (as in Definition \ref{GSKmE}) and a technical parity condition from Lemma \ref{etacharoddGSK}. Theorem \ref{mainexact} then finds explicit constants $R(\pi_0)$ and $M(\pi_0) \neq 0$ such that for split-level principal congruence subgroups $U$ of $G^\infty$ as $\vol(U) \to 0$: 
\[
\vol(U)^{R(\pi_0)/\dim G} L(U) \sum_{\pi \in \mc{AR}_\disc(G)} \1_{\pi_\infty = \pi_0} \dim((\pi^\infty)^U) = M(\pi_0) + O(\vol(U)^C)
\]
for some correction factor $L(U)$ made precise in the theorem statement and some constant $C \geq 1/\dim G$. The sum should be thought of as the multiplicity of $\pi_0$ inside $L^2$ functions on the automorphic quotient $G(F) \bs G(\A)/ U$ (which has volume proportional to $\vol(U)^{-1}$).

The theorem also allows weighting the count by a Weyl-symmetric polynomial $P$ in the Satake parameters $s_{\pi_v}$ at a place $v$ where $U$ has hyperspecial component:
\begin{multline}\label{intro exact bound} 
\vol(U)^{R(\pi_0)/\dim G} L(U) \sum_{\pi \in \mc{AR}_\disc(G)} \1_{\pi_\infty = \pi_0} \dim((\pi^\infty)^U) P(s_{\pi_v}) \\ = M(\pi_0)M(P) + O(\vol(U)^C q_v^{A + B \deg P})
\end{multline}
for some explicit constant $M(P)$ and inexplicit constants $A,B$.  

The second result, Theorem \ref{mainupper}, applies to any CM extension $E/F$ 
and $\pi_0$ arbitrary cohomological. It only provides upper bounds 
\begin{equation}\label{intro upper bound}
\sum_{\pi \in \mc{AR}_\disc(G)} \1_{\pi_\infty = \pi_0} \dim((\pi^\infty)^U) P(s_{\pi_v}) = O(\vol(U)^{-R(\pi_0)/\dim G} q_v^{A + B \deg P}).
\end{equation}
When $\pi_0$ is discrete series, we recover the ``trivial bound'' of $R(\pi_0) = \dim G$. Otherwise, we get an improvement $R(\pi_0) < \dim G$. 

Though sharp upper bounds for cohomology growth are known in several cases (see \cite{CM11,Mar14}, as well as \cite{BMM17} which also points towards possible extensions), we believe that \eqref{intro exact bound} gives the first exact asymptotic for multiplicity counts of cohomological non-tempered automorphic representations, as well as the first with an estimate on the sub-leading term for this same class. Together with the upper bounds \eqref{intro upper bound}, it gives many new corollaries:

\subsubsection{Corollary: cohomology}
Our results give bounds for the growth of Betti numbers in towers of arithmetic manifolds. In this context, it is traditional and simplest, though by no means necessary, to fix $G_\infty$ isomorphic to $U(p,q)$, with $p+q = N$, at one infinite place, and compact at all the others. Letting~$U = K(\mf n)$ be the principal congruence subgroup associated to the ideal $\mf n$, 
 the resulting $\Gamma(\fn) = G(F) \cap K(\fn)$ are then cocompact lattices in $U(p,q)$. Our results in this context imply two types of corollaries. 

First, using the bounds of \ref{section on LM results} and the explicit description of packets of cohomological representations in \ref{sectionunitaryparamcombo} and \ref{upqparam}, we give an algorithm to compute upper bounds on the growth of the Betti numbers~$h^k(\Gamma(\fn))$ and of the dimensions~$h^{p,q}(\Gamma(\fn))$ of any piece of the Hodge decomposition. Though we don't expect the resulting upper bounds to be sharp in general, they should be in many cases and they always give a non-trivial power saving over the volume of $X(\fn) = \Gamma(\fn) \dom G_\infty/K_\infty$ in degrees strictly below $\frac{1}{2}\dim X(\fn)$. 
The algorithm, described in \ref{section upper bounds algo}, is as follows: 
\begin{enumerate}
    \item Given a degree of cohomology or a Hodge weight, list all representations~$\pi_\infty$ for which the corresponding $(\fg,K)$-cohomology is nonvanishing using the parameterization of \ref{upqparam}. 
    \item For each representation $\pi_\infty$, compute the shape $\Delta^{\max}(\pi_\infty)$ as in \ref{subsection ArthurSL2Max}. 
    \item Use Theorem \ref{mainupper} to give upper bounds $$ \dim\Hom(\pi_\infty, L^2(G(F)\dom G(\BA)/K(\fn))) \ll |\fn|^{R(\Delta^{\max}(\pi_\infty))}.$$ 
    \item From Matsushima's formula and the growth $\gg |\fn_i|^{1-\epsilon}$ of components of $G(F)\dom G(\BA)/K(\fn_i)$, deduce upper bounds for $h^k(\Gamma(\fn))$. 
\end{enumerate}
In practice, the combinatorics involved in 
(1) and (2) rapidly get complicated, but in some cases, the bounds can be expressed succinctly. For example, when $r = \min(p,q)$ is the smallest degree carrying the cohomology of a non-trivial representation, the contributions appear in weights $(r,0)$ and $(0,r)$ and we have 
\[ 
h^{r,0}(\Gamma(\fn)) + h^{0,r}(\Gamma(\fn)) \ll_\epsilon |\fn|^{N+\epsilon}. 
\] 

As a second corollary, our exact asymptotics give lower bounds on a range of degrees 
when $E/F$ is unramified. For example, using again $r = \min(p,q)$, Corollary \ref{cohmain} exhibits lattices such that for $1 \leq j \leq |p-q|-1$ and $j \not \equiv N \mod 2$, 
\[ 
h^{kj, (r-k)j}(\Gamma(\fn)) \gg |\fn|^{Nj}, \quad 0 \leq k \leq r. 
\] 
These results apply to a wider range of degrees when $U(p,q)$ is farther away from being quasisplit. In the extremal case where the noncompact part of $G_\infty$ is $U(N-1,1)$, we deduce that in degrees $j$ whose parity is opposite to that of $N$, the upper bounds of \cite{MS19} are sharp. 

\subsubsection{Corollary: Averaged Sato-Tate}
The error bound on our sub-leading term in Theorem \ref{mainexact} is as strong as in \cite{ST16} so we can mimic their argument and prove an averaged Sato-Tate law. This is Theorem \ref{SatoTate}. 

More specifically, given an odd GSK-maxed representation $\pi_0$ on an unramified unitary group $G$, we compute an unordered sequence of pairs $((T_i, d_i))_{1 \leq i \leq r}$ that is common to all elements of the set $\Delta^{\max}(\pi_0)$ described by the algorithm in \S\ref{deltamaxalgo}. To this list of pairs we associate a group
\[
H_{\pi_0} = (U_{E/F}(T_1) \times U_{E/F}(1)) \times \cdots \times (U_{E/F}(T_r) \times U_{E/F}(1))
\]
and an $L$-embedding $\varphi : \Ld H_{\pi_0} \to \Ld G$ constructed in three stages: first we embed the second coordinate of each pair into $\Ld U_{E/F}(d_i)$ through the cocharacter corresponding to the parameter of the trivial representation. Then we take the tensor product embedding
\[
\Ld (U_{E/F}(T_1) \times U_{E/F}(d_i)) \to \Ld U_{E/F}(T_i d_i)
\]
followed by the diagonal embedding
\[
\prod_i \Ld U_{E/F}(T_i d_i) \into \Ld U_{E/F}(T_1d_1 + \cdots + T_r d_r) = \Ld G. 
\]
The group $\Ld H_{\pi_0}$ has a canonical Sato-Tate measure on the space of Satake parameters for each splitting type $\theta$ of prime in $E/F$. We let $\mu^{\ST(\pi_0)}_\theta$ be the pushforward of this measure to $G$. 

For each finite place $v$ of the right splitting type and open compact $U \subseteq G^\infty$, we then define the empirical measure:
\[
\mu^{\pi_0}_{U,v} = \sum_{\substack{\pi \in \mc{AR}_\disc(G) \\ \pi_\infty = \pi_0}} \dim \lf((\pi^\infty)^U\ri)\delta_{s_{\pi_v}}
\]
as a sum of delta-measures on the space of Satake parameters $s_{\pi_v}$. Then Theorem~\ref{SatoTate} show that for certain sequences of $U_i$ and $v_i$ such that $\vol(U_i)^{-1}$ grows much faster than $q_{v_i}$, we have weak convergence
\begin{equation}\label{intro Sato-Tate}
C(\pi_0, U_i)^{-1} \mu^{\pi_0}_{U_i,v_i} \to \mu^{\ST(\pi_0)}_\theta
\end{equation}
for an appropriate scaling factor $C(\pi_0, U)$ as long as a parity condition from Lemma~\ref{etacharoddGSK} holds. This can be heuristically interpreted as evidence that most $\pi \in \mc {AR}_\disc(G)$ with $\pi_\infty = \pi_0$ are functorial transfers from $H_{\pi_0}$ through $\varphi$.

\subsubsection{Corollary: Sarnak-Xue}
We prove in Theorem \ref{SarnakXue} that the bounds in Theorem \ref{mainupper} achieve, and often beat, the Sarnak-Xue bounds, 
provided that $G_\infty$ does not have a $U(2,2)$ factor. It is worth noting that the local bounds proved in \cite{MS19} were already good enough to achieve the Sarnak-Xue threshold. The new work here is global---using them as an input in a more general framework that allows us to bound error terms appearing when one moves beyond $U(N,1)$.

Despite our improvements on Sarnak-Xue, we do not expect our bounds to be optimal. Through some heuristics relating to GK-dimension, we conjecture an optimal exponent $R_0(\pi)$ in Section \ref{sectionconjecturalbound}. By the algorithm at the end of \S\ref{deltamaxalgo}, we associate to $\pi_0$ an unordered sequence of pairs $((T_i, d_i))_i$ corresponding to the ``most tempered'' decomposition of Arthur parameters into cuspidals that allows component $\pi_0$ at infinity. As in the Sato-Tate discussion above, this sequence should be thought of as determining an $L$-embedding $\Ld H_{\pi_0} \into \Ld G$ with respect to which almost all forms associated to $\pi$ with $\pi_\infty = \pi_0$ are primitive. Then:
\begin{conj}
The optimal exponent on $\vol(U)$ in \eqref{intro upper bound} is
\[
R_0(\pi_0) = \f 12 \lf(N^2 - \sum_i T_i^2 d_i \ri) + \sum_i \lf(T_i^2 + \f12 T_i (T_i - 1)(d_i^2 - 1) \ri).
\]
\end{conj}
We compare our bound $R$, the optimal bound $R_0$, the Sarnak-Xue threshold, and the trivial bound $\dim G$ in many cases in Table \ref{grcomp}.

\subsubsection{Conditionality}
Our argument depends heavily on Mok's and Kaletha-Minguez-Shin-White's endoscopic classifications for unitary groups \cite{Mok15} and \cite{KMSW14}. Both of these depend on the unpublished weighted twisted fundamental lemma. The second in addition pushes many technical details to a specific reference ``KMSb'' that is not yet publicly available. 

We note that the dependence of \cite{Art13} on its unpublished references A25-27 and \cite{Mok15} on their unitary analogues has recently been resolved by \cite{AGIKMS24}.  

\subsection{Summary of Argument}
We prove the result using the Arthur-Selberg trace formula $I_\disc^G$ (see \cite{Art05} for a review). For the purposes of this project, this is the closest known approximation to an explicit ``geometric side'' formula for 
\[
\sum_{\pi \in \mc{AR}_\disc(G)} \tr_\pi f
\]
for a compactly supported, smooth test function $f$ on $G(\A)$. There are two main obstacles in applying it directly to our statistical problem:
\begin{itemize}
\item Given our chosen $\pi_0$, we would need to find a smooth compactly supported test function $f_\infty$ such that $\tr_\pi f_\infty = \1_{\pi_\infty = \pi_0}$ to pick out only automorphic representations with component $\pi_0$ at infinity. This would be the simplest way to understand
\[
m(\pi_0, f^\infty) := \sum_{\substack{\pi \in \mc{AR}_\disc(G) \\ \pi_\infty = \pi_0}} \tr_{\pi^\infty}(f^\infty) = \sum_{\pi \in \mc{AR}_\disc(G)} \tr_\pi(f_\infty f^\infty).
\]
\item The errors terms in the trace formula are quite inexplicit in general. 
\end{itemize}
Both these obstacles can be removed for the case of $\pi_0$ a discrete series. Here, $f_\infty$ can be chosen to be a pseudocoefficient by \cite{CD90}. The results of \cite{Art89} (with some addenda in \cite{Fer07}) then show that the geometric side of Arthur's invariant trace formula $I^G_\disc(f_\infty f^\infty)$ simplifies to something tractable with $f_\infty$ a pseudocoefficient. In \cite{ST16} (with some addenda in \cite{Dal22}), this formula was understood well enough for the purposes of computing asymptotics with error terms as we desire.

As soon as we try to generalize to all cohomological $\pi_0$, finding $f_\infty$ becomes a much larger issue---in fact, for non-tempered cohomological $\pi_0$, there is no such test function by the results of \cite{CD90}. The endoscopic classification of \cite{KMSW14} gives a way out: $\pi_0$ is contained the $A$-packet $\Pi_{\psi_\infty}$, a finite set of unitary irreducible representations attached to an $A$-parameter $\psi_\infty$. It turns out that we can find a pseudocoefficient 
which, while it doesn't isolate $\pi_0$ amongst all unitary irreducibles, does so amongst those with which it shares an $A$-packet (Lemma \ref{charps}). 

Next, the endoscopic classification also gives a decomposition
\[
I^G_\disc(f_\infty f^\infty) = \sum_{\psi \in \Psi_\el(G)} I_\psi^G(f_\infty f^\infty)
\]
into pieces corresponding to global $A$-parameters $\psi$. The $I^G_\psi$ term only involves traces against automorphic representations $\pi$ such that $\pi_\infty \in \Pi_{\psi_\infty}$ where $\psi_\infty$ is the associated local parameter at infinity. Understanding $m(\pi_0, f^\infty)$ is then reduced to finding an explicit, geometric expression for the part of the trace formula corresponding only to $A$-parameters $\psi$ with component $\psi_\infty$ at infinity such that $\pi_0 \in \Pi_{\psi_\infty}$. 

We do something slightly different, defining instead a global invariant $\Delta$ called the \emph{(refined) shape} of an $A$-parameter $\psi$. The $\Delta$'s have two key properties:
\begin{itemize}
\item $\Delta$ determines the local parameter $\psi_\infty$ at infinity if it is cohomological,
\item Following \cite{Tai17}, there is a an inductive method to write
\[
I^G_\Delta(f_\infty f^\infty) := \sum_{\psi \text{ with inv. } \Delta} I^G_\psi(f_\infty f^\infty)
\]
as a linear combination of terms of the form $I^H_\disc(f'_\infty (f')^\infty)$ with the $f'_\infty$ pseudocoefficients on smaller groups $H$---i.e. terms that are already understood explicitly by \cite{ST16} and \cite{Dal22}. 
\end{itemize}
Summing the inductive expressions for those $\Delta$ that correspond to $\psi_\infty$ with $\pi_0 \in \Pi_{\psi_\infty}$, we would therefore get an explicit geometric formula for 
$
m(\pi_0, f^\infty)$. 

For technical reasons, we instead work with the analogous summand $S_\Delta$ of Arthur's stable $S_\disc$ (see \cite[\S3.1-3.3]{Art13}). The two end up being more or less interchangeable for asymptotics by the ``hyperendoscopy'' techniques of \cite{Fer07} as used in \cite{Dal22}. 

\subsubsection{Outline of the Paper}

Sections \ref{sectionAC} and \ref{sectionAJ} give background material: \S\ref{sectionAC} is focused on the endoscopic classification and \S\ref{sectionAJ} on the real representation theory surrounding cohomological representations, their $A$-packets, and the pseudocoefficients. The definition of refined shape is discussed in \S\ref{sectionshapes} while the inductive procedure to understand $S_\Delta$ is explained in \S\ref{strategy}. Going from~$S_\Delta$ to~$m(\pi_0, f^\infty)$ is the work of \S\ref{sectionlimitmult}. 

The inductive procedure of~\S\ref{strategy} in general requires the construction and computation of certain transfers of~$f^\infty$: the conjectural stable and ``Speh'' transfers described in \S\ref{sectionconjecturaltransfer}. 
In \S\ref{sectionlocaltransfers}, rather than attempting to understand these in general, 
we compute them in various already-known cases while otherwise proving only inequalities. This is the main barrier towards producing exact asymptotics for general CM extension~$E/F$ and arbitrary cohomological $\pi_0$. 



Sections \ref{step3gen} and \ref{step3dom} make up the technical work of squeezing as much information about $S_\Delta$ as possible from this partial information
. The key results are Propositions~\ref{shapeineq} and \ref{goodshapefactor}, on general and odd GSK shapes respectively. 

A reformulation of the main trick of~\cite{Ger20} plays a critical role as Proposition~\ref{SMineq}. This partially resolves technical issues coming from the sign $\eps_\psi(s_\psi)$ in the stable multiplicity formula~\ref{stablemult}. It is also at this point that the payoff from the added complexity of Ta\"ibi's inductive expansion over a more straightforward approach happens: by allowing us to fully leverage the epsilon sign trick, we can bound terms coming from more many shapes than in previous works. 

Finally, Section \ref{sectionSdeltaasymptotics} completes the full inductive analysis of $S_\Delta$, and establishes the main technical results: Theorems \ref{simpleshapesmain} and \ref{goodshapebound}. The latter requires one final detail: a strengthening of \ref{simpleshapesmain} to Corollary \ref{simpleshapesstronger} using local bounds from~\cite{MS19}
. In~\S\ref{sectionconjecturalbound}, we highlight a heuristic for and possible strategy to prove conjectural optimal versions (Conjectures~\ref{simpleshapesconj} and \ref{simpleshapesconjeps}) of Corollary \ref{simpleshapesstronger}.

All that remains is the previously mentioned work in \S\ref{sectionlimitmult} of writing $m(\Delta, f^\infty)$ in terms of $S_\Delta$, and the computation of the numbers $R(\pi_0)$ and $M(\pi_0)$ in \S\ref{sectionunitary} via a parameterization of cohomological representations. This gives our main results: exact asymptotics in Theorem \ref{mainexact} and upper bounds in Theorem \ref{mainupper}. The final section~\ref{sectioncorollaries} computes explicit examples and applies the main theorems to growth of cohomology, Sato-Tate equidistribution in families, and Sarnak-Xue density.

\subsubsection{Possible extensions}
We highlight the technical obstacles that should be overcome to generalize our results further. The first is generalizing the computation of the stable transfer of indicators of congruence subgroups in Lemma \ref{splitconstant} to non-split places. This would extend Theorems \ref{mainexact} and \ref{mainupper} to $\mf n$ divisible by non-split primes. More importantly, it would allow for the extension of our techniques to quasisplit symplectic and orthogonal~$G$, where there are no ``split places'' $v$ such that $G_v \cong \GL_{n,v}$. 

Next, improving the bounds of \S\S\ref{sectionlocaltransfers} and \ref{sectionSdeltaasymptotics}
on Speh transfers of indicators of congruence subgroups would tighten the bound in Theorem \ref{mainupper}, possibly to the conjectural optimal value \ref{simpleshapesconj}. As explained in \S\ref{sectionconjecturalbound}, it would help tremendously to have a good 
understanding of the local character expansions of Speh representations through the rich interplay of ideas involving generalized Whittaker models and $A$-parameters as studied in \cite{MW87} and \cite{JLZ22}. Beyond this, 
the exact computation of these transfers would allow for the extension of the precise aysmptotics of Theorem~\ref{mainexact} beyond GSK-maxed $\pi_0$ to general ones.

Relatedly, the sign $\eps_\psi(s_\psi)$ in the stable multiplicity formula \ref{stablemult} complicates the inductive expansion of $S_\Delta$. As such, we only compute exact asymptotics for~$\Delta$ such that this sign is always positive. 
 This is the main barrier towards Theorem \ref{mainexact} applying to all GSK-maxed $\pi_0$ instead of just the odd GSK-maxed. 
Proving that when they are not identically 1, these signs cause cancellation in $S_\Delta$ would enable us to remove the ``odd'' restriction. See the discussion around Conjecture \ref{simpleshapesconjeps}.

Finally, proving the existence of Stable and Speh transfers at ramified places as in \S\ref{sectionconjecturaltransfer} would extend Theorem \ref{mainexact} to unitary groups for arbitrary CM extensions.

\subsection{How to Read}
Sections \ref{sectionAC} and \ref{sectionAJ} are background material that can mostly be skipped by experts. Sections \ref{sectionshapes} and \ref{strategy} are the conceptual heart of the paper and should be understood 
in full detail before moving on. Sections \ref{sectionlocaltransfers}, \ref{step3gen}, and \ref{step3dom} are technical details involved in implementing the strategy of \S\ref{strategy}. We recommend skipping them on a first read through and referring back depending on need/interest while reading \S\ref{sectionSdeltaasymptotics}, \ref{sectionlimitmult}. The most important results from the technical sections are Propositions \ref{shapeineq}, \ref{goodshapefactor}, and Corollary \ref{charshapeformula}. 

Finally, Section \ref{sectionunitaryparamcombo} contains many pages of 
involved but 
elementary combinatorial arguments with the parameterization of cohomological representations on unitary groups. 
We recommend reading enough to understand the definitions and statements while ignoring proofs. Some parts of Subsection \ref{sectionSX} on Sarnak-Xue density should be treated similarly. 

Due to the length of the write-up and density of cross-references in later sections, we highly recommend reading this work electronically on a PDF reader that can handle intra-document hyperlinks and that has a back button.

\subsection{Acknowledgements}
The idea for this project started in two places: when the first author was taught about \cite{Tai17} by Olivier Ta\"ibi  at the CIRM workshop ``Periods, functoriality and L-functions'' and in conversations between the two authors at the 2022 Arizona Winter School. Many helpful conversations also happened at the 2022 Midwest Representation Theory Conference, the ESI workshop ``Minimal Representations and Theta Correspondence'', the IHES summer school on the Langlands Program, and the ``Community Building in the Langlands Program'' conference in Bonn.  

Masao Oi provided us the full argument of Lemma \ref{finiteBcomponent} and Jeffrey Adams explained to us the argument of \S\ref{sectionAJcompatible}. In addition to many useful exchanges, Simon Marshall suggested the strategy to compute the Sarnak-Xue invariants in \S\ref{sectionSX}. We would also like to thank Patrick Allen, Alexander Bertoloni-Meli, Antonio Cauchi, Ga\"etan Chenevier, Andrea Dotto, Peter Dillery, Matt Emerton, Melissa Emory, Shai Evra, Jessica Fintzen, Solomon Friedberg, Wee Teck Gan, Radhika Ganapathy, Henrik Gustafsson, Alexander Hazeltine, Pol van Hoften, Ashwin Iyengar, Tasho Kaletha, Gil Moss, Samuel Mundy, Alberto Minguez, Yiannis Sakellaridis, Peter Sarnak, Gordan Savin, David Schwein, Sug Woo Shin, Joel Specter, Loren Spice, and Tian An Wong 
for pointing out many useful arguments and also for steering us away from previous proof strategies that might not have been the most feasible. 

The first author was supported by NSF postdoctoral grant 2103149 while working on this project.

\subsection{Notation} 
\subsubsection{Global variables:} As some notation used throughout:

\begin{itemize}
     \item $E/F$ a CM extension of number fields with rings of integers $\mc O_F$, $\mc O_E$
    \item $\infty$ the set of infinite places of $F$
    \item places of $F$ will be denoted by $v$, with completion $F_v$
    \item $q_S$ is the product of residue field degrees over a finite set of finite places $v$ 
    \item $\Gamma_F$ the absolute Galois group of $F$
    \item $\Gamma(E/F) = \langle \sigma \rangle$, the Galois group of $E$ over $F$
  \item $\om_{E/F}$ is the order-$2$ character associated to the quadratic extension $E/F$
  \item $\star_v$ is the local component at $v$ of the structure $\star$.
  \item $\star^S$ and $\star_S$ are components of $\star$ at/away from a set of places $S$ 
  \item ``Irreps'' are irreducible representations
  \item ``Unirreps'' are unitary irreducible representations 
  \item $v_1 \boxtimes v_2$ is the 
  representation of $H_1 \times H_2$ corresponding to reps $v_i$ 
  of $H_i$ 
  \item $[d]$ is the $d$-dimesional irrep of $\SL_2$
  \item $\1_X$ is the indicator function of the set $X$
  \item $\bar \1_X := \vol (X)^{-1}1_X$ is the indicator distribution of the set $X$ 
  \item $\1_{x=y}$ is an indicator function if $x=y$ 
  \item $\widecheck G$ is the unitary dual of abstract group $G$. 
\end{itemize}
\noindent Adelic Groups:
\begin{itemize}
    \item $G_v = G(F_v)$ for $v$ a place of $F$ and $G/F$ reductive
    \item $G_S, G^S = G(\A_S), G(\A^S)$ respectively for $S$ a set of places of $F$
    \item $\Om_G, \Om_{G,F}$ is the (geometric, $F$-rational) Weyl group of $G$
    \item $\rho_G$ is the half-sum of positive roots of $G$
    \item $K^G_v$ is a chosen hyperspecial of $G/F$ reductive at unramified place $v$. 
    \item $\ms H^\ur(G_v)$ is the unramfied Hecke algebra for $G/F$ with respect to $K^G_v$ 
    \item $K^G_v(q_v^k)$ is the $k$th Moy-Prasad filtration group of $K^G_v$ for $G/F$ reductive and unramified at place $v$.
    \item $K^G(\mf n)$ for $\mf n$ an ideal of $\mc O_F$ at which reductive $G/F$ is unramified is the product of $K^G_v(q_v^k)$ that is the congruence subgroup corresponding to $\mf n$
    \item $\Pi^G_\disc(\lb)$ is the discrete series $L$-packet with infinitesimal character $\lb$ for the real group $G$
    \item $\varphi_{\pi_d}$ is the pseudocoefficient of the discrete series representation $\pi_d$
    \item $\EP_\lb$ is the (endoscopically normalized) Euler-Poincar\'e function for infinitesimal character $\lb$
    \item $\lb[d]$ is an infinitesimal character built from $\lb$ as in formula \ref{infcharform}. 
\end{itemize}
\noindent The Endoscopic Classification:
\begin{itemize}
    \item $G(N)$ is the $\GL_N$-like group defined in \S\ref{ACgroups}
    \item $\wtd G(N)$ is the $\GL_N$-like twisted group defined in \S\ref{ACgroups}
    \item $U(N) := U_{E/F}(N)$ is the quasisplit unitary group as in \S\ref{ACgroups}
    \item $U(p,q)$ is the indefinite unitary group with signature $(p,q)$ over $\R$
    \item $\wtd {\mc E}_\el(N), \mc E_\el(G)$ are the set of elliptic endoscopic groups of $\wtd G(N), G$ respectively as in \S\ref{ACendo}
    \item $\wtd {\mc E}_\sm(N)$ are the simple endoscopic groups of $\wtd G(N)$ as in \S\ref{ACendo}
    \item $U_+(N), U_-(N)$ are the two non-equivalent realizations of $U(N)$ as an element of $\wtd {\mc E}_\sm(N)$ (see \S\ref{ACendo}). 
    \item $f^H$ for $f$ a test function on reductive $G/F$ is a choice of endoscopic transfer to some $H \in \mc E_\el(G)$
    \item $f^N$ for $f$ a test function on $G \in \wtd {\mc E}_\el(N)$ is a test function on $\wtd G(N)$ that transfers to $f$
    \item $\psi = \oplus \tau_i[d_i]$ is an Arthur parameter with cuspidal components $\tau_i$ as in \S\ref{ACdefparam}
    \item $\wtd \Psi_\el(N), \Psi_\el(G)$ are the sets of elliptic parameters associated to $\wtd G(N)$ and $G$ respectively as in \S\ref{ACdefparam} and \S\ref{ACassigntogroup}
    \item The ``Arthur-$\SL_2$'' of a parameter is an unordered partition $Q$ representing its restriction to the Arthur-$\SL_2$
    \item $\pi_\psi$ is the automorphic representation of $G(N)$ corresponding to $\psi$ in \S\ref{ACparamtorep}
    \item $\td \pi_\psi$ is the extension of $\pi_\psi$ to $\wtd G(N)$ as in \S\ref{ACrepexten}
    \item $\varphi_\psi$ is the $L$-parameter associated to $A$-parameter $\psi$ as in  \eqref{AtoLparam}
    \item $\mc S_\psi, \mc S_{\psi_v}$ are Arthur's component groups as in \S\ref{ACglobalSpsi}, \ref{AClocalSpsi}
    \item $S^\natural_\psi, S^\natural_{\psi_v}$ are Kaletha's larger component groups as in \S\ref{ACglobalSpsi}, \ref{AClocalSpsi}
    \item $s_\psi, s_{\psi_v}$ are  special elements in these component groups as in \S\ref{ACglobalSpsi}, \ref{AClocalSpsi} 
    \item $\eps_\psi$ is the identified character on global $\mc S_\psi$ as in \S\ref{ACeps}
    \item $\Pi_\psi(G), \Pi_{\psi_v}(G_v)$ are the $A$-packets associated to global or local $A$-parameters
    \item $\eta^{\psi_v}_{\pi_v}$ is the local character of $S^\natural_{\psi_v}$ associated to $\pi_v \in \Pi_{\psi_v}$
    \item $\eta^\psi_\pi$ is the character of $\mc S_\psi$ associated to $\pi \in \Pi_\psi$
    \item $\tr_{\psi_v} := \tr^{G_v}_{\psi_v}$ is the stable packet trace for the parameter $\psi_v$ on $G_v$
\end{itemize}
\noindent Shapes:
\begin{itemize}
    \item $\Delta = (T_i, d_i, \lb_i, \eta_i)$ is a (refined) shape as in \S\ref{Shapesdef}
    \item $\psi \in \Delta$ means that the $A$-parameter $\psi$ has shape $\Delta$
    \item $\Sigma_{\lb, \eta}$ is the simple shape as in \S\ref{Shapesdef}
    \item $\mc S_\Delta, s_\Delta$ are the component groups and special elements associated to $\Delta$
    \item $\psi^\Delta_\infty$ is the local component at infinity associated to $\Delta$ 
    \item $H(\Delta)$ is the $H \in \wtd {\mc E}_\el(N)$ such that $\psi \in \Delta$ implies that $\psi \in \Psi_\el(H)$ 
    \item $\Delta(\pi_0)$ is the set of shapes $\Delta$ such that $\pi_0 \in \Pi_{\psi_\infty^\Delta}$  
    \item ``GSK'' and ``odd GSK'' are conditions on shapes defined in \ref{GSK}
\end{itemize}
\noindent Trace Formulas:
\begin{itemize}
    \item $\mc{AR}_\disc(G)$ is the set of discrete automorphic representations of $G/F$
    \item
    $I^G, S^G$ are Arthur's invariant and stable trace formulas for $G$
    \item $I_\disc^G, S_\disc^G$ are their discrete parts
    \item $R^G_\disc$ is the trace against $\mc{AR}_\disc(G)$
    \item $I^G_\psi, S^G_\psi$ are the summands of $I^G_\disc, S^G_\disc$ associated to parameter $\psi$
    \item $I^G_\Delta, S^G_\Delta$ are the summands of $I^G_\disc, S^G_\disc$ associated to the shape $\Delta$ 
    \item $\star^N$ is the version of any of the variants above associated to $\wtd G(N)$ 
\end{itemize}
\noindent Asymptotics:
\begin{itemize}
    \item $|\mf n|$ is the norm of the ideal $\mf n$ of $\mc O_F$. 
    \item $\Gamma_{n_1, \dotsc, n_k}(\mf n_i)$ is an Euler factor associated to $\mf n$ and the list $n_1, \dotsc, n_k$ in \S\ref{asymptoticsetup}
    \item $\bar R(\Delta)$ is an upper bound on the growth rate 
    in Theorem \ref{simpleshapesmain}
    \item $R(\Delta)$ is a tighter upper bound on growth rate 
    in Corollary \ref{simpleshapesstronger}
    \item $R_0(\Delta)$ is the conjectural optimal growth rate 
    in Conjecture \ref{simpleshapesconj}
    \item $L_\Delta(\mf n)$ is an Euler factor 
    appearing in Theorem \S\ref{goodshapebound}. 
    \item $\tau'(G)$ is a modified Tamagawa number of reductive $G/F$ as in \cite[(9.5)]{ST16}.
\end{itemize}
\noindent Cohomological irreps of unitary groups
\begin{itemize}
    \item $\mc P(N), \mc P(p,q), \mc P_1(p,q)$ are combinatorial parameter sets defined in \ref{orderedpartitions}
    \item $\beta, \delta$ are reduction maps between these parameterizing sets defined in \eqref{eq partition maps}
    \item $\Delta^{\max}(\pi_0)$ is a subset of $\Delta(\pi_0)$ determined by \ref{Deltamax}
    \item $R(\pi_0)$ is the common value of $R(\Delta)$ for $\Delta \in \Delta^{\max}(\pi_0)$
     \item $Q^{\max}(\pi_0)$ is the set of Arthur-$\SL_2$'s of elements of $\Delta^{\max}(\pi_0)$ as in \ref{Qmax}
     \item $Q_\can(\pi_0)$ is an 
     unordered partition assigned to 
     $\pi_0$ in \ref{Qcan}
    \item ``GSK-maxed'', ``odd GSK-maxed'' are conditions on cohomological representations $\pi_0$ of unitary groups defined in \ref{GSKm}, \ref{GSKmE}
\end{itemize}


\subsubsection{Shorthand for non-factorizable functions}\label{nonfactorizable}
Certain transfer maps from factorizable functions on a group $G$ to functions on a product of groups $H_1 \times H_2$ may not always have image in factorizable functions. 

However, they will always land in linear combinations of factorizable functions. Therefore, we will use the ``mystical gate" notation
\[
\prodf_i f_i.
\]
to represent the sum of the factored term over this linear combination.

At some points, 
to elide the fact that a transfer to a group like $H^d$ may not be the same on each $H$-factor, 
we will use the even more abusive notation
\[
(f_1)^{d \oplus}
\]
to represent a sum over the factorizable pieces of the product over factors 
in the $H^d$. 

\subsubsection{Sequences}\label{notationsequences}
Given a finite sequence $n_1,\dotsc,n_k$, we define the sequence
\[
n_1^{(r_1)}, \dotsc, n_k^{(r_k)} := \overbrace{n_1, \dotsc, n_1}^{r_1 \text{ copies}}, \dotsc, \overbrace{n_k, \dotsc, n_k}^{r_k \text{ copies}}.
\]
Also, if $L_1$ and $L_2$ are sequences, ``$L_1, L_2$'' will represent their concatenation. 

Finally, if $P = (p_1, \dotsc, p_k)$ is an ordered partition of $N$ and $a = (a_i)_i$ is a list of length $N$, the $P$-parts of $a$ are defined by partitioning $a$ in order according to $P$:
\[
\overbrace{a_1, \dotsc, a_{n_1}}^{a^P_1}, \overbrace{a_{n_1 +1}, \dotsc, a_{n_1 + n_2}}^{a^P_2}, \cdots, \overbrace{a_{N - n_k  +1}, \dotsc, a_{N}}^{a^P_k}.
\]

\section{A-parameters and the Classification}\label{sectionAC}
We attempt to concisely summarize the parts of endoscopic classification that are relevant to this project.

An Arthur/endoscopic classification for a group $G$ is conceptually a ``transfer'' of two known facts about automorphic representations on $\GL_n$---
\begin{itemize}
    \item The classification of the discrete spectrum in \cite{MW89},
    \item The local Langlands correspondence for their local components
\end{itemize}
---to a parameterization of automorphic representations of $G$ and their local components. 

We will focus on the versions by Mok and Kaletha-Minguez-Shin-White from \cite{Mok15} and \cite{KMSW14} for quasisplit and general unitary groups, respectively. Our summary will be in two pieces:
\begin{itemize}
    \item A formalism of local and global parameters which encapsulates the known information on the $\GL_n$ side.
    \item A description of how the automorphic spectrum on $G$ decomposes into pieces that correspond to each parameter together with a description of the structure of each of these pieces.
\end{itemize}
We will not go over background for endoscopy or the stable trace formula since sections 2.1 and 3.1-3 of \cite{Art13} already give a good, relatively concise introduction with an eye towards the endoscopic classification.

\subsection{Groups Considered}\label{ACgroups}
We begin by defining certain groups and $L$-embeddings.

Fix a totally real number field $F$ and totally complex quadratic extension $E/F$. For each $N > 0$, consider the group 
\[
G(N) = \Res^E_F \GL_{N, E}.
\]
Let $\theta_N$ be the automorphism of $G(N)$ in the outer class of conjugate inverse transpose that fixes the standard pinning. In particular, $\theta_N$ is an involution. It can be written as $\theta_N(g) = \Ad(J_N)(\bar g^{-t})$ for a choice of $J_N$. 
\subsubsection{Unitary groups}

 Let $U_{E/F}(N)/F$ be the reductive group 
 \begin{equation} \label{definition of the unitary group in coordinates} 
 U_{E/F}(N,F) = \{ g \in GL_N(E) : \theta_N(g) = g \} 
 \end{equation}

 It is a quasisplit unitary group and therefore a form of $\GL_N/F$. We can choose a Borel and maximal torus $(B,T)$ to be the upper triangular and diagonal $\theta$-fixed matrices respectively.

 For any place $v$ of $F$, we consider $U_{E/F}(N,F_v)$. When $v$ is split in $E$, we have $U_{E/F}(N,F_v) \simeq GL_N(F_v)$. Otherwise, $U_{E/F}(N,F_v)$ is the unique quasisplit unitary group of rank $N$ over $F_v$. 
 
 Finally, because $E/F$ is a CM field, all inner forms of $U_{E/F}(N, F_\infty)$ have discrete series.


\subsubsection{L-groups and embeddings}

All our groups split over $E$ so in all our $L$-groups, the action of $\Gamma_F$ factors through $\Gamma(E/F)$. We have 
\[{^L}G(N) = (GL_N(\BC) \times GL_N(\BC)) \rtimes \Gamma_F\] with $\sigma$ swapping the two copies of $GL_N(\BC)$. We also have
\[{^L}U_{E/F}(N) = GL_N(\BC) \rtimes \Gamma_F,\] with $\sigma(g) = \Ad(J_N)(g^{-t})$. 

For $\kappa = \pm 1$, we have $L$-embeddings \begin{equation} \label{eq base change embeddings}
    \xi_\kappa: {^L}U_{E/F}(N) \to {^L}G(N).
\end{equation} The explicit coordinates are not important to us and can be found in \S2.1 of \cite{Mok15}. 

\subsubsection{Endoscopic data}\label{ACendo}
We are interested in the twisted endoscopic groups of $\widetilde G(N) = G(N) \rtimes \theta_N$. As in  \S2.4.1 in \cite{Mok15}, these are parameterized as
\[
\wtd {\mc E}_\el(N) = \{U_{\kappa_1}(N_1) \times U_{\kappa_2}(N_2) : \kappa_i = \pm 1, N = N_1 + N_2, \kappa_1 \kappa_2 = (-1)^{N - 1}\}
\]
with each $U_\pm(N_i)$ isomorphic to the quasisplit unitary group $U_{E/F}(N_i)$  (and thus to each other). The sign $\kappa_i$ determines the specific $L$-embedding $\xi_\kappa$. Among these we highlight the simple endoscopic groups:
\[
\wtd {\mc E}_\sm(N) := \{U_+(N), U_-(N)\}.
\]

We are also interested in the endoscopic groups of $G = U_{E/F}(N)$. As enumerated in \cite[\S1.1.1]{KMSW14}, these are parameterized as
\[
\mc E_\el(G) = \{U_{E/F}(N_1) \times U_{E/F}(N_2) : N = N_1 + N_2, N_1 \geq N_2\}.
\]
We do not need the full information of the endoscopic triples involved. Beware that our $\mc E_\el(G)$ is the $\overline{\mc E}_\el(G)$ of \cite{KMSW14}. 

\subsubsection{Inner forms}
We will also consider extended pure inner forms of $U_{E/F}(N)$ as in \cite[\S0.3.3]{KMSW14}. Since the general definition is not relevant to our computation, we simply recall the enumeration of possibilities for unitary groups. 

Let $G \in \wtd {\mc E}_\sm(N)$. Then in the local case:
\begin{itemize}
    \item If $v$ is non-Archimedean and split in $E$, the extended pure inner forms of $G_v \cong \GL_{N,v}$ are of the form $\Res^{D_v}_{F_v} \GL_m$ for $D_v$ a division algebra over $F$. They are associated to the invariant $a_v = N \cdot \inv(D_v)$. 
    \item If $v$ is non-Archimedean non-split\footnote{While the current as-of-this-comment draft of \cite{KMSW14} only says inert, this seems to be a typo since the arguments they give work for ramified places as well.} in $E$, the extended pure inner forms of $G_v \cong U_{E/F}(N)_v$ are:
    \begin{itemize}
        \item $U_{E/F}(N)_v$ itself, with associated invariant $a_v = 0$,
        \item another form associated to $a_v = 1$. If $N$ is odd, this form is isomorphic as a group to $U_{E/F}(N)_v$. If $N$ is even, it is the unique non-quasisplit inner form of $U_{E/F}(N)_v$. 
    \end{itemize}
    \item If $v$ is Archimedean real in $F$, then the extended pure inner forms of $G_v \cong U_{\C/\R}(N)$ are the $U(p,q)$ for $p+q = N$ and associated invariant $a_v = N(N-1)/2 + q$. Note that $U(p,q) \neq U(q,p)$ as extended pure inner forms even though they are isomorphic as groups.
\end{itemize}
A choice of local extended pure inner form at each $v$ comes from a global extended pure inner form if and only if $a_v = 1$ for almost all $v $ and  $\sum_v a_v$ is even.
Note that if we only care about inner forms as groups, the second condition is irrelevant for~$N$ odd: we can always switch an infinite-place $G_v = U(p,q)$ to $U(q,p)$, which is isomorphic as a group but has opposite $a_v$. 

We are particularly interested in the isomorphism-as-groups classes of extended pure inner forms that are unramified at all finite places. Such forms only exist when $E/F$ is unramified at all finite places (e.g. $\Q[\sqrt 3, i]/\Q[\sqrt 3]$). 
In that situation, casework with respect to the parity of $N$ gives:
\begin{lem}\label{unraminnerforms}
Assume $E/F$ is unramified at all finite places. Then every isomorphism-as-groups class of extended pure inner forms $G$ of $G^* \in \wtd {\mc E}_\sm(N)$ that is unramified at all finite places has a representative satisfying $G^\infty = (G^*)^\infty$ and is therefore determined by
\[
G_\infty = \prod_{v \in \infty} U(p_v, q_v).
\]
The possible such choices of $G_\infty$ are exactly the ones where the following is even:
\[
\f{N(N-1)}2 |\infty| + \sum_{v \in \infty} q_v.
\]
\end{lem}
Later on, when normalizing local transfer factors, we only consider these representatives to ensure consistency with the fundamental lemma at all finite places.


\subsection{Global Parameters} 
\label{ACdefparam}
\begin{dfn}
A global $A$-parameter of rank $N$ is a conjugate self-dual (through~$\theta_N$) formal expression
\[
\psi = \tau_1[d_1] \oplus \cdots \oplus \tau_k[d_k]
\]
up to reordering the summands and where each $\tau_i$ is a cuspidal automorphic representation of $G(T_i)$ (equivalently, one of $\GL_{T_i}/E$), $d_i \in \Z^+$, and $\sum_i T_i d_i = N$. 
\end{dfn}

\begin{dfn}\label{parameterwords}
We say $\psi = \tau_1[d_1] \oplus \cdots \oplus \tau_k[d_k]$ is
\begin{itemize}
\item
\emph{cuspidal} if $k=1$ and $d_1 = 1$,
\item
\emph{simple} or \emph{stable} if $k=1$,
\item
\emph{generic} if each $d_i = 1$,
\item
\emph{elliptic} if each $\tau_i$ is itself conjugate self-dual and the $\tau_i[d_i]$ are distinct.
\end{itemize}
\end{dfn}

\begin{dfn}
Let $\wtd \Psi_\el(N)$ be the set of elliptic global parameters of $G(N)$.
\end{dfn} 
Henceforth, all global parameters considered will be elliptic.

\subsubsection{Representations}\label{ACparamtorep}
As explained in \cite[\S1.3]{Art13}, the main result of \cite{MW89} associates a unique automorphic representation $ \pi_\psi$ of $\GL_n/E$ to each Arthur parameter $\psi$. First, for simple $\tau[d]$, consider the parabolic induction
\[
\Ind_{P(\A)}^{\GL_N(\A)} (\tau|\det|^{(d-1)/2} \boxtimes \tau|\det|^{(d-3)/2} \boxtimes \cdots \boxtimes \tau|\det|^{-(d-1)/2}) 
\]
where $P$ is the parabolic associated to the ordered partition $(\dim \tau, \dotsc, \dim \tau)$. We define~$\pi_{\tau[d]}$ as the unique Langlands quotient of this induction: it exists and is unitary. 

For general Arthur parameter $\psi = \bigoplus_i \tau_i[d_i]$, we let
\[
\pi_\psi := \Ind_{P(\A)}^{\GL_N(\A)} (\boxtimes_i \pi_{\tau_i[d_i]})
\]
where $P$ is the appropriate parabolic. This is always unitary and irreducible. 

\subsubsection{Canonical extensions to $\wtd G(N)$}\label{ACrepexten}
Fix a Whittaker datum $\om$ for $G(N)$ inducing local Whittaker data $\om_v$ on each $G_v(N)$. Then each $\pi_\psi$ for $\psi \in \wtd \Psi(N)$ has a canonical extension $\td \pi_\psi := \td \pi_{\psi, \om}$ to $\wtd G(N)$ as explained in \cite[\S 2.2]{Art13} or \cite[\S 3.2]{Mok15}. We warn that this ``choice of sign'' is not just a technicality to be ignored\footnote{As the authors learned at their own expense.}---it enters crucially into the computation of various sign characters in the works of Arthur and Mok and is the main difficulty in understanding Conjecture \ref{stabletransferconj1} on stable transfer.

This extension $\td \pi_\psi$ is a product of extensions $\td \pi_{\psi,v}$ of each $\pi_{\psi,v}$. By the Langlands classification, each $\pi_{\psi,v}$ is the Langlands quotient of a parabolic induction:
\[
\Ind_{P_v}^{G(N)_v} (\sigma_1|\det|^{r_i} \boxtimes \cdots \boxtimes \sigma_k |\det|^{r_k}),
\]
where each $\sigma_i$ is tempered and therefore generic. We choose the $\theta$-action on $\sigma_i$ to be the one that acts as $+1$ instead of $-1$ on its one-dimensional space of Whittaker functionals with respect to $\om_v$. 

Finally, since $\psi$ is conjugate self-dual, we necessarily have $r_j = -r_{k-j}$ and that~$\sigma_j$ and $\sigma_{r-j}$ are conjugate-duals of each other. Therefore we can choose $P$ to be fixed by $\theta$ and can define the action of $\theta$ on $\pi_{\psi,v}$ as coming from the induction of the actions on each $\sigma_i$.

\subsubsection{Assignment to groups in $\wtd {\mc E}_\el(N)$}\label{ACassigntogroup}
As in \cite[Rmk 2.4.6]{Mok15}, every $\psi \in \wtd \Psi_\el(N)$ can be assigned a unique element of $\wtd {\mc E}_\el(N)$ through which it should be thought of as ``factoring'': If $\tau \in \wtd \Psi_\el(T)$ is cuspidal, the sign $\delta$ such that $\tau$ factors through $U_{\delta}(T)$ is determined in \cite[Thm 2.5.4]{Mok15}. More generally, $\tau[d]$ is assigned to $U_{\kappa}(Td)$ where 
\[
\kappa = \delta (-1)^{(T-1)(d-1)}.
\]

Finally, if $\psi = \bigoplus_i \tau_i[d_i]$ with each $\tau_i \in \wtd \Psi_\el(T_i)$, let $N_O$ be the sum of $T_i d_i$ such that $\tau_i[d_i]$ are \emph{orthogonal}: i.e. $\delta_i(-1)^{T_i +d_i} = 1$. Similarly, let $N_S$ be defined similarly for the $\tau_i[d_i]$ that are the opposite: \emph{symplectic}. The discussion after 2.4.6 in \cite{Mok15} assigns to $\psi$ the group
\[
U_{(-1)^{N_O - 1}}(N_O) \times U_{(-1)^{N_S}}(N_S) \in \mc E_\el(N_O + N_S). 
\]

\begin{dfn}
For $G^* \in \wtd {\mc E}_\el(N)$, let $\Psi_\el(G^*)$ be the subset of $\wtd \Psi_\el(N)$ assigned to $G^*$. 
\end{dfn}

\subsubsection{As morphisms} \label{ss as morphisms}
We can interpret global parameters as morphisms into $\Ld{G(N)}$. This is a technical workaround for the absence of the conjectural global Langlands group and will be useful for discussing component groups later. 

Given a parameter
\[
\psi = \bigoplus_i \tau_i[d_i] \in \wtd \Psi_\el(N),
\]
let $\tau_i \in \Psi_\el(H_i) \subseteq \Psi_\el(T_i)$ for $H_i$ a group and $T_i$ a number. Let $\mu_i$ be the embedding
\[
\mu_i : \Ld H_i \into \Ld G(T_i)
\]
from \eqref{eq base change embeddings}. Define the fiber product
\[
\mc L_\psi := \prod_i (\Ld H_i \to W_F).
\]
Then we define the map
\[
\psi' : \mc L_\psi \times \SL_2 \into \Ld G(N) : \psi' = \bigoplus_i \mu_i \boxtimes [d_i]
\]
where $[d]$ is the $d$-dimensional irreducible representation of $\SL_2$. 

By construction, for the $G \in \wtd {\mc E}_\el(N)$ such that $\psi \in \Psi_\el(G)$, $\psi'$ factors through~$\Ld G$. Finally we refer to the restriction of $\psi'$ to $\SL_2$ as the Arthur-$\SL_2$ of $\psi$. 

\subsection{Local Parameters}\label{AClocalcomp}
\begin{dfn}
Let $G/F$ be a reductive group. A local $A$-parameter for $G$ at~$v$, denoted $\psi_v \in \Psi_v(G)$, is a $\wh{G}$-conjugacy classes of $L$-morphisms 
\[
\psi_v : L_{F_v} \times \SL_2 \to \Ld G
\]
where
\begin{itemize}
    \item the local Langlands group $L_{F_v}$ is the Weil group $W_{F_v}$ if $v$ is Archimedean and the Weil-Deligne group $WD_{F_v}$ if $v$ is non-Archimedean,
    \item $\psi_v |_{L_v}$ is a bounded $L$-parameter.
\end{itemize}

A generalized local $A$-parameter, $\psi_v \in \Psi_v^+(G)$, is the same object without the boundedness condition. As shorthand, we write $\wtd \Psi_v^\star(N) := \Psi_v^\star(\wtd G(N))$.
\end{dfn}

We may write $\psi_v \in \wtd \Psi_v(N)$ as:
\[
\psi_v = \bigoplus_i \tau_i[d_i] := \bigoplus_i \tau_i \boxtimes [d_i],
\]
with $[d]$ the $d$-dimensional irrep of $\SL_2$ and each $\tau_i$ a representation of $L_{F_v}$. We call 
the restriction of $\psi_v$ to the $\SL_2$ factor the Arthur-$\SL_2$ of $\psi_v$; as in the global case, it can be represented by an unordered partition of $N$.

As we will see below, we need to consider the set $\wtd \Psi^+_v(N)$ because the Ramanujan conjecture is at present unknown. As in Section \ref{ACassigntogroup}, we have decompositions
\[
\wtd \Psi_v(N) = \bigsqcup_{G \in \wtd {\mc E}_\el(N)} \Psi_v(G), \qquad  \wtd \Psi^+_v(N) = \bigsqcup_{G \in \wtd {\mc E}_\el(N)} \Psi^+_v(G)
\]
determined by parities $\eta_{i,v}$ assigned to irreducible $\tau_i$ (see \S2.2 in \cite{Mok15}). Thus, we extend the definitions of simple, stable, generic, and elliptic from Definition \ref{parameterwords} to local parameters. 

Finally, every local parameter $\psi_v$ has an associated $L$-parameter $\varphi_{\psi_v}:$
\begin{equation}\label{AtoLparam}
\varphi_{\psi_v} : L_{F_v} \to \Ld G : w \mapsto \psi_v \lf(w, \begin{pmatrix} |w| & 0 \\ 0 & |w|^{-1} \end{pmatrix} \ri).
\end{equation}

\subsubsection{Localization} \label{ss localization of parameters}
There is a localization map $\wtd \Psi(N) \to \wtd \Psi^+_v(N)$. Consider 
\[
\psi = \bigoplus_i \tau_i[d_i] \in \wtd \Psi(N)
\] 
with each $\tau_i \in \Psi(T_i)$ cuspidal. By the local Langlands correspondence, the component $\pi_{\tau_i,v}$ at $v$ of $\pi_{\tau_i}$ corresponds to an $L$-parameter
\[
\varphi_{i,v} : L_{F_v} \to \Ld G(N)_v.
\]
We then define the associated local $A$-parameter as:
\[
\psi_v := \bigoplus_i \varphi_{i,v} \boxtimes [d_i]. 
\]
Currently, $\psi_v$ is only known to be an element of $\Psi_v^+(N)$. However, if the Ramanujan conjecture held, each $\varphi_{i,v}$ would be bounded since they correspond to local components of cuspidal automorphic representations. This would give $\psi_v \in \Psi_v(N)$. 

Localization is consistent with the global picture: comparing with the construction of $\pi_\psi$ shows that $(\pi_\psi)_v$ has $L$-parameter $\varphi_{\psi_v}$. Moreover, \cite[Cor 2.4.11]{Mok15} and the subsequent discussion show that if $\psi \in \Psi(G^*)$ for $G^* \in \wtd {\mc E}_\el(N)$, then $\psi_v$ factors through $\Ld G^*$, so that the localization map restricts to
\[
\Psi(G^*) \to \Psi_v^+(G^*).
\]
for $G^* \in \wtd {\mc E}_\el(N)$. In the same discussion, Mok explains how to construct localization maps
\begin{equation}
    L_{F_v} \to \mc L_\psi
\end{equation}
for any $\psi \in \Psi(N)$ such that $\psi_v$ is the pullback of $\psi$ through the localization.

\subsection{Centralizer Subgroups and \lm{$\epsilon$}-Characters} \label{sectionSAndEpsilon}
We describe some invariants attached to parameters $\psi$ which appear in the description of the $\psi$-part of the automorphic spectrum of~$G^* \in \wtd {\mc E}_\el(N)$. 



\subsubsection{Global Centralizers}\label{ACglobalSpsi}
To each global parameter $\psi \in \Psi(G^*)$ with $\psi'$ as in \ref{ss as morphisms}, Mok attaches a component group $\mathcal{S}_\psi$ defined as follows: 
\begin{align*}
S_\psi(G^*) &:= Z_{\wh G^*}(\image \psi'), \\ 
\mathcal{S}_\psi &:= \pi_0(S_\psi/Z(\wh{G^*})^{\Gamma_F}).
\end{align*}
In addition, \cite{KMSW14} attaches a larger component group denoted $S^\natural_\psi$. By the discussion around (1.3.6) therein, for unitary groups we may use the formula:
\[
S^\natural_\psi = \pi_0(S_\psi).
\]
Also define a distinguished element:
\[
s_\psi := \psi'(1 \times -1) \in S^\natural_\psi. 
\]

These centralizer groups are explicitly computed in \cite{KMSW14} around (1.3.6): if
\[
\psi = \bigoplus_{i \in I} \tau_i[d_i]
\]
with $\tau_i$ cuspidal, then the $I^+$ mentioned is all of $I$ since ellipticity of $\psi$ means that the $\tau_i[d_i]$ all have multiplicity $1$. Thus, there are canonical isomorphisms
\begin{equation}\label{Spsiform}
S^\natural_\psi =(\Z/2)^I, \qquad  \mc S_\psi =  (\Z/2)^I / (\Z/2)^{\mathrm{diag}}
\end{equation}
where the distinguished element $s\psi$ is given by:
\[
s_\psi = \bigoplus_{\substack{i \in I \\ d_i \text{ even}}}1.
\] 
Note that $s\psi$ is trivial in $\mathcal{S}\psi$ if all $d_i$ have the same parity. 

\subsubsection{Local centralizers}\label{AClocalSpsi}
Mok also defines local component groups:
\begin{align*}
S_{\psi_v}(G^*) &:= Z_{\wh G^*_v}(\image \psi_v), \\ 
\mc S_{\psi_v} &:= \pi_0(S_{\psi_v}/Z(\wh{G^*_v})^\Gamma_{F_v}).
\end{align*}
We also similarly have an $S^\natural_{\psi_v}$. As explained at the ends of \S1.2.4 and \S1.2.2 in \cite{KMSW14}, we may use the formula
\[
S^\natural_{\psi_v} := \begin{cases}
    \pi_0(S_{\psi_v}) & v \text{ non-split} \\
    \det(\wh G_v) & v \text{ split}
\end{cases}
\]
in the case of unitary groups. We also define distinguished element
\[
s_{\psi_v} := \psi_v(1 \times -1) \in \mc S_{\psi_v}.
\]
Just as in the global case, $\mc S_{\psi_v}$ and $S^\natural_{\psi_v}$ can be computed explicitly, though the lack of a corresponding ``elliptic'' condition makes this slightly more complicated---see \cite[\S 1.2.4]{KMSW14} again for details. We will only consider local component groups explicitly for very specific parameters at $\infty$, so these details aren't relevant here. 

Finally, the localization maps $L_{F_v} \to \mc L_\psi$ induce corresponding maps $\mc S_\psi \to \mc S_{\psi_v}$ and~$S^\natural_\psi \to S^\natural_{\psi_v}$. Under these maps, we identify $s_{\psi_v}$ with $s_\psi$. 

%

\subsubsection{\lm{$\eps$} characters}\label{ACeps}
Fix $G^* \in \wtd {\mc E}_\el(N)$. The third structure attached to a global parameter $\psi \in \Psi(G^*)$ is a character $\eps_\psi$ on $\mc S_\psi$. 

\begin{dfn}
Let $\psi \in \Psi(G^*)$ and $(\rho, V)$ a finite-dimensional representation of~$\Ld G^*$. The resulting action $\rho_\psi : S_\psi \times \mc L_\psi \times \SL_2$ can be factored into irreducibles:
\[
\rho_\psi = \bigoplus_{j \in J} \sigma_j \otimes \gamma_j \otimes \delta_j.
\]
Following \cite[\S 1.5]{Art13}, let $J' \subset J$ be the set of indices such that $\gamma_j$ is symplectic and $\eps(1/2, \gamma_j) = -1$.
Then we define
\[
\eps_\psi^\rho : S_\psi \to \C : s \mapsto \prod_{j \in J'} \det(\sigma_j(s)).
\]
\end{dfn}

\begin{dfn} \label{defn epsilon character}
Let $\eps_\psi := \eps_\psi^{G^*}$ be $\eps^\rho_\psi$ for $\rho$ the adjoint representation of $\Ld G^*$ on $\Lie \wh G^*$. Note that it factors through $\mc S_\psi$. 
\end{dfn}



It then follows directly from the last observation of \S \ref{ACglobalSpsi} that: 
\begin{lem} \label{lemma epsilon parity}
If $\psi = \oplus_i \tau_i[d_i]$ and all the $d_i$ have the same parity, then $\epsilon_\psi(s_\psi) \equiv 1$.    
\end{lem}


\subsection{Main Theorems of the Classification}
We are now ready to state the two main theorems of the endoscopic classification:

\subsubsection{Local packets}
The first gives the existence of local $A$-packets. Let $G^* \in \mc E_\el(G)$. We recall from \cite[(0.3.1)]{KMSW14}, that one associates a complex-valued character $\chi_{G_v}$ of~$Z(\wh G_v)^\Gamma$ to every extended pure inner form $G_v$ of $G^*_v$.
This association is a bijection if $v$ is non-Archimedean, and $\chi_{G^*_v}$ is trivial. 

We define $\Irr(S^\natural_{\psi_v}, \chi)$ to be the set of trace characters of irreps of $S^\natural_{\psi_v}$ that pullback to $\chi$ through (the non-exact sequence) $Z(\wh G_v)^\Gamma \to S_{\psi_v} \to S^\natural_{\psi_v}$ (beware that this set can be empty---there is a condition of being ``relevant'' defined in \cite[\S0.4, 1.2]{KMSW14} discussing when this happens). 

\begin{thm}[{\cite[1.6.1]{KMSW14}}]\label{LocalPackets}
Let $G^* \in \wtd {\mc E}_\el(N)$ and $\psi_v \in \Psi(G^*_v)$. Fix a Whittaker datum on $G^*_v$. Then for each extended pure inner form $G_v$ of $G^*_v$, there is an associated set $\Pi_{\psi_v}(G_v)$ of unitary representations of $G_v$ together with a map
\[
\eta := \eta_{G_v} : \Pi_{\psi_v} \to \Irr(S^\natural_{\psi_v}, \chi_{G_v}), \quad  \pi_v \mapsto \eta^{\psi_v}_{\pi_v}.
\]
These satisfy:
\begin{itemize}
    \item Assume $\psi$ is generic. If $v$ is non-Archimedean, then $\eta$ is a bijection. If $v$ is Archimedean, then the maps $\eta_{G'_v}$ for $G'_v$ such that $\chi_{G'_v} = \chi_{G_v}$ jointly give a bijection from the disjoint union of the $\Pi_{\psi_v}(G'_v)$. 
    \item The $\Pi_{\psi_v}(G_v)$ for generic $\psi_v$ partition the set of tempered unirreps of $G_v$. 
\end{itemize}
\end{thm}
We also make a definition:
\begin{dfn} 
Let $\psi_v$ be a local $A$-parameter for $G^* \in \wtd {\mc E}_\el(N)$, $G_v$ an extended pure inner form of $G^*_v$, and $f_v$ a test function on $G_v$. We define the stable packet trace:
\[ 
\tr_{\psi_v}(f_v) := \tr^{G_v}_{\psi_v}(f_v) := \sum_{\pi_v \in \Pi_{\psi_v}(G_v)}\eta^{\psi_v}_{\pi_v}(s_{\psi_v}) \tr_{\pi_v}(f_v). 
\]
\end{dfn}

For $G_v = G^*_v$ quasisplit, this satisfies:

\begin{thm}[Twisted Local Character Relation, {\cite[3.2.1(a)]{Mok15}}]\label{tlocalcharidentity}
Let $\psi_v$ and quasisplit $G^*_v$ be as in the above definition. Then $\tr^{G^*_v}_{\psi_v}$ is a stable distribution on $G^*_v$ and satisfies that for any test function $\td f$ on $\wtd G(N)$:
\[
\tr^{G_v^*}_{\psi_v}(\td f^{G_v^*}) = \tr_{\wtd \pi_{\psi_v}}(\td f).
\]
\end{thm}

For $G_v$ not quasisplit, the trace-identity property determining $\tr^{G_v}_{\psi_v}$ takes some more notation to describe---see Theorem \ref{localcharidentity} with $s=1$.

\subsubsection{Global packets}
As before, let $G^* \in \mc E_\el(G)$ and $G$ be an extended pure inner form of $G^*$. 
If $\psi \in \Psi(G^*)$ we define
\[
\Pi^G_\psi := \Pi_\psi = \sideset{_{}^{}}{_{}'}\prod_v \Pi^G_{\psi_v},
\]
where the product is restricted so that $\eta^{\psi_v}_{\pi_v} = 1$ at almost all places. 

We recall from (0.3.3) in \cite{KMSW14} that
$
\prod_v \chi_{G_v} = 1,
$
so that for any $\pi \in \Pi_\psi$,
\[
\eta^\psi_\pi := \prod_v \eta^{\psi_v}_{\pi_v}
\]
is a character on $\mc S_\psi$. The characters $\eta_\pi^\psi$ determine the decomposition of the discrete automorphic spectrum $\mc{AR}_\disc(G)$:
\begin{thm}[{Arthur's Multiplicity Formula, \cite[Thm 1.7.1]{KMSW14} }]\label{ArthurMultiplicty}
We have
\[
\mc{AR}_\disc(G) = \bigoplus_{\psi \in \Psi_\el(G^*)} \bigoplus_{\pi \in \Pi^G_\psi} m^\psi_\pi \pi
\]
with multiplicities given by the trace character pairing: 
\[
m^\psi_\pi = \langle \eps_\psi, \eta^\psi_\pi \rangle_{\mc S_\psi}.
\]
\end{thm}

\subsection{Trace Formula Decompositions}
Let $I^G$ and $S^G$ be Arthur's invariant and stable trace formulas for reductive $G/F$ respectively and $I^G_\disc$ and $S^G_\disc$ their discrete parts. Also define distribution
\[
R^G_\disc := \sum_{\pi \in \mc{AR}_\disc(G)} \tr_\pi.
\]

Now let $G^* \in \wtd {\mc E}_\el(N)$ and $G$ an extended pure inner form. For each $\psi \in \Psi_\el(G^*)$, \cite[\S3.1,3.3]{KMSW14} defines
\[
I_\psi^G, S_\psi^G
\]
as summands of $I_\disc^G$ and $S_\disc^G$. 

For elliptic  $\psi \in \Psi_\el(G)$, it turns out (e.g. from the stable multiplicity formula \ref{stablemult} or an argument like \cite[(3.9)]{Art89}) that:
\[
I_\psi^G := \sum_{\pi \in \Pi_\psi^G} m^\psi_\pi \tr_\pi
\]
so that we have the decomposition:
\begin{equation}\label{Ipsidecomp}
R^G_\disc = \sum_{\psi \in \Psi_\el(G^*)} I_\psi^G.
\end{equation}
Finally, recall the stabilization of the invariant trace formula:
\begin{equation} \label{stabilization}
I^G_\disc(f) = \sum_{H \in \mc E_\el(G)} \iota(G,H) S^H_\disc(f^H)
\end{equation}
for constants $\iota(G,H)$ and endoscopic transfers $f^H$. If $f = \prod_v f_v$
is factorizable, then we can take
\[
f^H = \prod_v f^{H_v}_v
\]
where, as explained in \cite{Kal16}, the $f^{H_v}v$ are defined up to non-canonical scalars that depend on choices of local transfer factors and multiply over $v$ to one.

\subsubsection{Stabilization and decomposition}
The key point we will use is that the stabilization \eqref{stabilization} of $I^G_\disc$ 
descends to the level of $I^G_\psi$ and $S^G_\psi$. First:
\begin{prop}[{\cite[\S1.4]{KMSW14}}]\label{prop s to endo}
Let $\psi \in \Psi_\el(G^*)$. Then there is a bijection from $\mc S_\psi$ to the set of pairs $(H(s), \psi^H(s))$ with $H(s) \in \mc E_\el(G)$, $\psi^H(s) \in \Psi(H(s))$ that pushes forward to $\psi$ up to conjugation by the subset of $\Out\lf(\wh{H(s)}\ri)$ produced by conjugation in $\wh G$. In this bijection $1 \mapsto (G^*, \psi)$. 

The analogous statement also holds for $\psi_v \in \Psi(G_v)$. 
\end{prop}

Then we have a local formula:

\begin{thm}[{Local Character Relation, \cite[Thm 1.6.1(4)]{KMSW14}}]\label{localcharidentity}
In the notation from \ref{prop s to endo}, we have that for any test function $f_v$ on $G_v$ and $s \in \mc S_{\psi_v}$:
\[
\sum_{\pi \in \Pi_{\psi_v}} \eta^{\psi_v}_{\pi_v} (s' s_{\psi_v}) \tr_\pi(f_v) = \tr_{\psi_v^H(s)}(f_v^{H(s)}). 
\]
Here, $s'$ is a lift of $s$ to $S^\natural_\psi$ that together with the chosen Whittaker datum on $G_v$ determines the transfer factors for endoscopic transfer $f_v^{H(s)}$ (this is related to the difference we are ignoring between $\mc E_\el(G)$ and $\overline{\mc E}_\el(G)$ in \cite{KMSW14}). 
\end{thm}

We also have a global formula:

\begin{thm}[{Stable Multiplicity Formula, \cite[Thm 5.1.2]{Mok15}}]\label{stablemult}
Let $\psi \in \Psi_\el(G)$. Then for all test functions $f$ on $G$
\[
S_\psi^{G^*}(f) = |\mc S_\psi|^{-1} \eps_\psi(s_\psi) \tr^{G^*}_{\psi}(f).
\]
\end{thm}

\begin{proof}
We can ignore the $\sigma$ term since we are restricting to elliptic $\psi$. 
\end{proof}

As a useful addendum:
\begin{prop}[{Endoscopic Sign Lemma, \cite[Lem 5.6.1]{Mok15}}]\label{endoscopicsignlemma}
Let $\psi \in \Psi(G^*)$ and let $s \in \mc S_\psi$ correspond to $(H, \psi^H)$ as in Proposition \ref{prop s to endo}. Then
\[
\eps^H_{\psi^H}(s_{\psi^H}) = \eps^G_\psi(s s_\psi).
\]
\end{prop}

We can use all the above to compute:

\begin{thm}\label{stablemultendo}
For any test function $f$ on $G$
\[
I_\psi^G(f) = \sum_{(H, \psi_H)}  \iota(G, H) S_{\psi^H}^{H}(f^{H})
\]
where $(H, \psi_H)$ ranges over $H \in \mc E_\el(G)$ and $\psi^H \in \Psi_\el(H)$ that pushes forward to $\psi$ (up to equivalence in $H$). Here, $f^H$ is the endoscopic transfer and $\iota(G,H)$ is the constant that appears in the stabilization of the trace formula. 

Furthermore, for any $s \in \mc S_\psi$
\[
S_{\psi^H(s)}^{H(s)}(f^{H(s)}) = 2|\mc S_\psi|^{-1} \eps_\psi(ss_\psi) \sum_{\pi \in \Pi^G_\psi} \eta_\pi^\psi (ss_\psi) \tr_\pi(f). 
\]
\end{thm}

\begin{proof}
The first statement follows from the exact definition of $S_\psi$ and $I_\psi$ in \cite[\S3.3]{KMSW14}. The second can be computed from the stable multiplicity formula, the local character relation, the endoscopic sign lemma that $\eps^H_{\psi^H}(s_{\psi^H}) = \eps^G_\psi(s s_\psi)$, and noting that $|\mc S_{\psi^H}|^{-1} = 2|\mc S_\psi|^{-1}$. 
\end{proof}

\section{AJ-packets and Pseudocoefficients} \label{sectionAJ}
\label{section Adams-Johnson packets}
We recall some more background at the real place. First, we define cohomological representations on $G_\infty$ and the $A$-packets that contain them. Second, we introduce some specific test functions to eventually plug into the trace formula. 

\subsection{Cohomological Representations}

\subsubsection{Infinitesimal characters} \label{ss infinitesimal character}
 Let $G_\infty$ be a reductive group over $\R$ and $\mf t$ be a Cartan subalgebra of $G_{\infty, \C}$. Assume for simplicity that $G_\infty$ has an elliptic maximal torus, as it will when we restrict it to be a unitary group. To irreducible representations of $G_\infty$, one associates an infinitesimal character $\lb \in \Om_{G_{\infty,\C}} \bs \Hom(\mf t, \C)$. This data is the same as a map from Weyl orbits of $X^*(\wh{\mf t}) = X_*(\mf t)$ to $\C$, which is further the same as a semisimple conjugacy class in $\wh{\mf g}$.

\begin{dfn} 
We say that the infinitesimal character $\lb$ is regular integral if it is that of a finite-dimensional representation of $G_{\infty, \C}$. 
\end{dfn}

If $\varphi_\infty$ is a Langlands parameter for $G_\infty$, then 
$\varphi|_{W_\C}$ is of the form $z \mapsto (z \bar z)^\mu (z/\bar z)^\nu$ for cocharacters $\mu, \nu \in X_*(\wh{\mf t})$ on some Cartan $\wh{\mf t}$ of $\wh G$.  All $\pi \in \Pi^G_{\varphi_\infty}$ then have infinitesimal character $\mu+\nu$.

\subsubsection{Cohomological representations}
Choose a maximal compact $K$ of $G_\infty$. Given a finite-dimensional representation $V_\xi$ with highest weight $\xi$ and a unirrep $\pi$ of $G_\infty$, we can consider the~$(\mf g, K)$-cohomology groups $H^i(\mf g, K ; \pi \otimes V_\xi)$ as in \cite{BW00}. 

\begin{dfn}
A unirrep $\pi$ of $G_\infty$ is called cohomological of weight $\xi$ if there is an~$i$ such that $H^i(\mf g, K ; \pi \otimes V_\xi) \neq 0$.
\end{dfn}

Every unirrep cohomological of weight $\xi$ necessarily has infinitesimal character $\xi + \rho_G$, which is regular integral by definition. For any real group $G_\infty$ and finite-dimensional representation $V_\xi$, there are only finitely many cohomological representations; an algorithm to list them and compute their cohomology is outlined in \cite{VZ84} where they are realized as ``cohomologically induced'' representations $A_{\mf q}(\xi)$ attached to $\theta$-stable parabolic subalgebras $\fq \subset \fg$. 

Finally, \cite{Sal99} proves that all unirreps with regular, integral infinitesimal character are cohomological. We therefore suggest using ``regular, integral infinitesimal character'' as a simpler working definition of cohomological.

\subsection{AJ-packets} \label{subsection AJ 3}
Both our multiplicity computations and subsequent cohomological applications will be phrased in terms of Adams-Johnson parameters, whose definition we now recall following \cite{kottwitz1990shimura}; see also \cite{AJ87, Arthur1989Unipotent}. 

\subsubsection{AJ-parameters} 
Denote the Weil group of $\BR$ by $W_\BR = W_{\BC} \sqcup j W_{\BC}$. 
\begin{dfn} \label{definition AJ parameters}
Let $ \psi_\infty: W_\BR \times \SL_2 \to {^L}G_\infty $  be an $A$-parameter, with $\wh L$ the centralizer of $\psi_\infty(W_\BC)$ in $\wh {G}$. 
 Then $\psi_\infty$ is an Adams-Johnson, or AJ-parameter, if:
\begin{itemize}
    \item[(i)] 
    $\psi_\infty(\SL_2)$ contains a principal unipotent element of $\wh {L}$,
    \item[(ii)] the identity component of $Z(\wh {L})^{W_\BR}$ is contained in $Z(\wh {G})$, 
    \item[(iii)] the infinitesimal character of the parameter $\varphi_{\psi_\infty}$ is regular. 
\end{itemize}
\end{dfn}

In condition (ii), the action of $W_\BR$ on $Z(\wh {L})$ is defined as conjugation by $\psi_\infty(W_\BR) \subset {^L}G$. Additionally, condition (i) implies that $\psi_\infty(W_\BC) \subset Z(\wh {L})$.  
It follows from \cite[Thm 5]{NP21}, that any parameter with regular integral infinitesimal character is an AJ-parameter.

We recall a more explicit description of AJ-parameters, summarizing the discussion in \cite[\S 8]{AMR18} in the case where $G_\infty$ has a compact maximal torus $T$ (since $E/F$ is~CM, all our unitary groups will satisfy this).  Let $(\wh {T},\wh {B}, X_{\alpha})$ be a $W_\BR$-stable  pinning of $\wh {G}$ defining~$\Ld G_\infty$. Then~$\wh {L}$ can be conjugated to be the Levi of a $\wh B$-standard parabolic. 
We additionally choose a Borel pair~$(T, B)$ of $G_{\infty,\BC}$ with~$T$ defined over~$\R$ as above. Via these splittings, $\wh {L}$ is identified with the dual of a Levi subgroup $L \subset G$ containing~$T$. One then constructs an embedding $\xi: \Ld L \to \Ld G_\infty$ such that $\psi_\infty = \xi \circ \psi_L$ for~$\psi_L$ the $A$-parameter associated to a unitary one-dimensional representation of $L$ whose differential we denote $\omega$; this is done in \cite[\S 5]{Arthur1989Unipotent}.

 Finally,  if $\psi_\infty$ is an AJ-parameter, let $I_\infty$ be the set of blocks of the Levi $\Ld L$ constructed above. There is a decomposition $\psi = \bigoplus_{i \in I_\infty} \psi_i$, where all of the $\psi_i$ are necessarily conjugate self-dual because they are pairwise distinct by regularity of the infinitesimal character. In particular, the multiplicity of each $\psi_i$ is 1. As such, the computations of \cite[\S 1.2.4]{KMSW14} give a canonical isomorphism
 \begin{equation}\label{SpsiAJ}
 S^\natural_{\psi_\infty} \simeq (\Z/2)^{I^+_\infty} = (\Z/2)^{I_\infty}.
 \end{equation}
Furthermore, the localization map $S^\natural_\psi \to S^\natural_{\psi_\infty}$
is the diagonal embedding induced by the surjection $I_\infty \to I$ and the canonical isomorphism \eqref{ACglobalSpsi}. 

\subsubsection{AJ-packets}\label{AJcombinatorial}
Adams-Johnson \cite{AJ87} construct packets attached to the above parameters. As described above, an AJ-parameter $\psi_\infty$, gives rise to a pair $(L,\omega)$ of a Levi subgroup $L \subset G_\infty$ and $\omega$ the differential of a unitary character of $L$, such that $\om+\rho_L$ is the infinitesimal character of $\psi_\infty$. 

Let $\Om(G,T)$, $\Om(L,T)$, and $\Om_{\BR}(G,T)$ be the Weyl groups of $G_{\infty, \BC}$, $L_{\BC}$, and~$G_\infty$ respectively. Then the elements of the packet $\Pi_\psi = \Pi(L,\omega)$ constructed by Adams-Johnson are in bijection with 
\[
\Sigma_L = \Om(L,T)\dom \Om(G,T)/\Om_{\BR}(G,T).
\] 
For each $w \in \Sigma_L$, consider the inner form $L_w = w^{-1}Lw$; it is the centralizer of an element $x \in i \Lie(T)$, itself giving rise to a $\theta$-stable parabolic subalgebra $\fq_w \subset \fg$. Similarly, let $\omega_w = w^{-1}\omega$. Then Adams-Johnson define \[ \Pi(L,\omega) = \{ A_{\fq_w}(\omega_w); w \in \Sigma_L \}, \] where $A_{\fq_w}(\omega_w)$ is the cohomologically induced representation from \cite{VZ84}. 

We give more details on AJ-packets of cohomological representations for unitary groups in \S \ref{upqparam}, including a combinatorial parameterization of the representations in a packet, their relation with cohomology of locally symmetric spaces, and explicit formulas for the characters of $S^\natural_\psi$ associated to each parameter. 

\subsubsection{Compatibility of descriptions}\label{sectionAJcompatible}
We check that Adams-Johnson's construction of packets matches that in \cite{KMSW14}. We thank Jeffrey Adams for explaining this point to us. 

By Theorem 4.18 in \cite{Ara22}, Adams-Johnson's packets are a special case of the ABV-packets defined in \cite{ABV92}. These ABV-packets have further recently been shown by Arancibia-Mezo \cite{AM2022Equivalent} to agree with the packets built by Mok \cite{Mok15} in the quasisplit case for any given parameter. Finally, the packets of \cite{KMSW14} on non-quasisplit unitary groups are determined by trace identities comparing them to the quasisplit inner form. These trace identities are automatically satisfied by ABV-packets, so we also get that ABV-packets match the packets of \cite{KMSW14}. 

In total, we may use the combinatorial description of Adams-Johnson to understand the structure of AJ-packets on all groups we are considering.






\subsection{Pseudocoefficients and Euler-Poincar\'e Functions} \label{ss EP functions}
We also recall the definitions of certain special test functions. Recall that a \emph{standard module} of $G_\infty$ is the full (possibly reducible) parabolic induction of a discrete series or limit of discrete series representation on a standard Levi $M$. 

If $\pi_d$ is a discrete series representation of $G_\infty$, the paper \cite{CD90} constructs \emph{pseudocoefficients} $\varphi_{\pi_d}$ satisfying
\[
\tr_{\sigma}(\varphi_{\pi_d}) = \1_{\sigma = \pi_d}
\]
for all standard modules $\sigma$. We also define the Euler-Poincar\'e function
\[
\EP_\lb = \f1{|\Pi_\disc(\lb)|}\sum_{\pi_d \in \Pi_\disc(\lb)} \varphi_{\pi_d},
\]
where $\Pi_\disc(\lb)$ is the discrete series $L$-packet of infinitesimal character $\lb$. Beware that our ``endoscopic normalization'' of Euler-Poincar\'e functions is different from the usual one in the literature to work better with endoscopic transfer. 
    
\subsection{Trace Identities}
As some useful facts about traces against AJ-packets:

\subsubsection{Character formulas}
First, let $\psi_\infty : W_\R \times \SL_2 \to \Ld G_\infty$ be a parameter for an AJ-packet with infinitesimal character $\lb$. We collect some combinatorial results about the character formulas of elements in $\Pi_{\psi_\infty}$. Recall that the character formula for a representation $\pi$ of $G_\infty$ is its expansion in the Grothendieck group as a linear combination of standard modules.

\begin{lem}\label{AJ8.8}
Let $\pi_d$ be a discrete series representation of $G_\infty$ with infinitesimal character $\lb$. Then there is a unique $\pi \in \Pi_{\psi_\infty}$ such that $\pi_d$ appears in the character formula for $\pi$. 
\end{lem}

\begin{proof}
This is a consequence of Lemma 8.8 in \cite{AJ87}.
\end{proof}

Next, let $\psi_\disc = \psi_{\disc}(\lb)$ be the discrete parameter with infinitesimal character~$\lb$ (this is the AJ-parameter with trivial Arthur-$\SL_2$; its associated packet $\Pi_{\disc}(\lb)$ is a discrete series $L$-packet
). We have an inclusion $S^\natural_{\psi_\infty} \subseteq S^\natural_{\psi_\disc}$. 

\begin{lem}\label{TaiAJ}
Let $\pi \in \Pi_{\psi_\infty}$ and $\pi_d$ any discrete series appearing in the character formula for $\pi$. Then $\eta^{\psi_\infty}_\pi = \eta^{\psi_\disc}_{\pi_d} |_{S^\natural_{\psi_\infty}}$. 
\end{lem}

\begin{proof}
The values of $\eta^{\psi_\infty}_\pi$ on the lifts $s'$ appearing in the endoscopic character identity \ref{localcharidentity} are determined by the realization of $\pi$ as $A(w \lb)$ and the values $\kappa(w)$ in Theorem 2.21 of \cite{AJ87}. It is therefore determined on all of $S^\natural_\psi$ since we also know that it also restricts to the character $\chi_{G_\infty}$ of Theorem \S\ref{LocalPackets}. However, the character formula in Theorem 8.2 of \cite{AJ87} can be seen to show that $\pi_d$ corresponds to the same $w$ as $\pi$ relative to a given choice of Whittaker datum. It therefore gets assigned the same values $\kappa(w)$. 

See also the discussion on page 57 of \cite{Tai17} summarizing parts of \cite[\S5]{Arthur1989Unipotent} for a more explicit computation of these characters in the quasisplit case. By a parenthetical note there, Ta\"ibi's argument should generalize to non-quasisplit groups through the methods of \cite[\S5.6]{Kal16rig}. 
\end{proof}


\subsubsection{Identities}
Now we can prove our trace identities:
\begin{lem}\label{charps}
Let $\pi_d$ be a discrete series of $G_\infty$ with infinitesimal character $\lb$. 
\begin{enumerate}
\item
Let $\pi_0$ be a a representation of $G_\infty$ such that $\pi_d$ appears in its character formula with coefficient $\varepsilon$. Then for all parameters $\psi_\infty$ such that $\pi_0 \in \Pi_{\psi_\infty}$: for all $\pi \in \Pi_{\psi_\infty}$,
\[
\tr_{\pi}(\varphi_{\pi_d}) = \begin{cases} \varepsilon & \pi = \pi_0 \\ 0 & \text{else} \end{cases}.
\]
\item
Let $\psi_\infty$ be a parameter 
 with infinitesimal character $\lb$. Then for all $s \in S^\natural_{\psi_\infty}$,
\[
\sum_{\pi_\infty \in \Pi_{\psi_\infty}} \eta^{\psi_\infty}_{\pi_\infty}(s) \tr_{\pi_\infty}(\varphi_{\pi_d}) = \eta^{\psi_\disc(\lb)}_{\pi_d}(s_{\psi_\infty}s )
\]
(where we implicitly use that $S^\natural_{\psi_\infty} \subseteq S^\natural_{\psi_\disc(\lb)}$). 
\end{enumerate}
\end{lem}

\begin{proof}
For (1), $\psi_\infty$ then has infinitesimal character $\lb$ and is therefore necessarily an AJ-parameter. Therefore $\pi_0$ is then the unique $\pi \in \Pi_{\psi_\infty}$ such that $\pi_d$ appears in the character formula for $\pi$ by \ref{AJ8.8}. 

For (2), choose $\pi_0$ to be the unique $\pi \in \Pi_{\psi_\infty}$ with $\pi_d$ in its character formula. Let it appear with coefficient $\varepsilon$. By (1), the sum of traces is $\varepsilon \eta^{\psi_\infty}_{\pi_0}(s)$. 

To compute $\varepsilon$, we use that by \cite[(8.10)]{AJ87} all discrete series appear in the character formula of the stable sum
\begin{equation}
\sum_{\pi \in \Pi_{\psi_\infty}} \eta^{\psi_\infty}_\pi(s_{\psi_\infty}) \pi
\end{equation} 
with multiplicity $1$. Therefore, $\varepsilon = \eta^{\psi}_{\pi_0}(s_{\psi_\infty})$. Lemma \ref{TaiAJ} finishes the proof.
%
%
\end{proof}

As an important special case using that $S^\natural_{\psi_\infty}$ is a 2-group:

\begin{cor}\label{straceAJ}
Let $\pi_d$ be a discrete series of $G_\infty$ with infinitesimal character $\lb$ and let $\psi_\infty$ be a parameter with infinitesimal character $\lb$. Then
\[
\tr_{\psi_\infty}(\EP_\lb) = \tr_{\psi_\infty}(\varphi_{\pi_d}) = 1.
\]
\end{cor} 

\section{Shapes}\label{sectionshapes}
We now come to the key definition of this paper: the invariant of a \emph{(refined) shape}~$\Delta$ attached to an an Arthur parameter $\psi$. We construct the 
invariant to satisfy:
\begin{itemize}
    \item $\Delta$ determines the $H \in \wtd {\mc E}_\el(N)$ attached to $\psi$. 
    \item $\Delta$ determines the local factor $\psi_\infty$ (and in particular the Arthur-$\SL_2$), as well as the pair $(\mS_\psi,s_\psi)$.
    \item For shapes $\Delta$ associated to \emph{cuspidal} parameters, there is a well-understood  
    geometric-side expression for the trace against the $\Delta$-part of the automorphic spectrum of $H$.

    \item The stable multiplicity formula for the $\psi$-part $S^H_\psi$ of $S^H_\disc$ can be well-enough understood inductively in terms of the $S^{H_i}_{\psi_i}$ for cuspidal $\psi_i$, 
    through algorithms uniform over all $\psi$ with shape $\Delta$. 
\end{itemize}
 These properties are in turn exactly those needed for the inductive argument of \S\ref{strategy}.

\subsection{Infinitesimal and Central Characters}
We need some preliminary details on infinitesimal and central characters:

\subsubsection{Infinitesimal characters in the classification} 
Let $G^* \in \wtd {\mc E}_\sm(N)$. Then $G^*_\infty = U^*_N(F_\infty)$ for $U^*$ the quasisplit inner form. The Lie algebra~$\wh{\mf g}_\infty = \gl_N(F_\infty \otimes_\R \C)$, so consider an infinitesimal character of $G^*$ as a semisimple matrix up to conjugacy, or in other words, 
an unordered sequence
\[
\lb = (\lb_1, \lb_2, \dotsc,  \lb_n)
\]
with $\lb_j \in F_\infty \otimes_\R \C$. We can further expand out each $\lb_j$ as a list of complex numbers~$\lb_{j,v}$ for each place $v \in \infty$ of $F$ (which is necessarily real). 


It is also sometimes useful to think of each local component $\lb_v$ as the generating function $\sum_j X^{\lb_{j,v}}$. In this way, if each $\tau_{i,v}$ has infinitesimal character $\lb^{(i)}_v$, we have infinitesimal character assignment
\begin{equation}\label{infcharform}
\lf( \bigoplus_i \tau_i[d_i] \ri)_v \mapsto \sum_i \lb^{(i)}_v[d_i] := \sum_i \lb^{(i)}_v \sum_{l=1}^{d_i} X^{\f{d+1}2 - l}. 
\end{equation}
It can be seen from this that the character of $\tau[d]$ determines that of $\tau$. 

Finally, recall from \S\ref{ss infinitesimal character} that we call $\lb$ regular integral if it is the infinitesimal character of a finite-dimensional representation. Equivalently, for each $v \in \infty$ the $\lb_{i,v}$ are distinct and are all integers if $N$ is odd or all half-integers if $N$ is even. 




\subsection{Definitions}\label{Shapesdef}
Our notion of shape is built off of the details of how Mok assigns an element of $\wtd {\mc E}_\el(N)$ to a parameter. 

In \ref{ACassigntogroup} we saw that for each cuspidal parameter $\tau \in \wtd \Psi_\el(N)$ with regular integral infinitesimal character $\lb$ at infinity and of parity $\delta = \pm 1$, there exists a unique $H = H(\delta) \in \wtd {\mc E}_\sm(N)$ depending only on $\delta$ such that $\tau$ is a parameter for $H(\delta)$. 
Motivated by this, we define:
\begin{dfn}
A \emph{(refined) shape} is a sequence
\[
\Delta = (T_i, d_i, \lb_i, \eta_i)_i
\]
up to permutation and where the $T_i$ and $d_i$ are positive integers, $\lb_i$ is an infinitesimal character of rank $T_i$, and $\eta_i = \pm 1$.  

We say that $\psi \in \Delta$, or that $\psi$ has (global) shape $\Delta$, if $\psi$ is elliptic and~$\psi = \bigoplus_i \tau_i[d_i]$ with each $\tau_i$ of rank $T_i$ and such that each $\tau_i$ has infinitesimal character~$\lb_i$ at infinity and is of parity $\delta_i = \eta_i$. 

We let $\Psi_\el(\Delta) \subseteq \wtd \Psi_\el(N)$ be the set of all elliptic, self-dual parameters on $G(N)$ of shape $\Delta$. 
\end{dfn}

\begin{note}
For this work, all shapes will be refined. Beware that this does not match the usage of shape elsewhere in the literature, e.g. \cite{MS19}.
\end{note}

\begin{note}
Unitary groups $G$ may be realized as twisted endoscopic groups of $\wtd G(N)$ in two different ways. The $\eta_i$ data is a technical necessity to select which of these realizations the shape corresponds to---it is not very important conceptually. 
\end{note}


\begin{dfn}
Let $\lb$ be an infinitesimal character of rank $N$ and $\eta = \pm 1$. Then~$\Sigma_{\lb, \eta}$ is the shape $(N, 1, \lb, \eta)$.
\end{dfn}

\begin{dfn}
The shape $\Delta$ is \emph{integral} if its \emph{total} infinitesimal character as determined by formula \eqref{infcharform} is regular integral.
\end{dfn}

In particular, if $\Delta$ is integral, then each of the $\lb_i$ must be regular integral, though this isn't sufficient. 

\subsection{Properties}
We list the key properties of shapes, and stress one last time that although we drooped the adjective, we refer to refined shapes:
\begin{prop}\label{shapetogroup}
Let the shape $\Delta$ have rank $N$ such that all the $\lb_i$ are integral. Then there is a group $H(\Delta) \in \wtd {\mc E}_\el(N)$ such that all $\psi \in \Delta$ are parameters for $H(\Delta)$. 
\end{prop}

\begin{proof}
The assignment described in \S\ref{ACassigntogroup} only depends on $T_i$, $d_i$, and $\eta_i$. 
\end{proof}
This gives us two quick corollaries:
\begin{cor}\label{grouptoshapes}
Let $G \in \wtd {\mc E}_\el(N)$. Then
\[
\Psi_\el(G) = \bigsqcup_{\Delta : H(\Delta) = G} \Psi_\el(\Delta).
\]
\end{cor}

\begin{proof}
This follows from the above since every parameter has a shape.
\end{proof}

\begin{cor}\label{sigmatogroup}
Let $\lb$ be an integral infinitesimal character of rank $N$. Then
\[
H(\Sigma_{\lb, \eta}) = U_{\eta(-1)^{N-1}}(N).
\]
In particular, for every $G \in \wtd {\mc E}_\sm(N)$ and infinitesimal character $\lb$ of a finite-dimensional representation on $G$, there is $\eta$ such that $H(\Sigma_{\lb, \eta}) = G$. 
\end{cor}

As three more facts, if $\Delta$ is a shape:
\begin{itemize}
\item
All $\psi \in \Delta$ correspond to the same pair $(S^\natural_\psi,s_\psi)$ by formula \eqref{Spsiform}. Call the common values $S^\natural_{\Delta}$ and $s_{\Delta}$. We can similarly define a common $\mc S_\Delta$. 
\item
All $\psi \in \Delta$ have the same Arthur-$\SL_2$, so the Arthur-$\SL_2$ of~$\Delta$ is well-defined. 
\item
The infinitesimal character at infinity 
of $\psi \in \Delta$ 
is determined by \eqref{infcharform}. 
\end{itemize}
Finally:
\begin{lem}\label{shapetoinfty}
Let $\Delta$ be an integral shape. There exists an AJ-parameter
~$\psi_\infty^\Delta$ such that for all $\psi \in \Delta$, $\psi_\infty=\psi_\infty^\Delta$.
 Furthermore, the induced localization map $S^\natural_\Delta \to  S^\natural_{\psi_\infty^\Delta}$ is also determined by $\Delta$. 
\end{lem}

\begin{proof}

The shape $\Delta$ determines the localization $\psi_\infty$ of any parameter $\psi \in \Delta$: for $\psi_\infty\mid_{\SL_2} = \oplus_i \nu(d_i)^{T_i}$, this is immediate. For the restriction $\psi_\infty\mid_{W_\BR}$, let $\psi = \tau_i[d_i]$ such that each term has infinitesimal character~$\lb_i$. Then by a known case of the Ramanujan conjecture as explained in \cite[Lem 6.1]{MS19}, the parameter associated to~$\tau_{i, \infty}$ is bounded with infinitesimal character~$\lb_i$ matching that of a finite-dimensional representation, so it is uniquely determined. The parameter $\psi_\infty$ is determined from this data following the constructions of \S\ref{ss localization of parameters}. 

By construction, the resulting $\psi_\infty$ has a regular integral infinitesimal character. Additionally, it maps the Arthur-$\SL_2$ to a principal $\SL_2$ of the Levi $\wh L = \prod_i GL_{d_i}^{T_i} \subset \wh{G}$, with $\wh L = Z_{\wh G}(\psi_\infty(W_\BC))$  following the assumption that the total infinitesimal character is regular. Then $\psi_{\infty}(W_\BC)\subset Z(\wh L)$, and since $\psi_\infty$ was built from parameters of discrete series, $\psi_\infty\mid_{W_\BC}$ factors through $z \mapsto \frac{z}{\bar{z}}$ and $W_\BR$ acts on $Z(\wh L)$ by inversion. Thus the identity component of $Z(\wh L)^{W_{\BR}}$ is trivial, and $\psi_\infty$ is an AJ-parameter following Definition \ref{definition AJ parameters}.

For the last statement, the map $I^+_\infty \to I^+$ described after formula \eqref{SpsiAJ} can be seen to depend only on $\Delta$. This determines the localization map.  
%
%
%
%
%
\end{proof}

\section{The Trace Formula with Fixed Shape}\label{strategy}
Let $G \in \wtd {\mc E}_\el(N)$ (the case of $G \in \wtd {\mc E}_\sm(N)$ suffices for us). If $\Delta$ is a shape such that $H(\Delta) = G$, define
\[
S^G_\Delta := \sum_{\psi \in \Delta} S^G_\psi.
\]
This $S^G_\Delta$ is the main technical building block in our applications, and understanding it is the key step in our argument.

For an 
infinitesimal character $\lb$, let $\EP_\lb$ be the 
Euler-Poincar\'e function introduced in \S\ref{ss EP functions}. 
The overarching goal for this section is:
\begin{goal}
Let $\lb$ be the infinitesimal character of a finite-dimensional representation of~$G_\infty$. Understand $S^G_\Delta(\EP_\lb f^\infty)$ as a linear combination of terms $I^H_\disc(\EP_{\lb'} (f^\infty)')$ on other groups $H$. 
\end{goal}
The terms $I^H_\disc(\EP_{\lb'} (f^\infty)')$ are well-understood---they were given an explicit formula in \cite{Art89}, which was in turn studied in great detail and bounded in \cite{ST16}. Achieving the goal would therefore give fine control over $S_\Delta$.

Ta\"ibi in \cite{Tai17} gave such a description in the level-$1$ case. This work is basically extending his method as much as possible to deeper levels.

\subsection{Overall Strategy}
We build up $S^G_\Delta(\EP_\lb f^\infty)$ from $I^H_\disc(\EP_{\lb'} (f^\infty)')$ terms in stages:
\begin{enumerate}
\item
Switch from $I$ to $S$ to allow for various transfers in the endoscopic classification using that $S^H_\disc(\EP_\lb f^\infty)$  can be expanded as a linear combination of $I^{H'}_\disc(\EP_{\lb'} (f^\infty)')$ terms through ``hyperendoscopy'' as in~\cite{Fer07}. 
\item
Define a sum of traces on each $\wtd G(M)$ called $S^M(\lb, \eta, f^\infty)$ satisfying that if $H = H(\Sigma_{\lb, \eta})$, $S^M(\lb, \eta, f^\infty) = S^H_\disc(\EP_\lb (f^\infty)^H)$. The use of traces on~$\wtd G(M)$ allows inducting by decomposing $\Delta$ down to its simple constituents.
\item
Define sub-sums $S^N_\Delta(f^\infty)$ of $S^N(\lb, \eta, f^\infty)$ corresponding to individual shapes. All the $S^N_\Delta(f^\infty)$ can be computed from terms $S^M(\lb, \eta, (f^\infty)')$ by an induction in two steps:
\begin{itemize}
\item
\[
S^N_{\Sigma_{\lb, \eta}}(f^\infty) = S^N(\lb, \eta, f^\infty) - \sum_{\substack{H(\Delta) = H(\Sigma_{\lb, \eta}) \\ \text{Inf. Char}(\Delta) = \lb \\ \Delta \neq \Sigma_{\lb, \eta}}}  S^N_\Delta(f^\infty). 
\]
When $N=1$, the sum of the ``correction terms'' vanishes---this is the base case of our induction.
\item
If $\Delta = (T_i, d_i, \lb_i, \eta_i)_i$ with infinitesimal character $\lb$, 
express $S_\Delta^N(f^\infty)$ 
in terms of the $S_{\Sigma_{\lb_i, \eta_i}}^{T_i}$. This is the hardest part of the argument and the partial results needed for our applications will be postponed to Sections \ref{step3gen} and \ref{step3dom}. 
\end{itemize}
\item
Transfer back to the classical group---$S^G_\Delta(\EP_\lb f^\infty)$ can be written as $S^{N(G)}_\Delta( f^\infty_1)$ as long as we can find $(f^\infty_1)^G = f^\infty$.
\end{enumerate}

In our actual argument, steps 2 and 3 will not be so clearly separated. We will use step 2's Proposition \ref{SMform} to switch freely between the perspective of traces on $G \in \wtd {\mc E}_\el(N)$ and traces on $\wtd G(N)$ as convenient to perform the induction of step 3.


Since there is a lot of notation introduced in this section, we summarize all the parts of the trace formula defined:
\begin{itemize}
    \item $I^G_\Delta, I^G_\psi$ are the pieces of the spectral decomposition of Arthur's $I^G_\disc$ on some $G \in \wtd {\mc E}_\el(N)$ corresponding to a shape $\Delta$ or individual parameter $\psi$.
    \item $S^G_\Delta, S^G_\psi$ are the same for Arthur's $S^G_\disc$.
\end{itemize}

\subsection{Step 1: Understanding \lm{$S^H$}}
The $S^H$ terms can be understood through the hyperendoscopy formula of \cite{Fer07}. We use notation from \cite{Dal22}, although the full generality there isn't necessary since the computation of the endoscopic data of classical groups in \cite{Wal10} shows that we will never have to take a $z$-extension.
\begin{thm}[Ferrari's Hyperendoscopy Formula]\label{hyperendoscopy}
Let $\mc{HE}_\el(H)$ be the set of non-trivial elliptic hyperendoscopic paths of $H$ as in \cite[\S4]{Dal22}. Then,
\[
S^H(\EP_\lb f^\infty) = I^H(\EP_\lb f^\infty) + \sum_{\mc H \in \mc{HE}_\el(H)} \iota(G, \mc H) I^{\mc H}(\EP_\lb^{\mc H} (f^\infty)^{\mc H})
\]
for constants $\iota(G, \mc H)$ and transfers $\star^{\mc H}$ defined there.
\end{thm}

\begin{proof}
See \cite[Thm 4.2.3]{Dal22}, and note that using $\EP_\lambda$ at infinity lets us elide the distinction between trace formulas and their discrete parts.
\end{proof}


 Ferrari gives an explicit formula to write the transfers $\EP_\lb^H$ as linear combinations of Euler-Poincar\'e functions. See \cite[\S5.1]{Dal22} for an English-language presentation though beware that there is a $\rho$-shift between the parameterization $\EP_\lb$ used here and the parameterization $\eta_\lb$ used there.

\subsection{Step 2: Understanding \lm{$S^M$}}
\label{ss SM}
Fix an infinitesimal character $\lb$ of rank $M$. Using Proposition \ref{shapetogroup}, define: 
\[
\Psi_\el(\lb, \eta) =  \bigcup_{\substack{\Delta : H(\Delta) = H(\Sigma_{\lb, \eta}) \\ \text{Inf. Char}(\Delta) = \lb}} \Psi_\el(\Delta)
\]
and define pieces of the spectral expansion:
\begin{align} \label{eq SM definition}
S^M(\lb,\eta, f^\infty)  
& := \sum_{\Delta : H(\Delta) = H(\Sigma_{\lb, \eta})} S^M_\Delta(f^\infty) \\
&:=  \sum_{\psi \in \Psi_\el(\lb, \eta)}  S^M_\psi(f^\infty) \nonumber \\
&:= \sum_{\psi \in \Psi_\el(\lb, \eta)}  \eps^H_\psi(s^H_\psi) m_\psi |\mc S_\psi|^{-1} \tr_{\wtd \pi_\psi^\infty}(f^\infty). \nonumber
\end{align} 
Here, $\pi_\psi$ is the automorphic representation of $\GL_M(\A_E)$ corresponding to $\psi$ as in \S\ref{ACparamtorep} and $\wtd \pi_\psi$ is its extension to $\wtd G(M)$ as in \S\ref{ACrepexten}. 
\begin{prop}\label{SMform}
Let $H = H(\Sigma_{\lb, \eta})$ which is necessarily in $\mc E_\sm(M)$. Then for a test function $f^\infty$ at the finite places,
\[
 S^M(\lb, \eta, f^\infty) = S^H(\EP_\lb (f^\infty)^H). 
\]
\end{prop}

\begin{proof}
Fix $\psi \in \Psi_\el(H)$ with infinitesimal character $\lb$ which is necessarily elliptic. Then the stable multiplicity formula \ref{stablemult} shows that
\begin{multline*}
S^H_\psi(\EP_\lb (f^\infty)^H) = \eps_\psi(s_\psi) |\mc S_\psi|^{-1} \lf( \sum_{\pi_\infty \in \Pi_{\psi_\infty}} \eta^{\psi_\infty}_{\pi_\infty}(s_\psi) \tr_{\pi_\infty}(\EP_\lb) \ri)\\
\lf( \sum_{\pi \in \Pi_{\psi^\infty}} \eta^{\psi^\infty}_{\pi^\infty}(s_\psi) \tr_{\pi^\infty}((f^\infty)^H) \ri).
\end{multline*}
 Corollary \ref{straceAJ} then gives for infinitesimal character $\lb$:
\[
 \sum_{\pi_\infty \in \Pi_{\psi_\infty}} \eta^{\psi_\infty}_{\pi_\infty}(s_\psi) \tr_{\pi_\infty}(\EP_\lb) = 1.
\]
Furthermore, multiplying together the twisted local character relation \ref{tlocalcharidentity} over all finite places shows that:
\[
\sum_{\pi \in \Pi_{\psi^\infty}} \eta^{\psi^\infty}_{\pi^\infty}(s_\psi) \tr_{\pi^\infty}((f^\infty)^H) = \tr_{\wtd\pi_\psi^\infty} (f^\infty). 
\] 
In total $S^H_\psi(\EP_\lb (f^\infty)^H) = S_\psi^M(\lb, \eta, f^\infty)$.

Since $S_\psi(\EP_\lb (f^\infty)^H) = 0$ for all $\psi$ with infinitesimal character not equal to $\lb$, summing over $\Psi_\el(\lb, \eta)$ and using Corollary \ref{grouptoshapes} gives that
\[
S^M(\lb, \eta, f^\infty) = \sum_{\psi \in \Psi_\el(H)} S^H_\psi(\EP_\lb (f^\infty)^H).
\]

By \cite[(3.9)]{Art89} and equation \eqref{Ipsidecomp}, we know that 
\[
I^H(\EP_\lb (f^\infty)^H) = R^H(\EP_\lb (f^\infty)^H) = \sum_{\psi \in \Psi_\el(H)} I^H_\psi(\EP_\lb (f^\infty)^H). 
\]
By a hyperendoscopy argument using the expansion in Theorem \ref{stablemultendo}, the same sum expansion holds for $S^H$. Therefore, we can conclude that
\[
S^M(\lb, \eta, f^\infty) = S^H(\EP_\lb (f^\infty)^H).
\]
This finishes the argument.
\end{proof}

\subsection{Step 3: The Induction} \label{ss step 3}
We give a heuristic overview to keep in mind for understanding step 3. All precise results will be postponed to \S\S\ref{step3gen},\ref{step3dom}. 

Recall the decomposition
\[
S^N(\lb,\eta,f^\infty) = \sum_{\Delta: H(\Delta) = H(\Sigma_{\lb,\eta})} S^N_\Delta(f^\infty)
\]
which we were using for the induction in step 3. We now want to understand $S_\Delta$ in terms of smaller groups. 

For the sake of heuristic understanding, we will instead consider the simpler
\[
S^{|H(\Delta)|}_\Delta(\EP_\lb (f^\infty)^{H(\Delta)}) = S^{|N|}_\Delta(f^\infty) := \sum_{\psi \in \Delta} m_\psi |S_\psi|^{-1} \tr_{\wtd \pi_\psi^\infty}(f^\infty)
\]
without the $\eps$-sign. Consider $\psi = \tau_1[d_1] \oplus \cdots \oplus \tau_k[d_k] \in \Delta$. Motivated by the way~$\wtd \pi_\psi$ is defined through parabolically inducing determinant twists of the $\tau_i$, assume we could define a ``generalized constant term'' map
\[
f^\infty \mapsto (f^\infty_i)_{\Delta,i}
\]
such that for all $\psi \in \Delta$
\[
\tr_{\wtd \pi_\psi^\infty}(f^\infty) = \prod_i \tr_{\wtd \pi_{\tau_i}^\infty}(f^\infty_{\Delta, i}).
\]
Because the infinitesimal character of $\Delta$ is regular and therefore disallows repeated $\tau_i$ factors, the possible (elliptic) $\psi \in \Delta$ are exactly the $\tau_1[d_1] \oplus \cdots \oplus \tau_k[d_k]$ for all choices of~$\tau_i \in \Sigma^{T_i}_{\lb_i, \eta_i}$. Therefore we get a heuristic factorization
\begin{equation}\label{heuristicfactorization}
S^{|N|}_\Delta(f^\infty) = C_\Delta \prod_i S^{|T_i|}_{\Sigma_{\lb_i, \eta_i}} (f^\infty_{\Delta, i}),
\end{equation}
for $C_\Delta$ a constant depending on the various $S_\psi$, $S_{\tau_i}$, and $m_\psi$'s that only depend on~$\Delta$. 

Obviously, we cannot simply ignore the $\eps$-sign and we do not have a actual definition of this generalized constant term. In fact, the definition of this generalized constant term would allow us to define the long-desired ``stable transfer'' between $G$ and its endoscopic groups, so it is likely very difficult.

However, \S\S\ref{step3gen},\ref{step3dom} will discuss enough partial results that an application to limit multiplicities at specifically split level can be completed.

\subsection{Step 4: Understanding \lm{$S^G_\Delta$}}
This step comes from a corollary to the arguments in step 2:
\begin{cor}
Fix a shape $\Delta$ of rank $N$ and let $G = H(\Delta)$ as in Proposition \ref{shapetogroup}. Then for any test function $f^\infty$ on $G(\A^\infty)$, there is $f^\infty_1$ on $\widetilde G(N)^\infty$ such that $(f^\infty_1)^G = f^\infty$. Furthermore,
\[
S_\Delta^G(\EP_\lb f^\infty) = S^N_\Delta(f^\infty_1).
\]
\end{cor}

\begin{proof}
The existence of $f^\infty_1$ comes from \cite[Prop 3.1.1(b)]{Mok15}. Then, arguing as in \ref{SMform} gives $S^G_\psi(\EP_\lb (f^\infty_1)^G) = S_\psi^N(f_1^\infty)$ for any $\psi \in \Delta$. Summing over all $\psi \in \Delta$ produces the result.
\end{proof} 

\section{Local Transfer}\label{sectionlocaltransfers}
The argument of \S \ref{ss step 3} requires constructing certain transfers of functions. We are unable to do this in full generality---this section focuses on either constructing various special cases or approximate versions satisfying good enough bounds. The last 
Subsection \ref{sectionconjecturaltransfer} will state the full desired conjectures. 

We will heavily use the shorthand from Section \ref{nonfactorizable} throughout. If $G_v$ is an unramified group over~$F_v$ then~$K_v : = K^G_v$ will be a hyperspecial subgroup for $G_v$.

\subsection{Split Places}
We first discuss some results that only hold at split places. Extending Lemmas \ref{splitconstant} and \ref{split[d]} to non-split places would similarly extend Theorems \ref{mainexact} and \ref{mainupper}.

Fix $G \in \wtd {\mc E}_\sm(N)$. Let $v$ be a finite place of $F$ which splits in $E$, so that $E \ten_F F_v \simeq F_v \times F_v$ with $\sigma$ permuting the two factors. Then $G(N)_v \simeq GL_N(F_v) \times GL_N(F_v)$ and the action of $\theta_N$ on $G(N)_v$ and $\Ld G(N)_v$ becomes 
\[
(g_1,g_2) \mapsto (\Ad(J_N)g_2^{-t}, \Ad(J_N)g_1^{-t}),
\]
where $\Ld G_v \simeq GL_N$ embedded in $\Ld G(N)_v$ as the fixed points of $\theta$ (note that both the $\pm$ embeddings of \cite[(2.1.9)]{Mok15} are the same in this case). 

We next outline how 
Arthur's local classification for $G_v$ as a element of $\wtd {\mc E}_\el(N)$ agrees with that
coming from the isomorphism $G_v \simeq \GL_N(F_v)$:

\begin{lem}\label{splitconsistency}
Let $v$ be a place of $F$ which splits in $E$ 
and denote $E \ten_F F_v$ by $E_v$. Let~$\pi_v$ be the irreducible conjugate self-dual representation of $G(N)(E_v)\simeq GL_N(F_v) \times GL_N(F_v)$ coming from the Arthur parameter $\psi_v$. Then:
 \begin{enumerate}
     \item $\pi_v$ is of the form $\pi_v^0 \ten (\pi_v^0)^\vee$ for some irrep $\pi_v^0$ of $\GL_N(F_v)$.
     \item The canonical extension of $\pi_v$ to $\wtd G(N)$ as in \cite[\S2.2]{Art13} has $\theta$ acting on $\pi_v^0 \ten (\pi_v^0)^\vee$ through $x \otimes y \mapsto y \otimes x$.
     \item For $\tilde f_v = (f^1_v,f^2_v) \in \mH(G(N)(E_v) \rtimes \theta)$, we can choose transfer $\td f_v^G = f^1_v \star {}^\theta \! f^2_v$ where we define ${}^\theta \! f_v(g) = f_v(\Ad(J_n) g^{-t})$.
     \item For $\tilde f_v = (f^1_v,f^2_v) \in \mH(G(N)(E_v) \rtimes \theta)$, $\tr_{\td{\pi}_v}(\tilde f_v) = \tr_{\pi^0_v}(f^1_v \star {}^\theta \! f^2_v)$ 
     \item $\tr_{\psi_v}(f_v) = \tr_{\pi^0_v}(f_v)$.
 \end{enumerate}

\end{lem}

\begin{proof}
Statement (1) follows from the description of irreps of 
$G_1 \times G_2$ and self-duality. 

For (2), first assume $\pi_v$ and therefore $\pi_v^0$ is tempered. A Whittaker functional on~$\pi_v$ is a product of a pair of functionals on $\pi_v^0$ and $(\pi_v^0)^\vee$. This product is preserved by the claimed $\theta$ since the space of Whittaker functionals is one dimensional. On the other hand, if $\pi$ isn't tempered,  the statement can be checked by the construction of $\theta$ through parabolic induction. This gives the second statement in all cases. 

Statement (3) is a special case of Corollary 1.1.6 in \cite{KMSW14}. 

For (4), $\td f_v (x \otimes y) = (f_v^1 x \otimes {}^\theta \! f_v^2 y)$. By admissibility of $\pi_v$ and smoothness of the~$f_v^i$, we can compute the trace by choosing a basis for the finite-dimensional vector space $V = (\pi_v)^U$ for some open compact $U$. The result follows from computing the action on the standard induced bases for $\Sym^2 V \oplus \wedge^2 V = V \otimes V$.

Finally, (5) follows from (3), (4), and the endoscopic character relation after choosing $f^1_v$ and  ${}^\theta \! f^2_v$ such that $f^1_v \star {}^\theta \! f^2_v = f_v$. 
\end{proof}

We derive two consquences of Lemma \ref{splitconsistency}:

\begin{lem}\label{splitconstant}
Let $v$ be a place that splits in $E$ and $\psi_v = \psi_{1,v} \oplus \psi_{2,v}$ be an Arthur parameter for $U_{E/F}(N)(F_v) \simeq \GL_N(F_v)$. Then $\psi$ factors through (the $L$-dual of) a Levi subgroup $M = \GL_{N_1}(F_v) \times \GL_{N_2}(F_v)$ and 
\[
\tr_{\psi_v} f = \prodf_{i=1,2} \tr_{\psi_{i,v}} f_{M,i},
\]
where $f_{M,1}$ and $f_{M,2}$ are the factors of the constant term map to $M$.
\end{lem}

\begin{proof}
The assumptions imply that $\td \pi_{\psi_v}$ is a parabolic induction of $\td \pi_{\psi_{1,v}} \otimes \td \pi_{\psi_{2,v}}$, so in the notation of Lemma \ref{splitconsistency} (1),  $\pi_{\psi_v}^0$ is a parabolic induction of $\pi^0_{\psi_{1,v}} \otimes \pi^0_{\psi_{2,v}}$. The result follows from \ref{splitconsistency} (5).
\end{proof}

\begin{lem}\label{split[d]}
Let $v$ be a place that splits in $E$ and $\psi_v = \psi_{1,v}[d]$ be an Arthur parameter for $U_{E/F}(N)(F_v) \simeq \GL_N(F_v)$. Then $\psi|_{W_F}$ factors through (the $L$-dual of) a Levi subgroup $M = (\GL_{N_1})^d$. 
Furthermore, for all test functions $f$ satisfying:
\begin{itemize}
    \item $f$ is supported on the kernel of $|\det|_v$,
    \item for all unirreps $\pi_v$ of $\GL_N(F_v)$, $tr_{\pi_v}(f) \geq 0$,
\end{itemize}
we have
\[
\tr_{\psi_v} f \leq (\tr_{\psi_{1,v}} f_{M,1})^{d\oplus},
\]
where $f_{M,1}$ is the constant term to $M$ restricted to the first factor. 
\end{lem}

\begin{proof}
This is a similar argument to \ref{splitconstant} using that $\td \pi_{\psi_v}$ is a summand in the Grothendieck group of the parabolic induction of 
\[
\td \pi_{\psi_{1,v}} |\det|^{\f{d-1}2} \otimes \cdots \otimes \td \pi_{\psi_{1,v}} |\det|^{\f{1-d}2}.
\]
We may ignore the determinant factors by the support condition. The inequality comes from the positivity condition applied to traces against the other summands in the Grothendieck group expansion. 
\end{proof}

\subsection{Unramified Places}
We next discuss the unramified places. Our results rely on the fact that 
unramified transfer through any $L$-map can be made very explicit via the Satake parameters. 
\subsubsection{Basic transfer result}
First we recall an ``Arthur packet fundamental lemma'' that was the key result making the strategy of \S\ref{strategy} work at level-$1$ in \cite{Tai17}:
\begin{lem}[Fundamental Lemma for $A$-packets]\label{packetfunlem}
Let $v$ be a place unramified in $E$ and 
\[
\psi_v = \bigoplus_i \tau_{i,v}[d_i]
\]
be an Arthur parameter for $U_{E/F}(N)(F_v)$. Then
\[
\tr_{\psi_v} \bar \1_{K_v} = \prod_i  \tr_{\psi_{i,v}} \bar \1_{K_{i,v}}
\]
for appropriately chosen hyperspecial subgroups $K_{\star, v}$ in appropriate $U_{E/F}(N_\star)(F_v)$.
\end{lem}

\begin{proof}
This follows from \cite[Lem 4.1.1]{Tai17} since $\psi_v$ is unramified if and only if each of the $\tau_{i,v}$ is. Then, just apply that for any $\pi_v$, $\tr_{\pi_v} \bar \1_{K_{i,v}} = \1_{\pi_v \text{ is unram.}}$.  

The $K_{\star, v}$ are chosen as in the fundamental lemma according to a choice of Whittaker datum. 
\end{proof}

Our eventual goal is to prove bounds in the style of \cite{ST16}, so we need a more general statement for any element of $\ms H^\ur(G_v)$. 

\subsubsection{Truncated Hecke algebras}
We first recall the notion of a truncated Hecke algebra from \cite{ST16}. The elements $\tau^G_\lb = \1_{K_v \lb(\varpi) K_v}$ for a chosen unformizer $\varpi$ and~$\lb \in X_*(A)^+$ generate $\ms H^\ur(G_v)$. Pick a basis $\mc B$ for $X_*(A)$ and define the norm
\[
\|\lb\|_\mc B = \max_{\om \in \Om} (\text{biggest }\mc B\text{-coordinate of }\om \lb).
\]
For $\lb \in X_*(A)$, define the truncated Hecke algebra
\[
\ms H(G,K)^{\leq \kappa, \mc B} = \langle \tau^G_\lb : \|\lb\|_\mc B \leq \kappa \rangle.
\]
It turns out (see \cite[\S2]{ST16}) that for any two $\mc B,\mc B'$, $\|\lb\|_\mc B$ and $\|\lb\|_{\mc B'}$ are proportional. All the bounds we use will depend on $\kappa$ only up to an unspecified constant. Therefore we can suppress the $\mc B$.

\subsubsection{Basis of characters}
Recall that the Satake transform gives an isomorphism
\[
\ms H^\ur(G_v) \iso \C[X_*(A)]^{\Om_F},
\]
where $A$ is a maximally split maximal torus of $G_v$ in good position with respect to~$K_v$. The right side of this isomorphism has a basis $\chi_\lb$ of trace characters of finite-dimensional representations $\lb$ of the twisted group $\wh G \rtimes \Frob_v$. 

Any unramified parameter $\psi_v$ determines an unramified $L$-parameter which determines a Satake parameter: a semisimple conjugacy class $\sigma_{\psi_v} \in \wh G \rtimes \Frob_v$. Because of \cite[Lem 4.1.1]{Tai17}, this satisfies that
\begin{equation}\label{sataketrace}
\tr_{\psi_v} \chi_\lb  = \tr_\lb(\sigma_{\psi_v}).
\end{equation}
See \cite[\S2.2]{ST16} for more detail.

The consistency of unramified packets constructed by Arthur/Mok and the Satake isomorphism is implicit in the isolation of the $\psi$-part $I_\psi$ of $I_\disc$---see \S3.3 in \cite{Art13} for example. It depends on the full fundamental lemma for all spherical functions.

\subsubsection{General unramified transfer}
Let $v$ be a place unramified in $E/F$ and 
\[
\psi_v = \bigoplus_i \tau_{i,v}[d_i] \in \Delta = (t_i, d_i, \lb_i, \eta_i)_i
\]
be an Arthur parameter for $U_{E/F}(N)(F_v)$. Let each $\tau_{i,v}$ be a parameter for $U(t_i)$. There is an associated embedding 
\begin{equation}\label{shapeLembedding}
\iota_\Delta: \m H := \Ld H_v \times \prod_i \Ld \GL_{d_i} := \prod_i \Ld (U(t_i) \times \GL_{d_i}) \into \Ld G_v.
\end{equation}
Let $\psi^I_n$ be the parameter of the trivial representation on $\GL_n$. We can write the Langlands parameter $\varphi_{\psi_v}$ corresponding to $\psi_v$ as the pushforward of
\[
\prod_i \tau_{i,v} \times \psi^I_{d_i}. 
\]
This gives the map on Satake parameters
\begin{equation}\label{satakepushforward}
\sigma_{\psi_v} = \mc S_\Delta((\sigma_{\tau_{i,v}})_i) := \iota_\Delta\lf(\prod_i \sigma_{\tau_{i,v}} \times \sigma^I_{d_i} \ri),
\end{equation}
for $\sigma^I_{d_i}$ the Satake parameter of the $d_i$-dimensional trivial representation. 

Restricting $\iota_\Delta$ 
to $\star \rtimes \Frob$ cosets determines an unramified transfer map:
\[
\mc T_\Delta : \ms H^\ur(G_v) \to \ms H^\ur(H_v) : \chi_\lb \mapsto \chi_\lb \circ \mc S_\Delta.
\]
Equation \eqref{sataketrace} then gives:
\begin{lem}\label{unramtrace}
With notation as above, let $f \in \ms H^\ur(G_v)$. Then
\[
\tr_{\psi_v}(f) = \prodf_i \tr_{\tau_{i,v}}(\mc T_{\Delta,i} f),
\]
where the $\mc T_{\Delta,i}$ are the factors of $\mc T_\Delta f$ at each $U(N_i)$. 
\end{lem}

The next lemma gives some control over the size of the $\mc T_{\Delta,i} f$. 
\begin{lem}\label{unramtransferbound}
Let $f \in \ms H^\ur(G_v)^{\leq \kappa}$ with $\|f\|_\infty \leq 1$. Then $\mc T_{\Delta} f \in \ms H^\ur(H_v)^{\leq \kappa}$ and $\|\mc T_\Delta f\|_\infty = O(q_v^{D\kappa} \kappa^E)$ for some constants $D$ and $E$ that only depend on $G$ and $\Delta$. 
\end{lem}

\begin{proof}
This is a slightly more complicated version of the argument of Lemma 5.5.4 in \cite{Dal22}. There is an additional step from bounding the trace of the Satake parameter of the trivial representation against the finite-dimensional irreps of $\Ld \GL_{d_i}$ that appear as factors in restrictions from $\Ld G_v$ to $\m H$. This can be seen to be $O(q_v^{D' \kappa})$ for some $D'$ by the Weyl character formula. 
\end{proof}

We similarly define a partial unramified transfer map $\mc T_{\bar \Delta}$ such that
\[
\tr_{\psi_v}(f) = \prodf_i \tr_{\tau_{i,v}[d_i]}(\mc T_{\bar \Delta,i} f).
\]
By the same arguments, this map satisfies Lemma \ref{unramtransferbound}. We will often suppress~$\Delta$ or $\bar \Delta$ in notation---which version of $\mc T$ we are using should always be clear by the context of what the $\mc T_i$'s have image in. We note for intuition that 
 $\mc T_{\bar \Delta}$ can be thought of as hyperendoscopic transfer as in Theorem \ref{hyperendoscopy}, see \cite[\S5.5]{Dal22} for details.

Finally, if $\Delta$ is a simple shape of the form $(1, d, \lb, \eta)$, then for $\psi \in \Delta$ the possible $\sigma_{\psi_v}$ from \eqref{satakepushforward} are all Satake parameters of characters. In addition, $H_v = U_1(F_v) = (G_v)^\ab = (G^\ab)_v$ (note that $G_{v, \der}$ is semisimple and simply-connected and asu such has trivial cohomology). We can therefore compute
\begin{equation}\label{Tcharactershape}
\mc T_\Delta(f_v)(h) =  \int_{G_{v, \der}} f_v(hg) \, dg.
\end{equation}

\subsection{General Places: Inequalities}
We can also say a little at general places:
\subsubsection{Twisted Bernstein components}
We recall some facts from \cite[\S6]{Rog88} on Bernstein components for twisted groups. Let $G_v$ be the $F_v$ points of a connected reductive group and let $\wtd G_v = G_v \rtimes \theta$ be a twisted group for some automorphism~$\theta$. We assume $G_v$ has a minimal parabolic $P_0$ and Levi factor $M_0$ that are both $\theta$-stable. Let $\mc L(G_v)$ be the set of standard Levis of $G_v$ with respect to $M_0$. 

\begin{dfn}
The twisted cuspidal supports for $\wtd G_v$ are pairs $(M, \sigma)$ with~$M \in \mc L(G_v)$ such that $(\theta M, \theta \sigma)$ is conjugate to $(M, \sigma)$ in $G_v$. 
\end{dfn}

\begin{dfn}
We say that a $\theta$-invariant irrep $\pi$ of $G_v$ has infinitesimal character~$(M, \sigma)$ if $\pi$ is an irreducible subquotient of $\Ind_{MP_0}^{G_v} \sigma$. 
\end{dfn}

Every $\theta$-invariant representation has an infinitesimal character that is a twisted cuspidal support.

\begin{dfn}
Let $\td \pi$ is an irrep of $\wtd G_v$ with non-zero (twisted trace) character. Then the infinitesimal character of $\td \pi$ is that of $\td \pi|_{G_v}$.
\end{dfn}

Note that $\td \pi|_{G_v}$ is necessarily $\theta$-invariant if $\td \pi$ has non-zero character. Furthermore, all extensions of $\td \pi|_{G_v}$ differ by a root of unity of order dividing that of $\theta$. 


\begin{dfn}
Let $(M, \sigma)$ be a twisted cuspidal support. Its twisted Bernstein component is the set of irreps  $\td \pi$ with non-zero character on $\wtd G_v$ such that their infinitesimal characters are of the form $(M, \sigma \chi)$ for $\chi$ an unramified character of $M$ (and $\sigma \chi$ $\theta$-invariant). 
\end{dfn}

Part of the main result of \cite{Rog88} is the following key point:
\begin{lem} \label{lem support Bernstein Rogawski}
Let  $f_v$ be a compactly supported, smooth function on $\wtd G_v$. Then $\pi \mapsto \tr_\pi f_v$ is supported on a finite number of twisted Bernstein components.
\end{lem}



Now specialize to the case of $G = \td G(N)$: 
\begin{prop}\label{finiteBcomponent}
Let $\mf s$ be a twisted Bernstein component of $\wtd G(N)_v$ and let~$H \in \wtd {\mc E}_\el(N)$. Then there is a finite list $\mf s_1, \dotsc, \mf s_n$ of Bernstein components of~$H_v$ such that the Arthur packets $\Pi_{\psi_v}$ for $\psi \in \Psi^+_v(H)$ with $\wtd \pi_\psi \in \mf s$ only contain representations in the $\mf s_i$.
\end{prop}

\begin{proof}
This follows from three facts: first, the compatibility of twisted endoscopic transfer of characters with Jacquet modules as in diagram (C.4) in \cite{Xu17}, second, the finiteness of $A$-packets, and third, the compatibility of transfer of characters with unramified character twist as in proposition 4.4 of \cite{Oi21}. 

We thank Masao Oi for pointing this out to us. 
\end{proof}

\subsubsection{Inequalities}
This lets us show:
\begin{lem}\label{gen[d]}
Let $v$ be a finite  place of $F$ and $f_v$ be a test function on $U_{E/F}(N)(F_v)$. Let $dN_1 = N$. Then there exists test function $\varphi_{v}$ on $U_{E/F}(N_1)(F_v)$ such that for all Arthur parameters $\psi_v = \psi_{1,v}[d]$ with $\psi_{\star, v}$ a parameter for $U_{E/F}(N_\star)(F_v)$ and $\psi_{1,v}$ cuspidal:
\[
|\tr_{\psi_v} f_v| \leq (\tr_{\psi_{1,v}} \varphi_{v})^{d}.
\]
\end{lem}

\begin{proof}
First,
\[
\tr_{\psi_v} f_v = \tr_{\td \pi_{\psi_v}} f_v^N.
\]
By Bernstein's admissibility theorem (as used in \cite[Prop 9.6]{ST16}) there is $C$ (without loss of generality, $C > 1$) such that for all unirreps $\td \pi'$ of $\wtd G(N)_v$, 
\[
|\tr_{\td \pi'} f_v^N| \leq C.
\]
Now consider the function from the unitary dual of $\wtd G(N_1)_v$ to $\C$ given by 
\[
\Phi_0 : \td \pi \mapsto tr_{\td \pi[d]} f_v^N.
\]
This is supported the on finitely many Bernstein components $\mf s_1, \dotsc, \mf s_i$ by Lemma \ref{lem support Bernstein Rogawski} and since
the cuspidal support of $\td \pi[d]$ determines that of $\td \pi$. 

Therefore, the representations $\pi \in \Pi_{\psi_{1,v}}$ for $\td \pi_{\pi_{1,v}} \in \mf s_i$ lie in a finite set of Bernstein components by Proposition \ref{finiteBcomponent}. Since $U_{E/F}(N_1)$ is not twisted,  one can find a function $\varphi_v$ on $U_{E/F}(N_1)$ so that $\tr_\star \varphi_v \geq C$ on each of these components (e.g. a scalar multiple of the indicator function of a small enough maximal compact depending on the Bushnell-Kutzko types associated to the Bernstein components). 

Finally, since $\psi_{1,v}$ is cuspidal and therefore simple, $\tr_{\psi_{1,v}} \varphi_v$ is a sum of various $\tr_\pi \varphi_v$ and in particular larger. Therefore, this choice of $\varphi_v$ suffices.
\end{proof}

\begin{lem}\label{genconstantineq}
Let $v$ be a place of $F$ and $f_v$ be a test function on $U_{E/F}(N)(F_v)$. Let $N_1 + \cdots + N_n = N$. Then there exists test functions $f_i$ on $U_{E/F}(N_i)(F_v)$ such that for all Arthur parameters $\psi_v = \psi_{1,v} \oplus \cdots \oplus \psi_{n,v}$ with each $\psi_{\star, v}$ a simple parameter for $U_{E/F}(N_\star)(F_v)$:
\[
|\tr_{\psi_v} f_v| \leq \prod_i \tr_{\psi_{i,v}} f_{i,v}.
\]
\end{lem}

\begin{proof}
This is the same argument as Lemma \ref{gen[d]}. 
\end{proof}

\subsection{General Places: Conjectural Equalities}\label{sectionconjecturaltransfer}
At this time, exact equality results for transfers at general places are not known. This is the main reason for the restriction of our exact asymptotic result \ref{mainexact} to unitary groups for unramified $E/F$. We state the desired conjecture:

\begin{conj}[Full Transfer]\label{stabletransferconj}
Let $\Delta = ((T_i, d_i, \eta_i, \lb_i))_i$ be a shape and $\varphi_v$ a trace-positive test function on $G_v$. Then there are trace-positive functions $\varphi_{i,v}$ on the appropriate quasisplit unitary groups of rank $T_i$ such that for all $\Delta \ni \psi = \bigoplus_i \psi_i[d_i]$ with $\psi_i \in (T_i, 1, \eta_i, \lb_i)$,
\[
\tr_{\psi_v} (\varphi_{v}) = \prod_i \tr_{\psi_{i,v}} (\varphi_{i,v}).
\]
\end{conj}

Except for the positivity statement, Conjecture \ref{stabletransferconj} would be implied by two local conjectures on the $\wtd G(N)_v$, the first of which is expected by experts:

\begin{conj}[Stable Transfer]\label{stabletransferconj1}
Let $\Delta = (\Delta_i)_i = (T_i, d_i, \eta_i, \lb_i)_i$ be a shape of rank $N$ and $\varphi_v$ a test function on $\wtd G(N)_v$. Then there are test functions $\varphi_{i,v}$ on each $\wtd G(T_i d_i)_v$ such that for all choices of $\psi_i \in \Delta_i$,
\[
\tr_{\wtd \pi_{\psi_v}}(\varphi_{v}) = \prod_i \tr_{\wtd \pi_{\psi_{i,v}}} (\varphi_{i,v}).
\]
\end{conj}

Note that $\pi_{\psi}$ is just the parabolic induction of the product of the $\pi_{\psi_i}$. In contrast, the relation between the between~$\wtd \pi_{\psi}$ and the~$\wtd \pi_{\psi_i}$ is more complicated and depends on the choice of extension in Section \ref{ACrepexten}.


The second necessary local conjecture is:
\begin{conj}[Speh Transfer]\label{conjspehtransfer}
Let $\Delta = (T, d, \eta, \lb)$ be a simple shape and $\varphi_v$ a test function on $\wtd G(Td)$. Then there is a test function $\varphi'_v$ on $\wtd G(T)$ such that for all $\psi \in (T, 1, \eta, \lb)$, 
\[
\tr_{\wtd \pi_{\psi_v[d]}} (\varphi_v) = \tr_{\wtd \pi_{\psi_v}} (\varphi'_v).
\]
\end{conj} 
For Speh transfer, even the relation between the untwisted representations $\pi_{\psi_v[d]}$ and $\pi_{\psi_v}$ is difficult due to the reducibility of the relevant parabolic inductions.

\section{Induction Step: General Shapes}\label{step3gen}
In the following two sections, we prove the identities needed in the third, inductive step \S \ref{ss step 3} of the strategy outlined in \S \ref{strategy}. While our goal in \S \ref{strategy} was to produce exact formulas, in our applications we only need to find an asymptotic formula for $S^H_\Delta(\EP_\lb f^\infty)$. Therefore, we will only need to solve two problems:
\begin{itemize}
    \item In this section, prove an upper bound (Theorem \ref{shapeineq}) for a general shape to show that terms for non-dominant shapes are negligible. 
    \item In \S \ref{step3dom}, find an exact asymptotic for terms with the types of shapes that could be dominant in our application. 
\end{itemize}
This section heavily uses our notations from \ref{nonfactorizable} for non-factorizable functions.

\subsection{Preliminary Bound}
For any shape $\Delta$ of rank $N$, recall from \S \ref{ss SM} that
\[
S^N_\Delta(f^\infty) := \sum_{\psi \in \Delta} S^N_\psi(f^\infty) =  \sum_{\psi \in \Delta}  \eps^{H(\Delta)}_\psi (s^{H(\Delta)}_\psi) m_\psi |S_\psi|^{-1} \tr_{\td \pi_\psi^\infty}(f^\infty).
\]
Note that Proposition \ref{SMform} still holds with subscripts of $\Delta$ added to both sides. 

By removing the $\eps$-sign, we also defined in \ref{ss step 3}:
\begin{equation}
S^{|H(\Delta)|}_\Delta(\EP_\lb (f^\infty)^{H(\Delta)}) = S^{|N|}_\Delta(f^\infty) := \sum_{\psi \in \Delta} m_\psi |S_\psi|^{-1} \tr_{\wtd \pi_\psi^\infty}(f^\infty)
\end{equation}
 To approximate away $\eps_\psi$ difficulties, we extend the main trick of \cite{Ger20} and prove a technical bound relating $S_\Delta^N$ to~$S_{\Delta'}^{|H'|}$ terms for a suitable choice of $H'$. 
 It relies on a condition that we will assume henceforth for various test functions:

\begin{dfn}
Let $S$ be a (finite or infinite) set of places of $F$ and $f_S$ a function on $G_S$. We say that $f_S$ is trace-positive on $S$ if for all unirreps $\pi_S$ of $G_S$, $\tr_{\pi_S} f_S \geq 0$. 
\end{dfn}

\begin{prop}\label{SMineq}
Let $H = H(\Delta) \in \wtd {\mc E}_\el(N)$. Assume that $f^\infty$ is trace-positive. Then there is an elliptic endoscopic group $H' = H_1 \times H_2$ of $H$ such that for any infinitesimal character $\lb'$ of $H'$ conjugate to $\lb$ over $H$:
\[
|S^{N}_\Delta((f^\infty)^N)| \leq C \prodf_{i=1,2} S^{H_i}_{\Delta_i}(\EP_{\lb'_i} (f^\infty)^i).
\]
Here we use shorthand: the data $\Delta, \lb'$, and $(f^\infty)^{H'}$ all factor into components for $H_1$ and $H_2$ denoted with appropriate sub/superscripts. The $C$ is a constant depending only on $H$ and $\Delta$. 

The $H'$ further satisfies that
\[
S^{H_1}_{\Delta_1}(\EP_{\lb'_1} (f^\infty)^1) S^{H_2}_{\Delta_2}(\EP_{\lb'_2} (f^\infty)^2) = S^{|H_1|}_{\Delta_1}(\EP_{\lb'_1} (f^\infty)^1) S^{|H_2|}_{\Delta_2}(\EP_{\lb'_2} (f^\infty)^2).
\]
\end{prop}

\begin{proof}
Moving the absolute value within the sum from an intermediate computation in \ref{SMform}, we know that
\[
 |S^{N}_\Delta((f^\infty)^N)|  \leq \sum_{\psi \in \Delta} m_\psi |\mc S_\psi|^{-1}
\lf( \sum_{\pi \in \Pi^H_{\psi^\infty}} \tr_{\pi^\infty}(f^\infty) \ri).
\]

Consider $\psi \in \Delta$. All such $\psi$ have the same $s_\psi$ and therefore the same $H'$ in the pair $(H', \psi^{H'})$ corresponding to $s_\psi$ as in Proposition \ref{prop s to endo}. The $\Out(\wh H')$-orbit of $(H', \psi')$ that comes from $(\psi, s_\psi)$ is the subset with image in~$\wh H'$ of an $\wh H$-conjugacy class. The elements of the orbit can therefore be determined by their infinitesimal character at $\infty$---i.e. the choice of $\lb'$ uniformly determines a unique choice of parameter $\psi'$ on $H'$ for each $\psi \in \Delta$. 

Denote by $f'$ the transfer to $H'$. By Theorem \ref{stablemultendo} applied to $H(s_\psi) = H'$, iterations of the  computations of Proposition \ref{SMform}, and the twisted character identity away from infinity:
\begin{multline*}
S_{\psi'}^{H'}(\EP_{\lb'} (f^\infty)') =  \eps^G_{\psi}(s_{\psi}s_\psi) m_{\psi'} |\mc S_{\psi'}|^{-1}  \lf( \sum_{\pi \in \Pi^H_{\psi^\infty}} \eta^{\psi^\infty}_{\pi^\infty}(s_\psi s_\psi) \tr_{\pi^\infty}(f^\infty) \ri) \\
=  m_{\psi'} |\mc S_{\psi'}|^{-1}  \lf( \sum_{\pi \in \Pi^H_{\psi^\infty}} \tr_{\pi^\infty}(f^\infty) \ri),
\end{multline*}
with the second equality following from $s_\psi^2=1$. The stable multiplicity formula \ref{stablemult} also gives
\[
S_{\psi'}^{H'}(\EP_{\lb'} (f^\infty)') = S_{\psi'}^{|H'|}(\EP_{\lb'} (f^\infty)')
\]
since $\epsilon^{H'}_{\psi'}(s_{\psi'})$. Summing over $\psi \in \Delta$ on one side and the corresponding $\psi'$ on the other gives
\[
|S^{N}_\Delta((f^\infty)^N)| \leq \f{m_{\psi} |\mc S_{\psi}|^{-1}}{m_{\psi'} |\mc S_{\psi'}|^{-1}}S_{\Delta}^{H'}(\EP_{\lb'} (f^\infty)')
\]
where $S^{H'}_\Delta$ is defined by summing over parameters that pushforward to something of the right shape on $H$. Both desired expressions then follow from factoring the~$H'$ term into terms for~$H_1$ and $H_2$ in these last two formulas.
\end{proof}

\subsection{Reduction to Simple Shapes}\label{rtss}
We now bound $S_\Delta$ in terms $S_{\Delta'}$ for cuspidal~$\Delta'$. We do this in two steps: in this section we reduce to simple shapes and in the next we pass from simple to cuspidal. 

To apply results of \S\ref{sectionlocaltransfers}, we present test functions in a particular form: choose $G \in \wtd {\mc E}_\el(N)$ and let $S = \infty \sqcup S_s \sqcup S_b$ be finite sets of places such that $S_s$ is split and $S_b$ contains all the ``bad'' ramified places. Write test functions as:
\[
f^\infty = \varphi_{S_b} f_{S_s} f^S
\]
where $\varphi_{S_b}$ and $f_{S_s}$ are arbitrary and $f^S \in \ms H^\ur(G^S)$. 
\begin{prop}\label{SNfactor}
Consider the shape $\Delta = (T_i, d_i, \lb_i, \eta_i)_i$. Assume that~$f_{S_s}$ and $f^S$ are trace-positive. Then
\[
S^{|N|}_\Delta((f_{S_s} f_{S_b} f^S)^N) \leq C_\Delta \prodf_i S^{T_i d_i}_{(T_i, d_i, \lb_i, \eta_i)}(((f_{S_s})_{M,i} \varphi_{i, S_b} \mc T_i f^S)^{T_id_i})
\]
for trace-positive $\varphi_{i, S_b}$ and a constant $C_\Delta$ that only depends on $\Delta$.
\end{prop}

\begin{proof}
We first work with a single $\psi = \tau_1[d_1] \oplus \cdots \oplus \tau_k[d_k] \in \Delta$ and compute the terms in the stable multiplicity formula. Iteratively applying Lemma \ref{splitconstant} on transfer for split groups, we realize 
\[
M = \prod_i \GL_{T_i d_i}(F_{S_s})
\]
as a Levi of $G(F_{S_s}) \simeq \GL_N(F_{S_s})$ with
\[
\tr_{\psi_{S_s}} f_{S_s} = \prodf_i \tr_{\tau_{i,S_s}[d_i]} (f_{S_s})_{M,i},
\]
where $(\star)_{M,i}$ is the $i$th factor of the constant term map to $M$. 

For the places not in $S$, Lemma \ref{unramtrace} gives that 
\[
\tr_{\psi^S} f^S= \prodf_i \tr_{\tau_i^S[d_i]} \mc T_i f^S,
\]
and multiplying over places in $S_b$, Lemma \ref{genconstantineq} constructs $\varphi_{i, S_b}$ so that
\[
|\tr_{\psi_{S_b}} (\varphi_{S_b})| \leq \prod_i \tr_{\tau_{i,S_b}[d_i]} (\varphi_{i,S_b}).
\]

Next, each $\eps_{\tau_i[d_i]}$ is trivial for each of the $\tau_i[d_i]$ since they are simple. Multiplying together the above trace identities, we have
\begin{equation} \label{eqn product trace}
    \lf| S^{|N|}_\psi((f_{S_s} \varphi_{S_b} f^S)^N) \ri| \leq C_\Delta \lf| \prodf_i S^{T_i d_i}_{\tau_i[d_i]}(((f_{S_s})_{M,i} \varphi_{i, S_b} \mc T_i f^S)^{T_i d_i}) \ri|
\end{equation}
for some constant $C_\Delta$ that only depends on $\Delta$.

By regularity of the infinitesimal character, none of the $(T_i, d_i, \lb_i, \eta_i)$ appear with multiplicity more than $1$. Therefore, summing \eqref{eqn product trace} over all $\psi \in \Delta$ is the same as iteratively summing over each factor $\tau_i[d_i] \in (T_i, d_i, \lb_i, \eta_i)$. 

To remove the absolute values on the right side of the summed equation, $(f_{S_s})_{M,i}$ and $\mc T_i f^S$ satisfy the same positivity condition: positivity for~$(f_{S_s})_{M,i}$ follows by ``adjointness'' of constant term and parabolic induction. That for $\mc T_i f^S$ comes from the ``adjointness'' between $\mc T_i$ and pushforward of Satake parameter. The $\varphi_{i,S_b}$ satisfy positivity by construction. 
Together with the fact that the $\eps_{\tau_i[d_i]}$ are trivial, this guarantees that the absolute value can be removed on the right of \eqref{eqn product trace}.  
\end{proof}

Doing the bookkeeping for what exactly $\Delta_1, \Delta_2$ are in Proposition \ref{SMineq} lets us bound $S^N$ instead of $S^{|N|}$:
\begin{cor}\label{constanttermineq}
Let the notation be as in the above discussion and assume that $f_{S_s}$, $\varphi_{S_b}$, and $f^S$ are all trace-positive.
Then for some functions $\varphi'_{i,S_b}$ and some constant $C_\Delta$ only depending on $\Delta$,
\[
|S^{N}_\Delta((f_{S_s} f_{S_b} f^S)^N)| \leq C_\Delta  \prodf_i S^{T_i d_i}_{(T_i, d_i, \lb_i, \eta_i)}(((f_{S_s})_{M,i} \varphi'_{i, S_b} \mc T_i f^S)^{T_i d_i}).
\]
Furthermore, $(f_{S_s})_{M,i}$, $\varphi'_{i,S_b}$ and $\mc T_i f^S$ satisfy the same positivity condition. 
\end{cor}

\begin{proof}
First, apply Proposition \ref{SMineq} to get a bound by $S^{|N|}$ terms and then apply Proposition \ref{SNfactor} to the two terms on the right side of the bound in \ref{SMineq}. Note that we need to use the full fundamental lemma (see e.g. Theorem 5.4.2 in \cite{Dal22}) to choose a transfer $f^S$ to $H'$ that is dual to the pushforward of Satake parameters. The $\varphi'$ are constructed through the above discussion applied to transfer to $H'$ of $f_{S_b}$. Finally, the transfer of $f_{S_s}$ to $H'$ is given by taking constant terms since we are in a degenerate case of $H'_{S_s}$ being a Levi subgroup.

The positivity condition for the various transfers follows from the exact same arguments as in the last proof. 
 \end{proof}
 
 \subsection{Full Bound}\label{fb}
Now we reduce from simple 
to cuspidal shapes. Let $G \in \wtd {\mc E}_\el(N)$ and $\Delta = (T, d, \lb, \eta)$ a simple shape for $G$, so that all $\psi \in \Delta$ are of the form $\psi = \tau[d]$. As before, choose a test function
 \begin{equation} \label{eq good test function}
 f^\infty = f_{S_s} f_{S_b} f^S
 \end{equation}
 on $G$ where all places in $S_s$ are split, $S_b$ contains all ramified places, and $f^S \in \ms H^\ur(G^S)$. Further assume:
 \begin{itemize}
     \item $f_{S_s} f_{S_b}$ is supported on the kernel of $|\det|_{S_s \sqcup S_b}$, \label{cond support}
     \item $f_{S_s}$ and $f^S$ are trace-positive. 
 \end{itemize}
 
The argument is very similar to the previous section. First,
 \begin{multline*}
 |S_\Delta^M((f_{S_s} f_{S_b} f^S)^M)| \leq  \sum_{\psi \in \Delta} |S^M_\psi((f_{S_s} f_{S_b} f^S)^M)| \\=  \sum_{\psi \in \Delta} m_\psi |\mc S_\psi| (\tr_{\psi_{S_s}} f_{S_s}) (\tr_{\psi^S} f^S) |\tr_{\psi_{S_b}} f_{S_b}|.
 \end{multline*}
 For all $\psi \in \Delta$, apply Lemma \ref{split[d]} to get a Levi $M_{S_s} \simeq {\GL_T}^d$ of $\GL_M(F_{S_s})$ so that
 \[
 \tr_{\psi_{S_s}} f_{S_s} \leq (\tr_{\tau_{S_s}} (f_{S_s})_{M,1})^{d\oplus}.
 \]
We also apply \ref{gen[d]} to construct the functions $\varphi'_{S_b}$ satisfying
\[
|\tr_{\psi_{S_b}} \varphi_{S_b}| \leq (\tr_{\tau_{S_b}} \varphi'_{S_b})^{d\oplus}
\]
since the $\tau_{S_b}$ are necessarily cuspidal. Note that the $\varphi'_{S_b}$ so constructed is trace-positive by definition. Applying Lemma \ref{unramtrace} then gives
\begin{equation}\label{[d]boundeq}
|\tr_{\psi^\infty}(f_{S_s} \varphi_{S_b} f^S)| \leq (\tr_{\tau^\infty}((f_{S_s})_{M,1} \varphi'_{S_b} \mc T f^S))^{d\oplus}.
\end{equation}

Finally, summing over all $\tau[d] \in \Delta$ gives that:
\begin{prop}\label{[d]bound}
With notation and conditions from the above discussion:
\[
  |S_\Delta^M((f_{S_s} \varphi_{S_b} f^S)^M)| \leq C_\Delta (S_{\Sigma_{\lb, \eta}}^T((f_{S_s})_{M,1} \varphi'_{S_b} \mc T f^S)^T)^{d \oplus}
 \]
for some constant $C$ that only depends on $\Delta$. Recall that $(f_{S_s})_{M,1}$ and $\mc T f^S$ are trace positive and we can choose $\varphi'_{S_b}$ to be so too.
\end{prop}

\begin{proof}
Expanding out and using that $\Sigma_{\lb, \eta}$ is simple:
\begin{multline*}
(S_{\Sigma_{\lb, \eta}}^T((f_{S_s})_{M,1} \varphi'_{S_b} \mc T f^S)^T)^{d \oplus} \\
= \sum_{\tau \in \Sigma_{\lb, \eta}} m_\psi^d |\mc S_\psi|^d (\tr_{\tau^\infty} ((f_{S_s})_{M,1} \varphi'_{S_b} \mc T f^S))^{d\oplus} + \text{cross terms},
\end{multline*}
where trace-positivity comes from the conditions in the above discussion and arguments similar to Corollary \ref{constanttermineq}. Since $m_\psi$ and $|S_\psi|$ only depend on $\Sigma_{\lb, \eta}$, the inequality then follows from  \eqref{[d]boundeq} since trace positivity gives that the cross terms are all positive. We use here that $\Delta$ is simple so $\epsilon_\psi(s_\psi)$ is identically 1. 
%
%
\end{proof}

We will need a slight technical variation of this: 
\begin{prop}\label{[d]boundvar}
Fix notation and conditions as in Proposition \ref{[d]bound}. Assume there is a constant $B$ such that for all $\tau \in \Sigma_{\lb, \eta}$:
\[
|\tr_{\td \tau^\infty} (((f_{S_s})_{M,1} \varphi'_{S_b} \mc T f^S)^T)|^{\oplus} \leq B
\] where the $\oplus$ represents a sum over factorizable summands. Then 
\[
  |S_\Delta^M((f_{S_s} \varphi_{S_b} f^S)^M)| \leq C B^{d-1} S_{\Sigma_{\lb, \eta}}^T(((f_{S_s})_{M,1} \varphi'_{S_b} \mc T f^S)^T)^\oplus
\]
for some constant $C$ that only depends on $\Delta$.
\end{prop}

\begin{proof}
This is the same argument as \ref{[d]bound} except we bound
\[
\sum_{\tau \in \Sigma_{\lb, \eta}} m_\psi^d |\mc S_\psi|^d (\tr_{\tau^\infty} ((f_{S_s})_{M,1} \varphi'_{S_b} \mc T f^S))^{d\oplus}
\]
directly instead of adding in cross terms. 
\end{proof}

Our final bound for a general shape then becomes:
\begin{prop}\label{shapeineq}
Let $\Delta = (T_i, d_i, \eta_i, \lb_i)_i$ and $G = H(\Delta) \in \wtd {\mc E}_\el(N)$. Let
\[
f^\infty = f_{S_s} \varphi_{s_b} f^S
\]
be a test function on $G^\infty$ where $S_s$  is all split places, $S_b$ contains all the ramified places, and $f^S \in \ms H^\ur(G^S)$. We further require: 
\begin{itemize}
   \item $f_{S_s}$ is supported on the kernel of $|\det|_{S_s}$,
    \item $f_{S_s}$, $\varphi_{S_b}$, and $f^S$ are trace-positive.
\end{itemize}
Let $f^\infty = f_{S_s} \varphi_{s_b} f^S$ be a test function as in \eqref{eq good test function}. Then there is a Levi subgroup 
\[
M \simeq \prod_i \GL_{T_i}(F_v)^{d_i}
\]
of $G(F_{S_s}) \simeq \GL_N(F_{S_s})$ and functions $\varphi'_{i,S_b}$ such that
\[
|S_\Delta^N((f_{S_s} f_{S_b} f^S)^N)| \leq C_\Delta \prodf_i (S^{T_i}_{\Sigma_{\lb_i, \eta_i}}(((f_{S_s})_{M,i} \varphi'_{i, S_b} \mc T_i f^S)^{T_i}))^{d_i \oplus}.
\]
Here $(f_{S_s})_{M,i}$ is the restriction of the constant term to $M$ to the first 
factor in $\GL_{T_i}(F_v)^{d_i}$, $\mc T_i$ is defined similarly, and the constant $C_\Delta$ only depends on $\Delta$. 

Finally, we may choose $\varphi'_{i,S_b}$ so that $(f_{S_s})_{M,i}$, $ \varphi'_{i, S_b}$ and $\mc T_i f^S$ satisfy the same support and positivity conditions \ref{cond support} as $f_{S_s}, f_{S_b}$, and $f^S$.
\end{prop}

\begin{proof}
The inequality comes from applying \ref{constanttermineq} and then \ref{[d]bound}. The first thing to check is that the necessary conditions for \ref{[d]bound} still hold after applying \ref{constanttermineq}. Positivity is guaranteed by the second implication of \ref{constanttermineq}.  Support follows from the integral formula for constant term. 

The final $\varphi'_{i,S_b}$, $(f_{S_s})_{M,i}$, and $\mc T_i f^S$ satisfy the bulleted conditions by the last implications of \ref{[d]bound}. 
\end{proof}



\section{Induction Step: Particular Shapes}\label{step3dom}
We compute more detailed information for certain specific kinds of shapes with the goal of providing exact asymptotics:
\subsection{Shapes of Characters}
Let $\Delta = (1, d, \lb, \eta)$ be a shape for $H = H(\Delta)$. If~$\psi \in \Delta$, then the packet ~$\Pi_{\psi_\infty}$ constructed by Adams-Johnson, see \S \ref{AJcombinatorial}, is a singleton containing a character~$\xi_\infty$ of~$H$. Let $\EP_{\xi_\infty} = \vol(A_{G, \infty} \bs G_\infty / G_{\der, \infty})^{-1} \xi^{-1}_\infty$ be the corresponding Euler-Poincar\'e function (note that by our assumptions on~$E/F$, $H$ is cuspidal in the sense of \cite{Art89} so $ \xi^{-1}_\infty$ is a valid test function factor for the global trace formula). 

\begin{prop}\label{charshapeformula1} 
Let $\Delta = (1,d,\lambda,\eta)$ a shape for $H$ and $\xi_\infty$ be as above. 
 Then,
\[
S^H_\Delta (\EP_{\xi_\infty}  f^\infty) = \sum_{\substack{\chi \in \mc{AC}_\disc(H) \\ \chi_\infty = \xi_\infty}} \tr_{\chi^\infty} f^\infty,
\]
where $\mc{AC}_\disc(H)$ is the set of one-dimensional representations in the discrete automorphic spectrum of $H$. 
\end{prop}

\begin{proof}
If $\psi \in \Delta$, any $\pi \in \Pi_\psi$ has $\pi_\infty = \xi_\infty$ as explained above. This implies that $\pi$ is one-dimensional by a well-known result (see \cite[Lem 6.2]{KST16} for example). Furthermore, $\psi$ is simple, so as distributions,
\[
S_\psi =  I_\psi = \sum_{\pi \in \Pi_\psi} \tr_\pi.
\]
Evaluating at our test function,
\[
S_\psi(\EP_{\xi_\infty}  f^\infty) = I_\psi(\EP_{\xi_\infty}  f^\infty) = \sum_{\pi \in \Pi_\psi} \tr_{\pi^\infty}(f^\infty). 
\]

On the other hand, any $\chi \in \mc{AC}_\disc(H)$ appears in some $A$-packet $\Pi_\psi$. If $\chi_\infty = \xi_\infty$, then since $H$ is quasisplit, $\psi_\infty$ has full Arthur-$\SL_2$ and infinitesimal character $\lb$, which forces $\chi \in \Delta$. Furthermore, we recall that characters appear with multiplicity at most one in the automorphic spectrum (realized as functions by evaluation). Therefore every such $\chi$ can only appear in one packet. 
In total, the union over $\Pi_\psi$ for $\psi \in \Delta$ is exactly the subset of $\mc{AC}_\disc(H)$ with infinite component $\xi_\infty$ and this union is disjoint.

Summing over $\psi \in \Delta$ then finishes the argument. 
\end{proof}

 Next, $\mc{AC}_\disc(H)$ are the characters of 
\[
\Xi(H) := H(F)^\ab \bs H(\A)^\ab
\]
so we can therefore write
\[
S^H_\Delta (\EP_{\xi_\infty}  f^\infty) = \f1{\vol(H^\ab_\infty)}\sum_{\chi \in \Xi(H)^\vee} \wh{\xi_\infty^{-1} f^{\infty, \ab}}(\chi),
\]
where $f^{\infty, \ab}$ is the pushforward (by integration against $H(\A^\infty)_\der$) of $f^\infty$ to $H(\A^\infty)^\ab$. Then:
\begin{cor}\label{charshapeformula}
Let $H = H(\Delta)$ for $\Delta=(1, d, \lb, \eta)$ corresponding to a character~$\xi_\infty$ on $H_\infty$. Then, with $f^{\infty, \ab}$ as above:
\[
S^H_\Delta (\EP_{\xi_\infty}  f^\infty) = \f{\vol(H(F)^\ab \bs H(\A)^\ab)}{\vol(H^\ab_\infty)} \sum_{\gamma \in H(F)^\ab} \xi_\infty^{-1} f^{\infty, \ab}(\gamma).
\]
\end{cor}

\begin{proof}
Since $H^\ab(F) \bs H^\ab(\A)$ is compact for our specific $H$, we get that $H(F)^\ab$ is cocompact in $H(\A)^\ab$. Since $H(F)^\ab$ is a subgroup of $H^\ab(F)$ and $H(\A)^\ab$ is an open subgroup of $H^\ab(\A)$, discreteness of $H^\ab(F)$ in $H^\ab(\A)$ gives discreteness of~$H(F)^\ab$ in $H(\A)^\ab$. 

Therefore Poisson summation gives the result.
\end{proof}

Our $H(\Delta)$ is necessarily isomorphic as a reductive group to $U(N)$ so we will have $H^\ab = U(1)$ in the above.

\subsection{Odd GSK Shapes}\label{bgs}
We will eventually focus on shapes similar to the Saito-Kurokawa case (d) at the end of \cite{Art04}: 
\begin{dfn}\label{GSK}
We say a shape $\Delta = (T_i, d_i, \lb_i, \eta_i)_{1 \leq i \leq k}$ is \emph{generalized Saito-Kurokawa} or \emph{GSK} if:
\begin{itemize}
    \item $d_1 = 1$ and for all $i > 1$, $T_i = 1$,
    \item The $d_i$ are distinct integers. 
\end{itemize}
We furthermore say it is \emph{odd GSK} if:
\begin{itemize}
    \item The $d_i$ are all odd. 
\end{itemize}
\end{dfn}
\begin{note}
We define GSK to isolate the shapes that will contribute as main terms in our final application to counting representations in \S\ref{sectionunitary}. Our techniques apply slightly more generally---in particular, the condition that the $d_i$ are distinct only appears because Lemma \ref{maxoversl2} forces any dominant term to satisfy it. 

We also leave out a case of ``equiparity GSK'', where \emph{all} $T_i = 1$ and all $d_i$ are even. This is for uniformity of later theorem statements.
\end{note}
We will get exact asymptotic bounds for such $S_\Delta$ since these are the ones where the only Speh representations that appear are characters. Understanding more general shapes would require the Speh Transfer Conjecture \ref{conjspehtransfer}. 

We recall the setup of Sections \ref{rtss} and \ref{fb} and follow a similar argument: choose $G \in \wtd {\mc E}_\el(N)$ and pick finite sets of places $S = \infty \sqcup S_s \sqcup S_b$ where $S_s$ is all split and~$S_b$ contains all the ramified places. We will look at test functions of the form:
\[
f^\infty = \varphi_{S_b} f_{S_s} f^S
\]
with $\varphi_{S_b}$ and $f_{S_s}$ arbitrary and $f^S \in \ms H^\ur(G^S)$. 
We will further assume that the pair $(\Delta, \varphi_{S_b})$ satisfies the Stable Transfer Conjecture \ref{stabletransferconj} and that $\Delta$ is odd GSK.

Consider
$\psi = \tau_1[d_1] \oplus \cdots \oplus \tau_k[d_k] \in \Delta$.  
For all such $\psi$, 
Lemma \ref{splitconstant} gives Levi 
\[
M = \prod_i \GL_{T_id_i}(F_{S_s})
\]
of $G(F_{S_s}) \simeq \GL_N(F_{S_s})$ so that
\[
\tr_{\psi_{S_s}} f_{S_s} = \prodf_i \tr_{\tau_{i, S_s}[d_i]} (f_{S_s})_{M,i},
\]
and Lemma \ref{unramtrace} gives that
\[
\tr_{\psi^S} f^S = \prodf_i \tr_{\tau_i^S[d_i]} (\mc T_i f^S).
\]
Note that the constant terms and $\mc T_i f^S$ satisfy the positivity condition by various ``adjointesses'' of trace and various transfers.

Finally, 
our assumption that $\varphi_{S_b}$ satisfies Conjecture \ref{stabletransferconj} gives $\varphi_{i, S_b}$ such that
\[
\tr_{\psi_{S_b}} \varphi_{S_b} = \prodf_i \tr_{\tau_{i, S_s}[d_i]} (\varphi_{i, S_b})
\]
where the $\varphi_{i, S_b}$ are trace-positive.

This produces:
\begin{prop}\label{goodshapefactor}
 With notation and conditions as above (in particular, that the pair $(\Delta, \varphi_{S_b})$ satisfies Conjecture \ref{stabletransferconj} and $\Delta$ is odd GSK),
\[
S^H_\Delta(\EP_\lb \varphi_{S_b} f_{S_s} f^S) = 2^{-k+1} \prodf_i S^{H_i}_{(T_i, d_i, \lb_i, \eta_i)}(\EP_{\lb_i[d_i]}(f_{S_s})_{M,i} \varphi_{i, S_b} \mc T_i f^S),
\]
where $H_i = H(T_i, d_i, \lb_i, \eta_i)$. 
\end{prop}

\begin{proof}
As in Section \ref{rtss}, we multiply together and sum the trace equalities above. Note that $\Delta$ having all $d_i$ of the same parity gives that $S^H_\Delta = S^{|H|}_\Delta$ by Lemma \ref{lemma epsilon parity}.  Furthermore, since we are working with unitary groups, $m_\psi = 1$ always. Finally, for $\psi \in \Delta$ as above, since the $\tau_i[d_i]$ are simple,
\[
|\mc S_\psi|^{-1} \prod_i |\mc S_{\tau_i[d_i]}| = 2^{-k+1}.
\]
This computes all the terms in the stable multiplicty formula. 
\end{proof}

\section{Level-Aspect Asymptotics}\label{sectionSdeltaasymptotics} 

\subsection{Setup}\label{asymptoticsetup}
In this section, we use the strategy outlined in \S\ref{strategy} to compute asymptotics of $S^G_\Delta(\EP_\lb f^\infty)$ for a specific sequence of $f^\infty$. 

Fix $N$ and $G \in \wtd {\mc E}_\sm(N)$ such that $G = H(\Sigma_{\lb, \eta})$. We use the setup from \cite{ST16}. First, as some notation:
\begin{itemize}
    \item  $K := K^G_R$ is a choice of hyperspecial subgroup of $G$ at a set of unramified places $R$. Similarly define $K^{G,R}$ if $R$ contains all ramified and infinite places. 
    \item $\dim \lb$ for $\lb$ a regular, integral infinitesimal character is the dimension of the associated finite-dimensional representation on $\GL_n \C$. 
    \item $K(\mf n) := K^G_R(\mf n)$ for $\mf n$ an ideal of $\mc O_F$ supported on a set of unramified places $R$ is the principal congruence subgroup of level $\fn$:
    \[
    K^G_R(\mf n) = \prod_{v | R} K^G_v(q_v^{v(\mf n)}),
    \]
    where $K^G_v(q_v^{v(\mf n)})$ is the $\mc O_{F_v}$-points the integral model defined by hyperspecial $K^G_v$ congruent to $1 \pmod{\mf q_v^{v(\mf n)}}$. Define $K^{G,R}(\mf n)$ similarly away from a set of places~$R$ containing all ramified and infinite places.
    \item
    $G_\infty^c$ is the compact form of $G_\infty$. A choice of measure on $G_\infty$ induces one on the compact form $G_\infty^c$ in the standard way through top forms on $G_{\infty, \C}$. 
\end{itemize}

Then, fix:
\begin{itemize}
    \item A finite set of finite places $S_0$ including all those where $E/F$ is ramified,
    \item An arbitrary test function $\varphi_{S_0}$ at $S_0$,
    \item A finite set of finite places $S_1$ disjoint from $S_0$, 
    \item An $f_{S_1} \in \ms H^\ur(G_{S_1})^{\leq \kappa}$ for some $\kappa$,
    \item An ideal $\mf n$ of $F$ relatively prime to $2,3,S_0, S_1$ and that is all split in $E$. We will let $\mf n$ vary and consider asymptotics as $|\mf n| \to \infty$.
    \item $S := \infty \cup S_0 \cup S_1$. 
\end{itemize}
Then define the test function:
\[
f^\infty_{\mf n} = \varphi_{S_0} f_{S_1}\bar \1_{K^{G,S}(\mf n)}
\]
where $\bar \1_{K^{G,S}(\mf n)}$ is the indicator function normalized 
by volume 
to have integral $1$. 

We will also need some constants related to $\mf n$:
\begin{itemize}
    \item The norm 
    \[
    |\mf n| := \prod_{v|\mf n} q_v^{v(\mf n)}.
    \]
    \item Euler factors: for $n, n_i \in \Z^+$:
    \begin{gather*}
    \Gamma_n(\mf n) := \prod_{v| \mf n} (1 - q_v^{-1})(1 - q_v^{-2}) \cdots (1 - q_v^{-n}),\\
    \Gamma_{-n}(\mf n) := \prod_{v| \mf n}(1+ q_v^{-1})(1 + q_v^{-2}) \cdots (1 + q_v^{-n}),\\
    \Gamma_{\pm n_1, \dotsc, \pm n_k}(\mf n) := \Gamma_{\pm n_1}(\mf n) \cdots \Gamma_{\pm n_k}(\mf n).
    \end{gather*}
\end{itemize}

Our initial input will be the level-aspect bounds of \cite{ST16}. 
\begin{thm}[Special case of {\cite[Thm 9.16]{ST16}}]\label{stbound}
With notation defined as above, there are constants $A,B,C,D,E$ depending only on $G$ with $C \geq 1$ such that whenever $|\mf n| \geq Dq_{S_1}^{E \kappa}$,
\[
|\mf n|^{-\dim G} \Gamma_N(\mf n)^{-1} I^G(\EP_\lb f^\infty_{\mf n}) = \Lambda + O(|\mf n|^{-C} q_{S_1}^{A + B \kappa}),
\]
where we define the mass:
\[
\Lambda = \Lambda(G, f, \varphi) = \varphi_{S_0}(1) f_{S_1}(1) \f{\dim \lb}{|\Pi_\disc(\lb)|} \f{\vol(G(F) \bs G(\A_F))}{ \vol(K^S) \vol(G_\infty^c)}. 
\]
\end{thm}

\begin{proof}
Since only split primes divide $\mf n$, we have
\[
[K : K(\mf n)] = |\mf n|^{\dim G} \Gamma_N(\mf n)
\]
by a standard formula for the sizes of the $\GL_n(\mc O_v/ q_v^{v(\mf n)} \mc O_v)$. 
\end{proof}

When $S_0 = S_1 = \emptyset$, the quotient of volumes is the modified Tamagawa number computed in \cite[Cor 6.14]{ST16}: 
\begin{equation}\label{tamagawa}
\f{\vol(G(F) \bs G(\A_F))}{ \vol(K^\infty) \vol(G_\infty^c)} = \tau'(G) = 2^{-(N-1) \deg F} \tau(G) L(\mathrm{Mot}_G) |\Om_G||\Om_{G_\infty}^c|^{-1}
\end{equation}
where $\tau(G)$ is the Tamagawa number, $L(\mathrm{Mot}_G)$ is the $L$-value of the motive from \cite{Gro97}, and $\Om_{G_\infty}^c$ is the Weyl group of the maximal compact at $\infty$. 

Since we are not tracking explicit values for $A, B, C, D,E$, we will allow them to change throughout the following argument.

\subsection{Bounds on Stable Trace}
To extend \ref{stbound} to $S^G$, we next recall a standard formula:
\begin{lem}\label{congruenceconstantterm}
Let $G$ be an unramified reductive group over $F$, $\mf n_i$ an ideal relatively prime to all places where $G$ is ramified, and $M$ a Levi component of parabolic subgroup $P$. Then we have identity of indicator functions normalized by volume:
\[
(\bar \1_{K^G(\mf n)})_M = I(\mf n) \bar \1_{K^M(\mf n)},
\]
where we recall that $(\star)_M$ denotes the constant term and 
\[
I(\mf n) = [K : K \cap K(\mf n)P]. 
\]
Furthermore, if $G = \GL_N$ then
\[
I(\mf n_i) = (1 + O(|\mf n|^{-2})) |\mf n|^{\dim G/P} \Gamma_{-1}(\mf n)^{\sigma(M)},
\]
where
\[
\sigma(M) = \rank_\ssm G - \rank_\ssm M.
\]
\end{lem}

\begin{proof}
This is well known---see for example the proof of Lemma 5.2 in \cite{MS19}.  
\end{proof}
Note for intuition later that $\dim G/P = 1/2(\dim G - \dim M)$. Using this:
\begin{prop}\label{sgbound}
Let $f_{\mf n}^\infty$ and $\Lambda$ be as in Section \ref{asymptoticsetup}. Then there are constants $A,B,C,D,E$  depending only on $G$ with $C \geq 1$ such that whenever $|\mf n| \geq Dq_{S_1}^{E \kappa}$,
\[
|\mf n|^{-\dim G} \Gamma_N(\mf n)^{-1} S^G(\EP_\lb f^\infty_{\mf n}) = \Lambda + O(|\mf n|^{-C}q_{S_1}^{A + B \kappa}).
\]
\end{prop}

\begin{proof}
We apply Theorem \ref{hyperendoscopy} and apply Theorem \ref{stbound} to each term. This 
is a 
much simpler version of the main result of \cite{Dal22} so we present it tersely. 

We need only show that all the non-$I^G$ summands can be put into the error term:
\begin{itemize}
\item In this case of unitary groups, there are a finite number of such terms. 
\item We can ignore dependence on $\varphi_{S_0}$. 
\item Lemma 5.5.4 in \cite{Dal22} allows us to bound the value at $1$ and the support of the $f_{S_1}^\mc H$. 
\item Lemma \ref{splitconstant} lets us iteratively use Lemma \ref{congruenceconstantterm} to bound the $(\bar \1_{K^{G,S}(\mf n_i)})^\mc H$.
\item
Averaging Corollary 5.1.6 in \cite{Dal22} shows that the transfer of the EP-function is a linear combination of a number of EP-functions uniformly bounded over $\mc H$. 
\end{itemize}
Putting all this together, \ref{stbound} shows that the sum of the non-$I^G$ terms is bounded above by the claimed error as long as we extremize $A,B,C,D,E$ appropriately over all hyperendoscopic groups.
\end{proof}

\subsection{The Induction}
Now we induct to bound the limit multiplicities restricted to specific shapes. Their are two pieces to this argument---first, as a consequence of Proposition \ref{SMform}: 
\begin{equation}\label{SGshapeexpansion}
S^G_{\Sigma_{\lb,\eta}}(\EP_\lb f^\infty_{\mf n}) = S^G(\EP_\lb f^\infty_{\mf n}) - \sum_{\substack{H(\Delta) = H(\Sigma_{\lb, \eta})\\ \text{inf. char}(\Delta) = \lb \\\Delta \neq \Sigma_{\lb, \eta}}} S^G_\Delta(\EP_\lb f^\infty_{\mf n}).
\end{equation}
Second, the bound in Proposition \ref{shapeineq} lets us show that the non-$\Sigma$ terms $S^G_\Delta$ have limit multiplicities controlled by the groups $M$ from which they are lifts.

We start with a technical trick that gives us the trace-positivity conditions needed to apply the results of the previous section:

\begin{lem}\label{tracepositivelinearcombination}
Let $f_v$ be a test function on $G_v$. Then there are trace-positive functions $f_1, \dotsc, f_k$ such that
\[
f_v = \lb_1 f_1 + \cdots + \lb_k f_k.
\]
Furthermore, if $f_v \in \ms H^\ur(G_v)^{\leq \kappa}$, then so are the $f_i$ and 
\[
\sum_i |\lb_i| \leq C, \qquad \sum_i \|f_i\|_\infty \leq C\|f_v\|_\infty
\]
for some uniform constant $C$. 
\end{lem}

\begin{proof}
The first statement holds in great generality by the Dixmier-Malliavin decomposition theorem as used \cite[Lem 3.5]{Sau97}. We recall some details: 

We can without loss of generality assume $f_v^* = f_v$ by symmetrizing:
\[
f_v = \lf(\f{f_v + f_v^*}2\ri) - i \lf(\f{f_v - f_v^*}{2i}\ri).
\]
By smoothness, there is a open compact subgroup $U$ such that $f_v = f_v \star \bar \1_U = \bar \1_U \star f_v$. Then
\[
f_v = \f14 (f_v + \bar \1_U) \star (f_v + \bar \1_U) - \f14 (f_v - \bar \1_U) \star (f_v - \bar \1_U),
\]
a linear combination of functions of the form $g \star g^*$, each necessarily trace-positive. 

Using $U = K_v$ gives the required bounds for the second statement. 
\end{proof}

We can now state the first important technical bound resulting from our methods:

\begin{thm}\label{simpleshapesmain}
Fix $G \in \wtd {\mc E}_\el(N)$, $f_{\fn}^\infty$, and $\Lambda$ as in Section \ref{asymptoticsetup}. 
Then there are constants $A,B,C,D,E$ depending only on $G$ with $C \geq 1$ such that whenever $|\mf n| \geq Dq_{S_1}^{E \kappa}$, we have
\begin{equation}\label{equation sigma shape bound}
|\mf n|^{-\dim G} \Gamma_N(\mf n)^{-1} S^G_{\Sigma_{\lb, \eta}}(\EP_\lb f^\infty_{\fn}) = \Lambda + O(|\mf n|^{-C}q_{S_1}^{A + B \kappa}).
\end{equation}
For $\Delta \neq \Sigma_{\lb,\eta}$, the same constants give the upper bound
\begin{equation} \label{equation non sigma shape bound}
 S^G_\Delta(\EP_\lb f^\infty_{\fn}) = O(|\mf n|^{\bar R(\Delta)}q_{S_1}^{A + B \kappa}),
\end{equation}
where the exponent $\bar R(\Delta)$ is given by
\begin{equation} \label{Equation R(Delta)=bar}
\bar R((T_i, d_i, \lb_i, \eta_i)_{1 \leq i \leq k}) = \f12 \lf( N^2 + \sum_i T_i^2 d_i \ri).
\end{equation}
\end{thm}

\begin{proof}
We induct on the $N$ such that $G \in \wtd {\mc E}_\el(N)$. For $N=1$, $\Sigma_{\lb, \eta}$ is the only possible shape so (by extreme overkill) this follows from Theorem \ref{stbound}. 

For the inductive step, we first argue that it suffices to show \eqref{equation non sigma shape bound}: we start from \eqref{SGshapeexpansion} and take $D$ and $E$ to be the maximum values over those for all smaller shapes $\Delta \neq \Sigma_{\lambda, \eta}$ appearing in the expansion. Then~\eqref{equation non sigma shape bound} will show that all the corresponding terms 
are lower order in $|\mf n|$ than the $S^G$-asymptotic from Proposition \ref{sgbound}. This produces the exact asymptotic \eqref{equation sigma shape bound}.

Therefore, let $\Delta = (T_i, d_i, \lb_i, \eta_i)_{1 \leq i \leq k}  \neq \Sigma_{\lb, \eta}$ have rank $N$. First, we apply Lemma \ref{tracepositivelinearcombination} on each of the finitely many unramified factors of $\varphi_{i, S_0}$ and $\mc T_i f_{S_i}$ to write
\[
S^G_\Delta(\EP_\lb f^\infty_{\fn}) = \sum_j \lb_j S^G_\Delta(\EP_\lb f^\infty_{\fn,j})
\]
where the sum of the $|\lb_j|$ is $O(C^{|S_1|}) = O(q_{S_1})$ and each $f^\infty_{\fn,j}$ is trace-positive. 

We then apply Proposition \ref{shapeineq} with $S_s$ the primes that divide $\mf n$ and $S_b = S_0$. Over $S_s$, let $M = \prod_i \GL_{T_i}(F_{S_s})^{d_i}$ be the Levi from \ref{shapeineq} and $P$ the parabolic defining the constant term. Then (using notation from \ref{shapeineq} to decompose $f_{\mf n,j}^\infty$),
\begin{align}\label{simpleshapesproofeq}
|S^G_\Delta(\EP_\lb f^\infty_{\fn,j})| &\leq C_\Delta \prodf \lf(S^{T_i}_{\Sigma_{\lb_i, \eta_i}}(((f_{S_s})_{M,i} \varphi'_{i, S_0} \mc T_i f_{S_1} \1_{K^S})^{T_i })\ri)^{d_i \oplus} \\
&= C_\Delta I(\mf n) \prodf \lf(S^{H_i}_{\Sigma_{\lb_i, \eta_i}}(\EP_{\lb_i} \bar \1_{K^{H_i}(\mf n)} \varphi'_{i, S_0} \mc T_i f_{S_1})\ri)^{d_i \oplus}, \nonumber
\end{align}
where each $H_i = H(\Sigma^{T_i}_{\lb_i, \eta_i})$ and the second equality uses Lemma \ref{congruenceconstantterm} together with factoring $\vol(K^M(\mf n))$ over places in $S_s$.

From Lemma \ref{unramtransferbound}, we know that $\mc T f_{S_1} \in \ms H^\ur(M_{S_1})^{\leq \kappa}$ and 
\[
\|\mc T f_{S_1} \|_\infty = \|f_{S_1}\|_\infty O(q_v^{\kappa F} \kappa^G).
\]
Therefore, we can use the inductive hypothesis on the $\Sigma_{\lb_i, \eta_i}$, summing the asymptotics of the $\Lambda$ and the error to get that each term in the product  \eqref{simpleshapesproofeq} is 
\[
O \lf( |\mf n|^{\dim H_i} \Gamma_{T_i^{(d_i)}}(\mf n) q_{S_1}^{A' + B'\kappa} \ri)
\]
 for some constants $A', B'$ depending on $\Delta$. Note that the positivity condition is guaranteed by the last implication of Proposition \ref{shapeineq}. Summing over $j$, this gives
\[
|S^G_\Delta(\EP_\lb f^\infty_{\mf n})| = O \lf(  \Gamma_{T_1^{(d_1)}, \dotsc, T_k^{(d_k)}}(\mf n) I(\mf n) |\mf n|^{\dim M} q_{S_1}^{A' + B'\kappa} \ri),
\]
ignoring all dependence on the $\varphi'$.

Recalling that $G_{S_s} \simeq \GL_{N,S_s}$, we can use the asymptotic for $I(\mf n)$ from \ref{congruenceconstantterm}:
\[
I(\mf n) |\mf n|^{\dim M} = O\lf(|\mf n|^{\dim P} \Gamma_{-1}(\mf n)^{\sigma(M)} \ri).
\]
We finally note that $\Gamma_{T_1^{(d_1)}, \dotsc, T_k^{(d_k)}}(\mf n)\Gamma_{-1}(\mf n)^{\sigma(M)} \leq 1$ and that
\[
\dim P = \bar R(\Delta) = \f12 \lf(N^2 + \sum_i T_i^2 d_i \ri).
\] 
This finishes the bound for $\Delta$ and therefore the induction. 
\end{proof}

We emphasize that the value 
$\bar R(\Delta)$, which is the dimension of the parabolic attached to the partition
\[
(T_1^{(d_1)}, \dotsc, T_k^{(d_k)}),
\]
 is an approximate upper bound to what we expect to be 
the true growth rate attached to $\Delta$. 
It is 
 maximized at $\bar R(\Sigma_{\lb, \eta}) = \dim G$ where it is exact.

\subsection{Improving the Bound}
We can prove a slightly tighter bound $R(\Delta)$. Instead of applying \ref{shapeineq} directly as in the proof of \ref{simpleshapesmain}, we will apply \ref{constanttermineq} and then use a different method to bound the factors for simple blocks $(T_i, d_i, \lb_i, \eta_i)$ with small $T_i$. For large $T_i$, we will proceed as before using the arguments of \S\ref{fb}. This argument is nothing more than a rephrasing in our context of the key technical trick that makes the bounds in \cite{MS19} work.

\subsubsection{Terms with $T_i = 1$}
In this case we actually prove a stronger exact formula:
\begin{lem}\label{T1bound}
When $T_i = 1$, the terms for the factors $S_{(T_i, d_i, \lb_i, \eta_i)}$ coming from the use of Corollary \ref{constanttermineq} implicit in \eqref{simpleshapesproofeq} satisfy the bound:
\begin{align*}
S^{H'_i}_{(1, d_i, \lb_i, \eta_i)} (&EP_{\lb_i[d_i]} \bar \1_{K^{H'_i}(\mf n)} \varphi_{i, S_0} \mc T_i f_{S_1}) \\ 
& = |\mf n|\Gamma_1(\mf n_i)  \f{\vol(H'_i(F)^\ab \bs H'_i(\A)^\ab)}{\vol(K^{(H'_i)^\ab,S}) \vol((H'_{i,\infty})^\ab)} \int_{H'_{i,\der, S_1,S_0}}  \mc T_i f_{S_1} \varphi_{i, S_0}(h) \,dh \\
& = O(|\mf n| \Gamma_1(\mf n) q_{S_1}^{A + B \kappa})
\end{align*}
as long as $|\mf n| > D q_{S_1}^{E \kappa}$. 
\end{lem}

\begin{proof}
By Corollary \ref{charshapeformula},
\begin{multline*}
S^{H'_i}_{(1, d_i, \lb_i, \eta_i)} (EP_{\lb_i[d_i]} \bar \1_{K^{H'_i}(\mf n)} \varphi_{i, S_0} \mc T_i f_{S_1}) \\
=  \f{\vol(H'_i(F)^\ab \bs H'_i(\A)^\ab)}{\vol((H'_{i,\infty})^\ab)} \sum_{\gamma \in H'_i(F)^\ab} (\bar \1_{K^{H'_i}(\mf n)})^\ab(\gamma) (\mc T_i f_{S_1})^\ab(\gamma) \varphi_{i, S_0}^\ab(\gamma) \xi^{-1}_i(\gamma)
\end{multline*}
where $\xi_i$ is the character of $H'_i$ associated to infinitesimal character $\lb_i[d_i]$.

At places dividing $\mf n_i$ we know that $H_i$ is a general linear group, so we compute 
\[
(\bar \1_{K^{H'_i}(\mf n)})^\ab = |\mf n| \Gamma_1(\mf n) \bar \1_{K^{(H'_i)^\ab}(\mf n)}
\]
using standard formulas for $|\SL_n(\mc O_F/\mf p_v^n)|$.
Next,  $H'_i(F)^\ab \subseteq (H')_i^\ab(F)$ so a trivial case of Lemma 8.4 of \cite{ST16} applied to $(H'_i)^\ab$ gives $D_i$ and $E_i$ such that whenever $|\mf n| \geq D_iq_{S_1}^{E_i \kappa}$, all terms in the sum vanish except for $\gamma = 1$. This finishes the argument.
\end{proof}

\subsubsection{Terms with $T_i = 2$}
\begin{lem}\label{T2bound}
When $T_i = 2$, the terms for factors $S_{(T_i, d_i, \lb_i, \eta_i)}$ coming from the use of Corollary \ref{constanttermineq} implicit in \eqref{simpleshapesproofeq} satisfy the bound:
\begin{multline*}
S^{H'_i}_{(2, d_i, \lb_i, \eta_i)} (EP_{\lb_i[d_i]} \bar \1_{K^{H'_i}(\mf n)} \varphi_{i, S_0} \mc T_i f_{S_1}) \\
= O(|\mf n|^{2d_i(d_i - 1) + d_i+3} \Gamma_{-1^{(2d_i-2)}, 2}(\mf n) q_{S_1}^{A + B \kappa}).
\end{multline*}
\end{lem}

\begin{proof}
We use Proposition \ref{[d]boundvar} instead of Proposition \ref{[d]bound}. Bound~$\tr_{\pi_{S_0}}(\varphi_{i,S_0})$ by a constant through Bernstein admissibility as in Lemma \ref{gen[d]}. Next, the Satake eigenvalues of unirreps of $\GL_N$ always have their $|\log_{q_v}(\cdot)|$ bounded by $(N-1)/2$ by the main result of \cite{Tad86} (this is the value achieved by the trivial representation). Therefore, arguments as in \cite[\S5.5]{Dal22} show that 
\[
\tr_{\pi_{S_1}}(\mc T_i f_{S_1}) = O(q_{S_1}^{A + B\kappa})
\]
for all unirreps $\pi_{S_1}$ of $\GL_2$ at $S_1$ (the number of factorizable summands of $\mc T f_{S_1}$ is bounded similarly). 

Finally, for any irrep $\pi^S$ we have, as in the proof of Lemma 5.2 in \cite{MS19}, 
\[
\tr_{\pi^S} \bar \1_{K^{\GL_2}(\mf n)} \leq \Gamma_{-1}(\mf n) |\mf n|.
\]
Applying \ref{[d]boundvar} and all the above bounds (and changing $A,B)$:
\begin{multline*}
S^{H'_i}_{(2, d_i, \lb_i, \eta_i)} (EP_{\lb_i[d_i]} \bar \1_{K^{H'_i}(\mf n)} \varphi_{i, S_0} \mc T_i f_{S_1}) \\
= O(|\mf n|^{d-1} \Gamma_{-1}(\mf n)^{d_i-1}  q_{S_1}^{A+B\kappa})\\
\times |\mf n|^{2d_i(d_i-1)} \Gamma_{-1}(\mf n)^{d_i -1} S^{H_i}_{\Sigma_{\lb_i, \eta_i}} (EP_{\lb_i} \bar \1_{K^{H_i}(\mf n)} \varphi_{i, S_0} \mc T_i f_{S_1}).
\end{multline*}
Where the factors on the second line outside the big-$O$ come from taking the constant term of $\bar \1_{K^H(\mf n_i)}$. Substituting in Theorem \ref{simpleshapesmain} produces the result noting that $H_i = H(2,1, \lb_i, \eta_i) \in \mc E_\el(2)$. 
\end{proof}

\subsubsection{Terms with $T_i = 3$}
\begin{lem}\label{T3bound}
When $T_i = 3$, the terms for factors $S_{(T_i, d_i, \lb_i, \eta_i)}$ coming from the use of Corollary \ref{constanttermineq} implicit in \eqref{simpleshapesproofeq} satisfy that for all $\eps>0$:
\begin{multline*}
S^{H'_i}_{(3, d_i, \lb_i, \eta_i)} (EP_{\lb_i[d_i]} \bar \1_{K^{H'_i}(\mf n)} \varphi_{i, S_0} \mc T_i f_{S_1})\\
= O_\eps(|\mf n|^{\f92d_i(d_i - 1) + (4 + \eps)d_i +5} \Gamma_{-1^{(d)}, 3}(\mf n_i) q_{S_1}^{A + B \kappa}).
\end{multline*}
\end{lem}

\begin{proof}
This is the same argument as Lemma \ref{T2bound} except that we use 
\[
\tr_{\pi^S} \bar \1_{K^{\GL_3}(\mf n)} \leq C(\eps) |\mf n|^{4+\eps}
\]
from corollary 9.2 in \cite{MS19}. Our $C(\eps)$ here is the product of Marshall-Shin's $C(\eps, q_v)$ for $q_v \leq q(\eps)$. 
\end{proof}

\subsubsection{The full bound}
Applying the previous results with $T_i = 1,2,3$ instead of directly applying Proposition \ref{shapeineq}: 
\begin{cor}\label{simpleshapesstronger}
The bound for 
\[
\Delta = (T_i, d_i, \lb_i, \eta_i)_{1 \leq i \leq k} \neq \Sigma_{\lb, \eta}
\]
in Theorem \ref{simpleshapesmain} may be tightened to
\[
S^G_\Delta(\EP_\lb f^\infty_{\fn}) = O(|\mf n|^{R(\Delta)}q_{S_1}^{A + B \kappa})
\]
under all the same conditions and where
\begin{multline} \label{Equation R(Delta)}
 R(\Delta)  = \bar R(\Delta) - \sum_{i : T_i = 1} \lf(\f12 (d_i^2 + d_i) - 1 \ri) - \sum_{i : T_i = 2} \lf(4 d_i - (d_i +3) \ri) \\ - \sum_{\substack{i : T_i = 3 \\ d_i > 1}} \lf(9 d_i - ((4 + 10^{-100})d_i +5) \ri).
\end{multline}
\end{cor}

\begin{proof}
For each summand of $\Delta$ with $T_i = 1$, let $H'_i = H((1, d_i, \lb_i, \eta_i))$. Then the method of proof of Theorem \ref{simpleshapesmain} implicitly bounds terms
\[
S^{H'_i}_{(1, d_i, \lb_i, \eta_i)}(\EP_{\lb_i[d_i]} \bar \1_{K^{H'_i}(\mf n)}  \varphi'_{i, S_0} \mc T_i f_{S_1}) = O(|\mf n|^{\f12 (d_i^2 + d_i)} \Gamma_{-1^{(d_i -1)}, 1^{(d_i)}}(\mf n) q_{S_1}^{A + B \kappa}).
\]
We instead use Lemma \ref{T1bound} to replace them by 
\[
S^{H'_i}_{(1, d_i, \lb_i, \eta_i)} (EP_{\lb_i[d_i]}f^\infty) = O(|\mf n| \Gamma_1(\mf n) q_{S_1}^{A + B \kappa}).
\]

For summands of $\Delta$ with $T_i = 2$, we similarly use Lemma \ref{T2bound} to replace
\begin{multline*}
O(|\mf n|^{2 (d_i^2 + d_i)} \Gamma_{-1^{(d_i -1)}, 2^{(d_i)}}(\mf n) q_{S_1}^{A + B \kappa}) \\
\mapsto O(|\mf n|^{2d_i(d_i - 1) + d_i+3} \Gamma_{-1^{(2d_i-2)}, 2}(\mf n) q_{S_1}^{A + B \kappa}).
\end{multline*}
For $T_i = 3$, we use \ref{T3bound} to replace
\begin{multline*}
O(|\mf n|^{\f92 (d_i^2 + d_i)} \Gamma_{-1^{(d_i -1)}, 3^{(d_i)}}(\mf n) q_{S_1}^{A + B \kappa}) \\
\mapsto O_\eps(|\mf n|^{\f92d_i(d_i - 1) + (4 + 10^{-100})d_i +5} \Gamma_{-1^{(d)}, 3}(\mf n) q_{S_1}^{A + B \kappa}).
\end{multline*}
Substituting in these stronger bounds produces the result. \footnote{Here, we were not very careful with the $\varphi_{S_0}$ terms since we are not making claims about how the error term depends on them.}  
\end{proof}

\begin{rmk} \label{Remark RQ}
    The $R(\Delta)$ is a better upper bound of the true growth rate than $\bar R(\Delta)$. It is obtained by
    making three modifications to the dimension count of the parabolic:
\begin{itemize}
    \item When $T_i = 1$, replace the dimension of the Borel in the $\GL_{d_i}$-block corresponding to that summand with $1$,
    \item When $T_i = 2$, replace the dimension of the Levi $\GL_2^{d_i}$ in the $\GL_{2d_i}$-block corresponding to that summand with $d_i + 3$,
    \item When $T_i = 3$, replace the dimension of the Levi $\GL_3^{d_i}$ in the $\GL_{3d_i}$-block corresponding to that summand with $(4 + \eps)d_i + 5$.
\end{itemize}
\end{rmk}

The next section describes a case where $R(\Delta)$ is the optimal growth rate.

\subsection{Odd GSK Shapes}
Now that we understand $S_{\Sigma_{\lb, \eta}}$, we can use Proposition \ref{goodshapefactor} to compute the limiting asymptotics for odd Generalized Saito-Kurokawa (GSK) shapes (recall: these are shapes $\Delta = (\Delta_i)_i= (T_i, d_i, \lb_i, \eta_i)_{1 \leq i \leq k}$ such that $d_1=1$, $T_i=1$ for $i\geq 2$, and the $d_i$ are odd and distinct).  
Keep the same setup as Section \ref{asymptoticsetup}. Additionally, 
let $\Delta$ as above be odd GSK with $H(\Delta) = G$ and let~$H_i = H(\Delta_i)$. 

We will now prove our second key technical bound, namely that $R(\Delta)$ is the exact growth rate for $S^G_\Delta$. 

\begin{note}
As stated before, the distinctness part of definition \ref{GSK} of odd-GSK will not be necessary for Theorem \ref{goodshapebound}, only for our application in Section \ref{sectionunitary}. The non-GSK shapes to which \ref{goodshapebound} applies 
will never be the dominant contribution to counts of forms.
\end{note}

\begin{thm}\label{goodshapebound}
Fix $f_{\mf n}^\infty$ and $G \in \wtd {\mc E}_\el(N)$ as in Section \ref{asymptoticsetup}. Assume $\Delta$ is odd GSK as in Definition \ref{GSK} and that the pair $(\Delta, \varphi_{S_0})$ satisfies Conjecture \ref{stabletransferconj}. 

Then there are constants $A,B,C,D,E$ with $C \geq 1$ depending only on $G$ and $\Delta$ such that whenever $|\mf n| \geq Dq_{S_1}^{E \kappa}$,
\begin{multline*}
|\mf n|^{-R(\Delta)} \Gamma_{L(\Delta)}(\mf n)^{-1} S^G_\Delta(\EP_{\lb} f_{\mf n}^\infty) \\= 2^{-k+1} \Lambda(H_1, \mc T_1 f_{S_1}, \varphi_{1, S_0}) \times \prodf_{i \geq 2} \Lambda^\ab(H_i, \mc T_i f_{S_1}, \varphi_{i, S_0}) + O(|\mf n|^{-C} q_{S_1}^{A + B \kappa}),
\end{multline*}
with growth rate from Corollary \ref{simpleshapesstronger} (which simplifies in this case)
\[
R(\Delta) = \f12 \dim H_1 + (k-1) + \f 12 \lf(\dim G - \sum_{i \geq 2} \dim H_i, \ri),
\]
indexing list
\[
L(\Delta) = T_1, 1^{(k-1)}, -1^{(k-1)}, 
\] 
and masses
\begin{align*}
\Lambda(H_1, \mc T_1 f_{S_1}, \varphi_{1, S_0}) &= \varphi_{1, S_0}(1) (\mc T_1 f_{S_1})(1) \f{\dim \lb_1}{|\Pi_\disc(\lb_1)|} \f{\vol(H_1(F) \bs H_1(\A_F))}{ \vol(K_{H_1}^S) \vol(H^c_{1, \infty})},\\
\Lambda^\ab(H_i, \mc T_i f_{S_1}, \varphi_{i, S_0}) &= \f{\vol(H_i(F)^\ab \bs H_i(\A)^\ab)}{\vol(K_{H_i^\ab}^S)\vol((H_{i, \infty})^\ab)} \int_{H_{i, \der, S_0, S_1}} \mc T_i f_{S_1} \varphi_{i, S_0}(h) \,dh.
\end{align*}
(Recall the $\mc T_i f_{S_1}$  and $\varphi_{i, S_0}$ are defined as in Lemma \ref{unramtrace} and Conjecture \ref{stabletransferconj}.)
\end{thm}

\begin{proof}
We first apply Proposition \ref{goodshapefactor} with $S_b = S_0$ and $S_s$ the places dividing~$\mf n$: whenever $|\mf n|$ is big enough,
\[
S^G_\Delta(\EP_\lb f_{\mf n}^\infty)
= 2^{-k+1} I(\mf n) \prodf_i S^{H_i}_{(T_i, d_i, \lb_i, \eta_i)} (EP_{\lb_i[d_i]} \bar \1_{K^{H_i}(\mf n)} \varphi_{i, S_0} \mc T_i f_{S_1}),
\]
where we applied Lemma \ref{congruenceconstantterm} to compute constant terms. 

For $i = 1$, $d_i = 1$ so from Theorem \ref{simpleshapesmain} we get (for each summand in the $\prodf$):
\begin{multline*}
S^{H_1}_{(T_1, 1, \lb_1, \eta_1)} (EP_{\lb_1[d_1]} \bar \1_{K^{H_1}(\mf n)} \varphi_{1, S_0} \mc T_1 f_{S_1}) \\ = |\mf n|^{\dim H_1} \Gamma_{T_1}(\mf n) \Lambda(H_1, \mc T_1 f_{S_1}, \varphi_{1, S_0}) + O(|\mf n|^{\dim H_1 - C} q_{S_1}^{A + B \kappa}).
\end{multline*}
For $i > 1$, $T_i = 1$, so we apply Lemma \ref{T1bound} and get:
\begin{multline*}
S^{H_i}_{(1, d_i, \lb_i, \eta_i)} (EP_{\lb_i[d_i]}f^\infty_{\mf n}) \\ = |\mf n|\Gamma_1(\mf n)  \f{\vol(H_i(F)^\ab \bs H_i(\A)^\ab)}{\vol(K^{H_i^\ab,S})\vol((H_{i, \infty})^\ab)} \int_{H_{i,\der, S_1,S_0}}  \mc T_i f_{S_1} \varphi_{i, S_0}(h) \,dh.
\end{multline*}
The result follows from multiplying and summing over factorizable summands in the $\prodf$ and taking the maximum over the various $A$'s through $E$'s above. We use the second part of Lemma \ref{congruenceconstantterm} to estimate $I(\mf n_i)$ and note that the  $\dim G/P$ there is $1/2(\dim G - \dim M)$.
\end{proof}

\begin{note}
We write the scaling factor in the theorem statement as it is to emphasize the exact growth in $\mf n$. It comes from the more conceptual formula:
\begin{equation*}
|\mf n|^{-R(\Delta)}  \Gamma_{L(\Delta)}(\mf n)= \lf( |\mf n|^{\f12(\dim G - \sum_i \dim H_i)}\Gamma_{-1}(\mf n)^{k-1}\ri) \lf( \prod_i [K^{\GL_{T_i}} : K^{\GL_{T_i}}(\mf n)] \ri)
\end{equation*}
where the first factor comes from the parabolic descent of the function $\bar \1_{K(\mf n)}$ and the second from the expected growth rates of counts on the groups $H_i$.
\end{note}

We give two examples with $S_0 = \emptyset$, removing dependence on Conjecture \ref{stabletransferconj}. 

\begin{ex}
Consider the case when $S_1 = S_0 = \emptyset$. Then the theorem reduces to 
\begin{multline}
|\mf n|^{-R(\Delta)}  \Gamma_{L(\Delta)}(\mf n)S^G_\Delta(\EP_\lb f_{\mf n}^\infty) \\
= 2^{-k+1} \f{\dim \lb_1}{|\Pi_\disc(\lb_1)|} \f{\vol(H(F) \bs H(\A_F))}{ \vol(K_{H}^\infty) \vol(H^c_\infty)} + O(|\mf n|^{-C})
\end{multline}
where 
\[
H = H_1 \times \prod_{i > 1} H_i^\ab.
\]
\end{ex}

\begin{ex}
Consider the case when $S_0 = \emptyset$ and $S_1$ is a singleton $\{v\}$. Let $\mu^{\pl, \ur}(H_v)$ be the Plancherel measure on the unramified spectrum of $\wh H_v$ where
\[
H = H_1 \times \prod_{i > 1} H_i^\ab
\]
as before. Then $H$ is dual to the group  associated to $\Delta$ in \eqref{shapeLembedding} and comes with a map on the space of Satake parameters as in \eqref{satakepushforward}:
\[
\mc S_\Delta : \wh H^{\ur, \temp}_v \to \wh G^\ur_v.
\]

Define the pushforward
\[
\mu^{\pl(\Delta), \ur}_v := \mu^{\pl(\Delta), \ur}(G_v) := (\mc S_\Delta)_*(\mu^{\pl, \ur}(H_v)). 
\]
By \eqref{Tcharactershape}, the factors related to $f_{S_1}$ in the $\Lambda$'s from Theorem \ref{goodshapebound} multiply to
\[
(\mc T_\Delta f_v)(1) = \mu^\Delta_v(\wh f_v) 
\]
by Fourier inversion. Substituting this into \ref{goodshapebound}, we get
\begin{multline}\label{DeltaPlancherelEquidistribution}
|\mf n|^{-R(\Delta)}  \Gamma_{L(\Delta)}(\mf n)S^G_\Delta(\EP_\lb f_{\mf n}^\infty) \\
= \lf(2^{-k+1}  \f{\dim \lb_1}{|\Pi_\disc(\lb_1)|} \f{\vol(H(F) \bs H(\A_F))}{ \vol(K_{H}^\infty) \vol(H^c_\infty)} \ri) \mu^{\pl(\Delta), \ur}_v(\wh f_v) + O(|\mf n|^{-C} q_{S_1}^{A + B \kappa}).
\end{multline}
This will be used to interpret main Theorem \ref{mainexact} as Plancherel equidistribution. 
\end{ex}

\subsection{General Shapes: Conjectural Optimal Bound}\label{sectionconjecturalbound}
Considerations of the notion of 
GK-dimension give us a heuristic for an optimal growth rate 
for any $\Delta$. 

For $\pi_v$ a representation of a $p$-adic group $G_v$, the Harish-Chandra-Howe local character expansion gives an expression
\[
\Theta_{\pi_v}(\exp g) = \sum_{O \in N} c_O(\pi) \wh \mu_O(g)
\]
where $N$ is the set of nilpotent orbits of $G_v$ acting on $\Lie G$, $\mu_O(G)$ is the Fourier transform of the $\delta$-measure on the orbit $O$, $c_O(\pi)$ are constants, and $g \in \Lie G$ is in a small enough open compact at the identity. 

\begin{dfn}
With the notation as above, let the GK-dimension of $\pi_v$ 
\[
d_{GK}(\pi_v) := \f12 \max\{\dim O : c_O(\pi) \neq 0\}.
\]
\end{dfn}
We can then compute:
\begin{lem}
Assume $G_v$ is unramified. Then
\[
\dim\lf(\pi_v^{K^G(q_v^n)}\ri) = \tr_{\pi_v} \bar \1_{K^G(q_v^n)} \asymp q_v^{n d_{GK}(n)}.
\]
\end{lem}

M{\oe}glin and Walspurger in \cite{MW87} associate to each $O \in N$ a particular ``degenerate Whittaker model'' $W_O$. They prove that the maximal $O$ such that $c_O(\pi) \neq 0$ are exactly the same as the maximal $O$ such that $\Hom(\pi_v, W_O) \neq 0$.

Now specialize to $G_v = \GL_n(F_v)$. For $\pi_v$ a tempered representation of $\GL_t(F_v)$ define $\pi_v[d]$ to be the Langlands quotient of the parabolic induction
\[
\Ind_{P}^{G_v} (\pi_v |\det|^{(d-1)/2} \boxtimes \pi_v |\det|^{(d-3)/2} \boxtimes \cdots \boxtimes \pi_v |\det|^{-(d-1)/2}).
\]
This is the local component of the Speh representations $\pi_{\tau[d]}$ in \S\ref{ACparamtorep}. 

Any tempered $\pi_v$ on $\GL_t(F_v)$ is generic and therefore satisfies 
\[
GK(\pi_v) = \f12 t(t-1).
\]
On the other hand, \cite{MW87} compute (see \cite{Mit20} for a summary) that the maximal~$O \in N$ such that $\Hom(\pi_v[d], W_O) \neq 0$ is the principal nilpotent orbit corresponding to the partition $(t^{(d)})$ through Jordan normal form. We can therefore compute
\[
GK(\pi_v[d]) = \f12(t^2d^2 - td^2) = \f12 d^2t(t-1)
\]
and get
\begin{equation}\label{bestfixedvectorbound}
\dim\lf(\pi_v[d]^{K(q_v^n)}\ri) \asymp  q_v^{ \f 12 t(t-1)(d^2-1)n}\dim\lf(\pi_v^{K(q_v^n)}\ri).
\end{equation}

By the Ramanujan conjecture, for any simple parameter $\psi[d]$, we expect all the~$\psi_v$ to correspond to tempered representations on the $\GL$ side (this is in fact known for our case, see lemma 6.1 in \cite{MS19}). Therefore we can use the heuristic \eqref{bestfixedvectorbound} instead of Lemmas \ref{T1bound}, \ref{T2bound},and \ref{T3bound} in Theorem \ref{simpleshapesstronger} and get:
\begin{conj}\label{simpleshapesconj}
The bound for any
\[
\Delta = (T_i, d_i, \lb_i, \eta_i)_{1 \leq i \leq k} \neq \Sigma_{\lb, \eta}
\]
in Theorem \ref{simpleshapesmain} may be tightened to
\[
S^{|G|}_\Delta(\EP_\lb f^\infty_{\mf n}) = O(|\mf n|^{R_0(\Delta)}q_{S_1}^{A + B \kappa})
\]
under all the same conditions and where:
\begin{align*}
 R_0(\Delta)  &:= \f12\lf(N^2 - \sum_i T_i^2 d_i^2\ri) + \sum_i \lf( T_i^2 + \f12 T_i(T_i-1)(d_i^2-1) \ri) \\
 &= \bar R(\Delta) - \sum_i \lf(\f12 T_i^2 d_i(d_i+1) - \lf(T_i^2 + \f12 T_i(T_i-1)(d_i^2-1) \ri) \ri).
\end{align*}
\end{conj}
We think of $R_0(\Delta)$ again as making a modification to the dimension count of the parabolic that gives $\bar R(\Delta)$: replace the dimension of the parabolic corresponding to partition $(d_i^{(t_i)})$ in the $\GL_{t_id_i}$-block on the diagonal with $T_i^2 + \f12T_i(T_i-1)(d_i^2-1)$. 

The main obstacle to proving Conjecture \ref{simpleshapesconj} is showing that the asymptotic~\eqref{bestfixedvectorbound} is uniform enough in $\pi_v$. This 
would require uniform upper bounds on the coefficients $c_O(\pi_v)$ 
for tempered $\pi_v$. 
In particular, we would need bounds 
for \emph{non-maximal} $O$ in the wavefront set of $\pi_v$ so the techniques of \cite{MW87} don't apply.

Of course, Conjecture \ref{simpleshapesconj} can only be exact for $S^{|G|}_\Delta$ instead of $S^G_\Delta$. In the case where all the $d_i$ have the same parity, $s_\psi = 1$ for all $\psi \in \Delta$ so these agree. Otherwise, the terms $S_\psi$ for different $\psi \in \Delta$ are attached to a varying sign $\eps_\psi(s_\psi)$. If we na\"ively assume some non-trivial cancellation, we get the following:

\begin{conj}\label{simpleshapesconjeps}
 Recall the setup and conditions for Theorems \ref{simpleshapesmain}, \ref{simpleshapesstronger}, and Conjecture \ref{simpleshapesconj}. Then, if all the $d_i$ have the same parity,
\[
C_{\eps,1} |\mf n|^{R_0(\Delta) - \eps} \leq S^G_\Delta(\EP_\lb f^\infty_{\mf n}) \leq C_{\eps,2} |\mf n|^{R_0(\Delta) + \eps}
\]
for all $\eps > 0$ and some constants $C_{\eps, 1}, C_{\eps,2}$ such that $C_{\eps, 2} = O_\eps(q_{S_1}^{A + B \kappa})$. 

If the $d_i$ have different parities, then
\[
S^G_\Delta(\EP_\lb f^\infty_{\mf n}) = o(|\mf n|^{R_0(\Delta)} q_{S_1}^{A + B \kappa}).
\]
\end{conj}

\section{Application to Limit Multiplicities}\label{sectionlimitmult} 

Let $G$ be an extended pure inner form of $G^* \in \wtd {\mc E}_\el(N)$ and fix a cohomological representation $\pi_0$ of $G_\infty$. Fix $f^\infty$ unramified outside of a finite set of finite places $S$. In this section, we use the stabilization of the trace formula to estimate 
\[
m^G(\pi_0, f^\infty) := \sum_{\pi \in \mc{AR}_\disc(G)} m_\pi \1_{\pi_\infty = \pi_0} \tr_{\pi^\infty}(f^\infty)
\]
in terms of the bound on terms $S^H_{\Delta'}(\EP_\lb (f^\infty)')$ from  \S\ref{sectionSdeltaasymptotics}. 

\subsection{Preliminaries} \label{section shape preliminaries}
 Let $\psi_\infty(\pi_0)$ be the Arthur parameters $\psi_\infty$ at infinity such that $\pi_0 \in \Pi^G_{\psi_\infty}$. This is a finite (possibly empty) set since all such Arthur parameters share an infinitesimal character. For each $\psi_\infty \in \psi_\infty(\pi_0)$, there is a finite number of $\Delta$ such that $\psi \in \Delta$ have infinite component $\psi_\infty$ as in Lemma \ref{shapetoinfty}---intuitively, these are parameterized by ways to group together simple factors of $\psi_\infty$ that share the same Arthur-$\SL_2$. 

In total, we get a finite (possibly empty) set of shapes $\Delta(\pi_0)$ such that
\begin{equation}
\pi_0 \otimes \pi^\infty \in \Pi_\psi \implies \psi \in \Delta \text{ for some } \Delta \in \Delta(\pi_0).
\end{equation}
The spectral decomposition then produces
\begin{equation}\label{mgpi0}
m^G(\pi_0, f^\infty) = \sum_{\Delta \in \Delta(\pi_0)} \sum_{\psi \in \Delta} \sum_{\pi \in \Pi_\psi} m^\psi_\pi \1_{\pi_\infty = \pi_0} \tr_{\pi^\infty}(f^\infty).
\end{equation}
Specializing to a single $\Delta \in \Delta(\pi_0)$, consider for any test function, $f_\infty$ on $G_\infty$:
\begin{equation}\label{IDG}
I_{\Delta}^G(f_\infty f^\infty) := \sum_{\psi \in \Delta} I^G_\psi(f_\infty f^\infty) = \sum_{\psi \in \Delta} \sum_{\pi \in \Pi_\psi} m^\psi_\pi \tr_{\pi_\infty}(f_\infty) \tr_{\pi^\infty}(f^\infty).
\end{equation}
%
%
Applying Lemma \ref{charps} part (1) to \eqref{IDG} and comparing to \eqref{mgpi0} then gives:
\begin{cor}
Let $\pi_d$ be a discrete series representation appearing in the character formula for $\pi_0$ with sign $\sigma$. Then:
\[
m^G(\pi_0, f^\infty) = \sigma \sum_{\Delta \in \Delta(\pi_0)} I^G_\Delta(\varphi_{\pi_d} f^\infty).
\] 
\end{cor}

\subsection{Stabilization}
Now we compute the $I_\Delta^G$ by stabilizing. We begin with a more conceptual formula.

Let $H \in \mc E_\el(G)$. Then $H = H_1 \times H_2$ for $H_i$ quasisplit unitary groups. Therefore, we can abuse notation and use each $H_i$ to also denote a representative in some $\wtd {\mc E}_\el(N_i)$ that is isomorphic as an algebraic group. Any shape $\Delta$ on $H$ corresponds to a finite and possibly empty set of shapes $\Delta_1 \times \Delta_2$ on $H_1 \times H_2$ that push forward to~$\Delta$. Stabilization of each $I^G_\psi$ for $\psi \in \Delta$ (through Theorem \ref{stablemultendo}) gives:
\begin{prop}\label{mpi0form1}
Let $\pi_d$ be a discrete series representation appearing in the character formula for $\pi_0$ with sign $\sigma$. Then:
\begin{multline*}
m^G(\pi_0, f^\infty) = \sigma \sum_{\Delta \in \Delta(\pi_0)} \sum_{H_1 \times H_2 \in \mc E_\el(G)}  \\ \sum_{\Delta_1 \times \Delta_2} \iota(G, H_1 \times H_2) \prodf_{i = 1,2} S^{H_i}_{\Delta_i}((\varphi_{\pi_d})^{H_i} (f^\infty)^{H_i}),
\end{multline*}
where the $\star^{H_i}$ terms represent the corresponding factors of transfers to $H$.
\end{prop}
Each $(\varphi_{\pi_d})^{H_i}$ can be chosen to be a linear combination of EP-functions by standard formulas for transfers of pseudocoefficients. Explicitly, consider a pair $(H(s), \Delta'(s))$ that pushes forward to $\Delta$ and the corresponding $s \in \mc S_\Delta$. Then
\[
\tr_{\psi_\infty^{\Delta'}}(\varphi_{\pi_d}^H) = \eta^{\psi_\infty^\Delta}_{\pi_0}(s's_\Delta) = \sigma \eta^{\psi_\infty^\Delta}_{\pi_0}(s')
\]
by the endoscopic character identity \ref{localcharidentity} where $s'$ is the lift of $s$ to $S^\natural_\Delta$ therein. For the second equality, we use $\eta^{\psi_\infty^\Delta}_{\pi_0}(s_\Delta) = \sigma$ as in the proof of part 2 of Lemma \ref{charps}. If $\Delta'$ has total infinitesimal character $\lb$, then this gives
\[
\tr_{\psi_\infty^{\Delta'}}(\varphi_{\pi_d}^H) = \sigma \eta^{\psi_\infty^\Delta}_{\pi_0}(s') \tr_{\psi_\infty^\Delta}(\EP_\lb),
\]
since the EP-function trace is just $1$. Changing notation a bit produces:
\begin{cor}\label{mpi0form}
Let $G$ be an extended pure inner form of $G^* \in \wtd {\mc E}_\el(N)$ and fix a cohomological representation $\pi_0$ of $G_\infty$. Then
\begin{align*}
m^G(\pi_0, f^\infty) &= \sum_{\Delta \in \Delta(\pi_0)} m^G(\pi_0, \Delta, f^\infty) \\ 
&:= \sum_{\Delta \in \Delta(\pi_0)} \sum_{s \in \mc S_\Delta} \iota(G, H(s)) \eta^{\psi_\infty^\Delta}_{\pi_0}(s') \sum_{\Delta_1 \times \Delta_2}  \prodf_{i = 1,2} S^{H_i(s)}_{\Delta_i}(\EP_{\lb_i} (f^\infty)^{H_i}),
\end{align*}
where $(H(s), \Delta'(s))$ is the pair of group and shape corresponding to $s \in \mc S_\Delta$ which is lifted to $s' \in S^\natural_\Delta$ as in Theorem \ref{localcharidentity}, $\Delta_1 \times \Delta_2$ ranges over shapes on $H(s) = H_1(s) \times H_2(s)$ such that the pair $(H(s), \Delta_1 \times \Delta_2)$ is equivalent to $(H(s), \Delta'(s))$, and $\lb_i$ is the total infinitesimal character of $\Delta_i$. 
\end{cor}

\begin{note}
The sum over $\Delta_1 \times \Delta_2$ will be a singleton unless $H_1 \cong H_2$ in which case there will be two terms that differ by transposing the factors. 
\end{note}

\begin{note}
We briefly discuss the relation between this and the formulas in \cite{Lab11} for transfers of pseudocoefficients. For simplicity, assume we are in a case where we never have $H_1 \cong H_2$. 

Then, for each $\Delta$ there is only ever one possible choice $\Delta_1 \times \Delta_2$ due to the fixing of infinitesimal characters at infinity in shapes. The sum over EP-functions in Labesse's formulas then comes from the sum over $\Delta(\pi_0)$---these then correspond to different $\Delta_1 \times \Delta_2$ with non-conjugate infinitesimal characters. 

Next, while different $\Delta \in \Delta(\pi_0)$ may have the same $\psi_\infty(\Delta)$ or even $\mc S_\Delta$, the embeddings $\mc S_\Delta \into \mc S_{\psi_\infty^\Delta}$ will differ. This accounts for unexpected differences in signs of coefficients of EP-functions in Labesse's formulas. 
\end{note}

\subsection{Transfer Factors}\label{sectiontransferfactors}
We eventually want to apply Proposition \ref{mpi0form} to $f^\infty$ as in \S\ref{asymptoticsetup}. In the case where $S_0 = \emptyset$, we will need to compute explicit endoscopic transfers at all places so we need to choose explicit local transfer factors that are consistent globally. 

First, pick a global Whittaker datum $\om$ on the quasisplit form $G^*$ of $G$. Since we are assuming that $S_0 = \emptyset$, the group $G$ is unramified at all finite places. Therefore $G^*_v = G_v$ for all finite $v$ and is in particular also unramified. This implies that $G^*$ can be defined over $\mc O_F$ so we can choose $\om$ so that the induced local data $\om_v$ are unramified/admissible everywhere as in \cite[\S7]{Hal93}. This allows us to use the fundamental lemma for each $G_v$, though beware that we need to choose hyperspecial subgroups $K_v$ consistent with $\om_v$. 

Next, \cite[\S4.4]{Kal18} shows that the choice of $\om$ also gives us compatible local transfer factors on $G$ itself (we note that $G$ has simply connected derived subgroup to make the extra term in Theorem 4.4.1 disappear). At finite places $G_v$, the local factors stay the same as for the $G^*_v$.  

Finally, to explicitly pin down characters at infinity in \S\ref{SectionMoeglinCharacterComputation}, we pick our global Whittaker datum consistent with \cite[Rmq 4.5]{MR19} (we can pick it consistent with any Whittaker datum up to equivalence for $G_\infty$ since $E/F$ is CM and $G^*$ is quasisplit). 

\subsection{Limit Multiplicities}
As a preliminary/example computation, we work out what the summand $m^G(\pi_0, \Delta, f^\infty)$ from Proposition \ref{mpi0form} is for
\[
\Delta = (T_1, 1, \lb_1, \eta_1), (1, d_2, \lb_2, \eta_2)
\]
with $d_2$ odd so that $\Delta$ is odd GSK. In our eventual application, this will be the dominant term in the sum. 

For this $\Delta$, we have $\mc S_\Delta \cong \Z/2$ and the non-identity element $s$ corresponds to
\[
H(s) = H_1 \times H_2 = U(T_1) \times U(d_2).
\]
If $T_1 \neq d_2$, there is a unique choice
\[
\Delta_1 \times \Delta_2 = (T_1, 1, \lb_1, \eta_1) \times (1, d_2, \lb_2, \eta_2)
\]
and $\iota(G,H(s)) = 1/2$. If $T_1 = d_2$, there are two choices for $\Delta_1 \times \Delta_2$ that correspond to the exact same product of $S$-terms and $\iota(G,H(s)) = 1/4$. Either way, \ref{mpi0form} reduces to
\begin{multline}\label{goodshapeex1}
    m^G(\pi_0, \Delta, f^\infty) = S_{\Delta}^{G^*}(\EP_\lb (f^\infty)^{G^*}) \\
    + \f12 \eta_{\pi_0}^{\psi^\Delta_\infty}(s') S^{U(T_1)}_{(T_1, 1, \lb_1, \eta_2)}(\EP_{\lb_1} (f^\infty)^{U(T_1)}) \times S^{U(d_2)}_{(1,d_2, \lb_2, \eta_2)}(\EP_{\lb_2} (f^\infty)^{U(d_2)}),
\end{multline}
where $G^*$ is the quasisplit form of $G$ and where the product implicitly includes a sum over factorizable summands of the transfer to $H(s)$. 

Now assume
\[
f^\infty := f^\infty_{\mf n} = \varphi_{S_0} f_{S_1} \bar \1_{K^{G,S}(\mf n)}
\]
is of the form in Section \ref{asymptoticsetup}. Furthermore, assume that the chosen transfers of $\varphi_{S_0}$ satisfy Conjecture \ref{stabletransferconj} for $\Delta$.

Then, by Theorem \ref{goodshapebound}, the first summand of \eqref{goodshapeex1} has main term
:
\begin{multline*}
\f12 |\mf n_i|^{\f12(N^2 + T_1^2 - d_2^2) +1} \Gamma_{T_1, -1, 1}(\mf n_i)\f{\dim \lb_1}{|\Pi_\disc(\lb_1)|} \f{\vol(H'(F) \bs H'(\A_f))}{\vol(K^S_{H'}) \vol((H'_\infty)^c)}  \\ 
\varphi_{1,S_0}(1) f_{S_1}^{H_1}(1) \int_{H_{2, \der, S_0, S_1}}  f^{H_2}_{S_1} \varphi_{2, S_0}(h) \, dh,
\end{multline*}
where $H' = H_1 \times H_2^\ab$. We use here that being unramified makes $\mc T_i f_{S_1} = f^{H_i}_{S_1}$.

Next, considerations as in Lemma \ref{splitconstant} give that the transfer of the $\bar \1_{K^{G,S}(\mf n)}$ term is a constant term to a Levi, so Lemma \ref{congruenceconstantterm} gives
\[
\bar \1_{K^{G,S}(\mf n)}^H =I(\mf n) \bar \1_{K^{H_1,S}(\mf n)} \times \bar \1_{K^{H_2,S}(\mf n)}.
\]
We can therefore use Lemma \ref{congruenceconstantterm} computing endoscopic transfers to get that the second summand of \eqref{goodshapeex1} has main term:
\begin{multline}\label{goodshapeex1intermediate}
    \f12 |\mf n|^{\f12(N^2 - T_1^2 - d_2^2)} \Gamma_{-1}(\mf n) \eta_{\pi_0}^{\psi^\Delta_\infty}(s')
    S^{H_1}_{(T_1, 1, \lb_1, \eta_2)}(\EP_{\lb_1} \varphi^{H_1}_{S_0} f^{H_1}_{S_1} \bar \1_{K^{H_1,S}(\mf n)}) \times \\ S^{H_2}_{(1,d_2, \lb_2, \eta_2)}(\EP_{\lb_2} \varphi^{H_2}_{S_0} f^{H_2}_{S_1} \bar \1_{K^{H_2,S}(\mf n)}).
\end{multline}
Theorem \ref{simpleshapesmain} then gives that the first $s^H$ factor in \eqref{goodshapeex1intermediate} has main term
\[
|\mf n|^{T_1^2} \Gamma_{T_i}(\mf n) \varphi^{H_1}_{S_0}(1) f_{S_1}^{H_1}(1) \f{\vol(H_1(F) \bs H_1(\A_f))}{\vol(K^S_{H_1}) \vol(H^c_{1, \infty})}
\]
and by Proposition \ref{charshapeformula}, the second $S^H$ factor in \eqref{goodshapeex1intermediate} is eventually
\[
|\mf n| \Gamma_1(\mf n) \f{\vol(H^\ab_2(F) \bs H^\ab_2(\A_f))}{\vol(K_{H^\ab_2}^S) \vol(H^\ab_{2,\infty})} \int_{H_{2, \der, S_0, S_1}}  f^{H_2}_{S_1} \varphi^{H_2}_{S_0}(h) \, dh
\]
after using an argument as in Corollary \ref{simpleshapesstronger} to remove all the terms in the sum except for $1$. 

After multiplying everything together, this shows that the summands for $G$ and $H_1 \times H_2$ in \eqref{goodshapeex1} have the exact same asymptotic dependence on $\mf n_i$. 

When $S_0 = \emptyset$, we 
no longer need Conjecture \ref{stabletransferconj} and can collect these terms reasonably cleanly: for some $A,B,C,D,E$ with $C \geq 1$, as long as $|\mf n| \geq Dq_{S_1}^{E \kappa}$:
\begin{multline}
    |\mf n|^{-\f12(N^2 + T_1^2 - d_2^2) -1} \Gamma_{T_1, -1, 1}(\mf n_i)^{-1} m^G(\pi_0, \Delta, f_{S_1} \bar \1_{K^{G,S}(\mf n)}) \\
    = \1_{\eta_{\pi_0}^{\psi^\Delta_\infty}(\mc S_\Delta) = 1} \f{\dim \lb_1}{|\Pi_\disc(\lb_1)|} \f{\vol(H'(F) \bs H'(\A_f))}{\vol(K^S_{H'})\vol((H'_\infty)^c)} \lf( f_{S_1}^{H_1}(1) \int_{H_{2, \der, S_1}}  f^{H_2}_{S_1}(h) \, dh \ri) \\
    + O(|\mf n|^{-C} q_{S_1}^{A + B\kappa}).
\end{multline}
If $S_0 = \emptyset$, then $G$ is unramified at all finite places so  $\eta^{\psi_\infty^\Delta}_{\pi_0}$ necessarily factors through~$\mc S_\Delta$ because of the conditions on the $\chi_{G_v}$ in Theorem \ref{LocalPackets} to glue together to a global group as in Lemma \ref{unraminnerforms}. 

The case when $\eta^{\psi_\infty^\Delta}_{\pi_0}(\mc S_\Delta)$ isn't trivial can actually be understood more simply. If $S_0 = \emptyset$, then each factor $f_v$ of the test function is either on $\GL_{N,v}$ or unramified. Therefore, its traces vanish on all $\pi_v \in \Pi_{\psi_v}$ such that $\eta^{\psi_v}_{\pi_v} \neq 1$. By the multiplicity formula \ref{ArthurMultiplicty}, this implies that
\[
m^\psi_\pi \tr_{\pi^\infty}(f^\infty) = 0
\]
whenever $\eta_{\pi_\infty}^{\psi_\infty} \neq 1$. In total, 
\[
\eta^{\psi_\infty^\Delta}_{\pi_0}(\mc S_\Delta) \neq 1 \implies m^G(\pi_0, \Delta, f^\infty) = 0.
\]

An extension of this computation to all odd GSK shapes gives:
\begin{thm}\label{mpi0delform}
Let $G$ be an extended pure inner form of $G^* \in \wtd {\mc E}_\el(N)$ that is unramified at all finite places and let $\pi_0$ be a cohomological representation of $G_\infty$.  Choose odd GSK $\Delta = (T_i, d_i, \lb_i, \eta_i)_{1 \leq i \leq k} \in \Delta(\pi_0)$ and let $f^\infty_{\mf n} = f_{S_1} \bar \1_{K^{G,S}(\mf n)}$ be as in Section \ref{asymptoticsetup} with $S_0 = \emptyset$. 

Then, if $\eta^{\psi_\infty^\Delta}_{\pi_0}(\mc S_\Delta) \neq 1$
\[
m^G(\pi_0, \Delta, f^\infty_{\mf n}) = 0.
\]
Otherwise, there are $A,B,C,D,E$ with $C \geq 1$ such that as long as $|\mf n| \geq Dq_{S_1}^{E \kappa}$:
\begin{multline*}
|\mf n|^{R(\Delta)} \Gamma_{L(\Delta)}(\mf n)^{-1} m^G(\pi_0, \Delta, f^\infty_{\mf n})
\\= \f{\dim \lb_1}{|\Pi_\disc(\lb_1)|} \f{\vol(H'(F) \bs H'(\A_f))}{\vol(K^S_{H'})\vol((H'_\infty)^c)} \lf(f_{S_1}^{H_1}(1) \prodf_{i > 1} \int_{H_{i, \der, S_1}}  f^{H_i}_{S_1}(h) \, dh \ri) \\
+ O(|\mf n|^{-C} q_{S_1}^{A + B \kappa})
\end{multline*}
where $H_i = H(T_i, d_i, \lb_i, \eta_i)$, $H' = H_1 \times \prod_{i > 1} H_i^\ab$, and $R(\Delta)$ is as in \eqref{Equation R(Delta)}.

\end{thm}

For our upper bound we can allow $S_0 \neq \emptyset$:

\begin{thm}\label{mpi0bound}
Let $G$ be an extended pure inner form of $G^* \in \wtd {\mc E}_\el(N)$ that may or may not be unramified and let $\pi_0$ be a cohomological representation of $G_\infty$. Let~$\Delta \in \Delta(\pi_0)$ and let $f^\infty_{\mf n} = \varphi_{S_0} f_{S_1} \bar \1_{K^{G,S}(\mf n)} $ be as in Section \ref{asymptoticsetup}. 

Then there are $A,B,D,E$ such that as long as $|\mf n| \geq Dq_{S_1}^{E \kappa}$
\[
m^G(\pi_0, \Delta, f^\infty_{\mf n}) = O(|\mf n|^{R(\Delta)} q_{S_1}^{A + B \kappa})
\]
where $R(\Delta)$ is as in Corollary \ref{simpleshapesstronger}.
\end{thm}

\begin{proof}
Apply Proposition \ref{mpi0form} and then \ref{simpleshapesstronger} to each term.
\end{proof}
Note that Theorem \ref{mpi0bound} does \emph{not} depend on Conjecture \ref{stabletransferconj} since Corollary \ref{simpleshapesstronger} doesn't.




\section{Explicit Computations on Unitary Groups}\label{sectionunitary}
In this section, we recall the explicit combinatorial parameterization of cohomological representations of $U(p,q)$ and their $A$-packets and Adams-Johnson parameters following \cite{MR19}; see also \cite{VZ84, Tra01, BC05}.  This allows us to compute the sets $\Delta(\pi_0)$ and  work out explicit limit multiplicity statements from Proposition \ref{mpi0form1} together with the bounds in Theorems \ref{mpi0delform} and \ref{mpi0bound}. 

\subsection{Cohomological Representations of \lm{$U(p,q)$}}\label{upqparam}

\subsubsection{Setup:}
We recall some general facts about the construction of cohomological representations. 
Let $G$ be a reductive Lie group with Lie algebra $\fg_0$, Cartan $\mf t_0$, and $\fg = \fg_0 \ten \BC$. Let $\fg = \fk \oplus \fp$ be the Cartan decomposition corresponding to a choice of maximal compact subgroup $K$ with Cartan involution $\iota$\footnote{In lieu of the traditional $\theta$, which we reserve for the involution defining our unitary group.}. Assume that $G$ has a compact Cartan subgroup~$T$ with Lie algebra~$\ft \subset \fk$. Let~$\Delta(\ft,\fg)$ be the root system for~$\ft$ in~$\fg$. 

In \cite{VZ84}, Vogan-Zuckerman introduce the notion of a~$\iota$-stable parabolic subalgebras~$\fq$ of~$\fg$, henceforth referred to as VZ subalgebras. To construct such~$\fq$, choose~$x \in i \ft_0$ and define $\lambda_0 \in \ft^*$ by $[x,\fg^\alpha] = \ip{\lambda_0}{\alpha}\fg^\alpha$ for $\alpha \in \Delta(\fg,\ft)$. Then there are standard Levi and root space decompositions:
\begin{equation}\label{eq VZ subalgebras}\fq := \fl \oplus \fu, \quad \fl := \ft \oplus \bigoplus_{\ip{\lambda_0}{\alpha}=0} \fg^\alpha, \quad \fu :=  \oplus \bigoplus_{\ip{\lambda_0}{\alpha}\geq0} \fg^\alpha. \end{equation} 

Let $L = Z_G(\lambda_0)$ be the Levi subgroup of $G$ with Lie algebra $\fl$.
Let $\Delta(\ft,\fu)$ be the roots of $\ft$ in $\fu$, and let $\xi: \fl \to \BC$ be a character such that: 
\begin{itemize}
    \item[(i)] $\xi$ is the differential of a one-dimensional representation of $L$, and 
    \item[(ii)] $\ip{\alpha}{\xi} \geq 0$ for $\alpha \in \Delta(\ft,\fu)$.
\end{itemize}

Then Vogan-Zuckerman define a representation $A_\fq(\xi)$ and show that if $V_\xi$ is the irreducible finite-dimensional $G$-representation of highest weight $\xi$, then~$A_\fq(\xi)$ and~$V_\xi$ have the same infinitesimal character~$\xi + \rho_G$ and 
\[ 
H^*(\fg,K;A_\fq(\xi) \ten V_\xi^\vee ) \neq 0. 
\] 
Moreover, any representation with nontrivial $(\fg,K)$-cohomology is isomorphic to $A_\fq(\xi)$ for some pair $(\fq,\xi).$ 

\subsubsection{The parameterization} \label{section parameterization}
The parameterization of cohomological representations of $U(p,q)$ will be given in terms of the following combinatorial data:
\begin{dfn}\label{orderedpartitions}
 For a pair of non-negative integers $p,q$ with $p+q = N$, let:
 \begin{itemize}
     \item[-] $\mathcal P(N)$ to be the set of \emph{ordered} partitions of $N$, i.e. tuples $(N_1,\dotsc,N_r)$ where $r$ is arbitrary, each $N_i$ is positive, and $\sum_i N_i = N$.
     \item[-] $\mathcal P(p,q)$ be the set of ordered bipartitions of $(p,q)$, i.e. the set of tuples of pairs $ ((p_1,q_1),\dotsc,(p_r,q_r)) $ where the $p_i,q_i$ are a non-negative integers with each $p_i + q_i >0$ and  $\sum_i p_i = p$, $\sum_i q_i = q$. 
     \item[-] $\mathcal P_1(p,q) \subset \mathcal{P}(p,q)$ be the subset consisting of expressions $((p_1,q_1),\dotsc,(p_r,q_r))$ where for each $i$, if either of $p_i$ or $q_i = 0$, then the other is $1$.
 \end{itemize}

\end{dfn}

There are natural surjective maps: \begin{align}\label{eq partition maps}
\beta&: \mathcal{P}(p,q) \to \mathcal{P}(N), \quad((p_1,q_1),\dotsc,(p_r,q_r)) \mapsto (p_1+q_1,\dotsc,p_r+q_r),\\
\gamma&: \mathcal{P}(p,q) \to \mathcal{P}_1(p,q)\end{align}  where $\gamma$ replaces any term of the form $(n,0)$ (resp. $(0,m)$) by $n$ copies of $(1,0)$ (resp. $m$ copies of $(0,1)$.) 

The bipartitions $B \in \mathcal{P}(p,q)$ parameterize VZ subalgebras following \cite[\S 1-3]{Tra01}, which proves:

\begin{prop}
Let $G=U(p,q)$. 
\begin{itemize}
    \item[(i)] There is a bijection $B \mapsto \mf q_B$ between $K$-conjugacy classes of VZ subalgebras of $\fg$ and $ \mathcal{P}(p,q)$. 
    \item[(ii)] $A_{\fq_B}(0) \simeq A_{\fq_{B'}}(0)$ if and only if $\gamma(B) = \gamma(B')$. 
\end{itemize}
\end{prop}
We write the Levi subgroup $L_{B}$ associated to $B \in \mathcal{P}(p,q)$ as 
\[ 
L_{B} = U(p_1,q_1) \times \cdots \times U(p_r,q_r). 
\]
To realize the bijection, embed $K = U(p,0) \times U(0,q)$ in $U(p,q)$, write 
\[ 
\ft_0 = \ft_0 \cap \fu_0(p,0) \oplus \ft_0 \cap \fu_0(0,q) \simeq \BR^{p} \times \BR^q,
\]  
and assign to $B = ((p_1,q_1),\dotsc,(p_r,q_r))$ the algebra constructed from $ix_{B}$, for 
\[ 
x_{B} = (\overbrace{r,\dotsc,r}^{p_1},\overbrace{r-1,\dotsc,r-1}^{p_2},\dotsc,\overbrace{1,\dotsc,1}^{p_r},\overbrace{r,\dotsc,r}^{q_1},\dotsc,\overbrace{1,\dotsc,1}^{q_r}) \in \ft_0 \subset \fk_0. 
\]
Next, let $\lambda$ be an infinitesimal character. Recall the notion of $P$-parts from \S\ref{notationsequences}. 
\begin{dfn}
We say that a regular integral infinitesimal character
\[
\lb = \lb_1 > \cdots > \lb_n
\]
is adapted to partition $P$ if the $P$-parts of $\lb$ are all of the form 
\[
X^r \sum_{i = 1}^n X^{(n-2i+1)/2}
\]
for some integer or half-integer $r$ and some integer $n$.
\end{dfn}
For example, the infinitesimal character $\rho_G$ of the trivial representation is adapted to all partitions of $N$. The following is also deduced from \cite{Tra01}: 
\begin{prop}
Let $G = U(p,q)$. The cohomological representations with regular integral infinitesimal character $\lb$ are all of the form $A_{\mf q_B}(\lb - \rho_G)$ for $B \in \mc P(p,q)$ such that~$\lb$ is adapted to $\beta(B)$. In particular, they are in bijection with the bipartitions~$B \in \mc P_1(p,q)$ with $\lambda$ adapted to~$\beta(B)$.  
\end{prop}

To compute the cohomology associated to the representations, one makes a choice of complex structure on $\fp$ i.e. on the quotient $G /K$. To do this, fix a Shimura datum for $G$ to be the conjugacy class of \[ h_K:\mathbb{S} \to U(p,q), \quad  h_K(z) = \left(\frac{z}{\bar{z}} I_p, I_q\right) \in U(p,0) \times U(0,q) \subset U(p,q).\] This induces a decomposition $\fp = \fp^+ \oplus \fp^-$ where $\Ad(h(z))$ acts on $\fp^+$ by $z/\bar{z}$ and on $\fp^-$ by its inverse.


\begin{lem} \label{LemmaCohomology}
Let $G=U(p,q)$ with a choice of Shimura datum as above. Let $V_\xi$ be a finite-dimensional representation with highest weight $\xi$. Let $B = ((p_1,q_1),\dotsc,(p_r,q_r)) \in \mathcal{P}(p,q)$ be such that $\xi+\rho_G$ is adapted to $\beta(B)$, and let 
\[  
R= pq-\sum_i p_iq_i, \quad R^+ = \sum_{i<j} p_iq_j, \quad  R^- = \sum_{i>j} p_iq_j. 
\] 
Then: 
\begin{enumerate}[(i)]
    \item the smallest value of $i$ such that $H^i(\fg,K; A_{\fq_{B}}(\xi)\ten V_\xi) \neq 0$ is $i=R$; 
    \item $H^j(\fg,K; A_{\fq_{B}}(\xi)\ten V_\xi) \neq 0$ if and only if $j = R + 2p$ for $0 \leq p \leq \sum_i p_i q_i$; 
    \item the Hodge weights of $A_{\fq_{B}}(\xi)$ in degree $R + 2p$ are $(R^+ + p,R^- + p)$. 
\end{enumerate}
\end{lem}
\begin{proof}

This follows from \cite[Thm 3.3, 5.4, 6.19]{VZ84}: the first nonzero degree of cohomology of $A_{\fq_B}(\xi)$ is $R = \dim \fu\cap \fp = \dim \fp - \dim \fl \cap \fp = pq-\sum_i p_iq_i$.  More precisely, the cohomology of $A_{\fq_{B}}(\xi)$ in degree $\dim \fu \cap \fp$ appears in weight $(R^+,R^-)=(\fu \cap \fp^+, \fu \cap \fp^-)$. All other weights $(a,b)$ for which $A_\fq(\xi)$ has cohomology are of the form $(R^+ + p, R^- + p)$. One computes $R^+$ and $R^-$ from $x_{B}$ and the choice of Shimura datum; see e.g. \cite[\S 5]{BC05} for more details. 
\end{proof}

\subsubsection{Packets of cohomological representations} \label{cohomological A-packets and partitions}

Recall form \S \ref{subsection AJ 3} that any parameter $\psi$ associated to a packet of cohomological representations is an AJ parameter. In particular, there is a Levi subgroup $$\wh L \simeq GL_{n_1} \times \cdots \times GL_{n_r} \subset \wh G$$ such that $\psi(W_\BC \times I) \subset Z(\wh L)$, and such that $\psi(1 \times SL_2)$ is a principal $\SL_2$ in $\wh L$. Thus, for some $t_i$, we have
\[
\psi\mid_{W_{\BC} \times \SL_2} = \bigoplus_{i=1}^r \chi_{t_i}\boxtimes [n_i],
\] 
where $[n_i]$ is the irreducible $n_i$-dimensional representation of $\SL_2(\BC)$ and
\[ 
\chi_{t_i}(z) = \left(\frac{z}{\bar{z}}\right)^{t_i/2}, \quad z \in W_{\BC} \simeq \BC^\times. 
\]  
 In this case, the total infinitesimal character of $\psi$ is 
\[
\lambda_\psi = \sum_{i=1}^r X^{\frac{t_i}{2}}\sum_{j=1}^{n_i}X^{\frac{n_i-2j+1}{2}}. 
\] 
Since we assume that $\lambda_\psi$ is the infinitesimal character of a finite-dimensional representation, we have for each $i$ that $t_i-n_i \equiv N \mod 2$, i.e. we are in the good parity case in the sense of \cite[\S 4]{MR19}. To a good parity parameter, M{\oe}glin-Renard attach the ordered partition $P=(n_1,\dotsc,n_r) \in \mathcal{P}(N)$: specifically, the unordered multiset of $n_i$ come from restricting to the Arthur-$\SL_2$ and the ordering is such that the infinitesimal character of the summand $\chi_{t_i} \otimes [n_i]$ of $\psi$ is the $i$th $P$-part of $\lb_\psi$ (in particular, $P$ is adapted to $\lb_\psi$). They also show that the corresponding packet is 
\[
\Pi_\psi = \{ A_{\fq_{B}}(\xi_{B}) \mid \beta(B) =P \}, 
\] 
where $\xi_{B} = \boxtimes_i \det^{(t_i+n_i-N)/2-n_{<i}} $ for $n_{<i} = \sum_{j<i} n_j.$ In short, the representations in the packet $\pi_\psi$ are in bijection with the bipartitions $((p_1,q_1),\dotsc,(p_r,q_r)) \in \mathcal{P}(p,q)$ such that $p_i+q_i = n_i$. 

M\oe glin-Renard do more: first they fix a Whittaker datum for a choice of quasisplit form $G^*$ of $G$ (see \cite[Rmq 4.5]{MR19}) which without loss of generality we can have match the one from Section \ref{sectiontransferfactors}.  Then, they write down explicitly the character $\eta_{\pi}: S^\natural_\psi \to \pm 1$ attached by Arthur to each representation in the packet $\Pi_\psi$ in terms of the bipartition $B$ such that $\pi=\pi_B$. These constructions are recalled in \S \ref{SectionMoeglinCharacterComputation}, where they are used.

\subsection{\lm{$\Delta(\pi_0)$} and \lm{$R(\Delta)$}}\label{sectionunitaryparamcombo}
Fix an extended pure inner form~$G$ of $G^* \in \wtd {\mc E}_\sm(N)$. Let 
\[
\pi_0 = \bigotimes_{v \in \infty} \pi_{0,v}
\]
be a cohomological representation of
\[
G_\infty = \prod_{v \in \infty} U(p_v, q_v).
\] 
We will study the set $\Delta(\pi_0)$ introduced in \S \ref{section shape preliminaries} and for each $\Delta \in \Delta(\pi_0)$, we compute the invariant $R(\Delta)$ from Theorem \ref{simpleshapesstronger} and the set of shapes that realize it.
\begin{dfn}\label{Deltamax}
Let $\Delta^{\max}(\pi_0)$ be the set of $\Delta \in \Delta(\pi_0)$ with maximal $R(\Delta)$.  We will denote this common maximal value $R(\pi_0)$.
\end{dfn}


\subsubsection{Ignoring $\eta_i$}
First, since each $\Delta = (T_i, d_i, \lb_i, \eta_i)_i \in \Delta(\pi_0)$ satisfies $H(\Delta) = G^*$, the $\eta_i$ are determined by $T_i$ and $d_i$ according to Section \ref{ACassigntogroup}. As such, we will ignore them in this section. 

\subsubsection{Arthur-$\SL_2$'s} \label{subsection ArthurSL2Max}
We first study the $\Delta$ according to their Arthur-$\SL_2$. These $N$-dimensional representations of $\SL_2$ correspond to unordered partitions of $N$ via their decomposition into irreducibles. One first sees that:
\begin{lem}\label{maxoversl2}
Among the shapes~$\Delta$ with Arthur-$\SL_2$ given by unordered partition
\[
Q = (a_1^{(r_1)}, \dotsc, a_k^{(r_k)})
\]
with $a_i$ distinct, the value of the invariant $R(\Delta)$ introduced in Theorem \ref{simpleshapesstronger} is maximized for shapes $\Delta = (r_i, a_i, \lb_i, \eta_i)_{1 \leq i \leq k}$, in which each distinct integer $a_j$ appears once. 
Denote by $R(Q)$ this maximized value of $R(\Delta)$.

Furthermore, for cohomological representations $\pi_0$, if there is $\Delta \in \Delta(\pi_0)$ whose Arthur-$\SL_2$ matches $Q$, then there is also $\Delta' \in \Delta(\pi_0)$ such that $R(\Delta') = R(Q)$.
\end{lem}

\begin{proof}
Constructing any other shape with the same Arthur-$\SL_2$ would require splitting up some $(r_i, a_i, \lb_i, \eta_i)$ into smaller blocks, which would decrease $R(\Delta)$. 

For the second part, we may always merge blocks in $\Delta$ with the same $a_i$ by concatenating their infinitesimal characters. This leaves $\psi^\infty_\Delta$ unchanged by the construction in Lemma \ref{shapetoinfty}. 
\end{proof}

This reduces the study of $R(\pi_0)$ to understanding the Arthur-$\SL_2$'s for $\Delta \in \Delta(\pi_0)$:
\begin{dfn}\label{Qmax}
Let $Q^{\max}(\pi_0)$ be the set of unordered partitions representing the Arthur-$\SL_2$'s of $\Delta \in \Delta(\pi_0)$ with maximal $R(\Delta)$. Equivalently by Lemma \ref{maxoversl2}, it is the set of Arthur-$\SL_2$'s of $\Delta \in \Delta^{\max}(\pi_0)$. 
\end{dfn}
By Lemma \ref{shapetoinfty}, the possible Arthur-$\SL_2$'s for $\Delta \in \Delta(\pi_0)$ are the possible Arthur-$\SL_2$'s for $\psi_\infty$ with $\pi_0 \in \Pi_{\psi_\infty}$. We can therefore enumerate them by our classification of cohomological representations.

Fix a place $v$ and let $\pi_{0,v}$ correspond to $B_v = (p_{i,v}, q_{i,v})_i \in \mc P_1(p_v, q_v)$ and infinitesimal character $\lb_v$. We next study $\Delta(\pi_{0,v})$: the union of $\Delta(\pi')$ over all~$\pi'$ with $\pi'_v = \pi_{0,v}$. Recall from \eqref{eq partition maps} that $\beta(B_v)$ is the ordered partition of $N$ associated to $B_v$. We define some combinatorial objects:
\begin{itemize}
    \item $\beta_{>1}(B)$ is the unordered subpartition of $\beta(B)$ corresponding to parts with size bigger than $1$. 
    \item $Q_p(\pi_0)$ (resp. $Q_q(\pi_0)$) is the unordered partition $(n_j)_{j \in J}$ where the $j$ correspond to runs of length $n_j$ of consecutive $(p_i, q_i)$ of the form $(1,0)$ (resp. $(0,1)$) such that the corresponding piece of $\lb$ is of the form 
    \[
    X^r \sum_{i=1}^{n_j} X^{(n-2i +1)/2}.
    \]
\end{itemize}
 Next, for two unordered partitions $Q_1 = (n_i)_{i \in I}$ and $Q_2 = (m_j)_{j \in J}$, we say that $Q_2$ refines $Q_1$ if there is a map $J \to I$ such that the sum of $m_j$ over the fiber at $i$ is $n_i$.

\begin{lem}\label{enumsl2}
The possible Arthur-$\SL_2$'s for $\Delta \in \Delta(\pi_{0,v})$ correspond exactly to unordered partitions
\[
(X,Y,\beta_{>1}(B_v)),
\]
where $X$ refines~$Q_p(\pi_{0,v})$ and~$Y$ refines~$Q_q(\pi_{0,v})$.
\end{lem}

\begin{proof}
Lemma \ref{LemmaCohomology} tells us that the~$P \in \mc P(p+q)$ that correspond to $\psi_\infty$ with~$\pi_0 \in \Pi_{\psi_\infty}$ are produced by merging runs of consecutive $1$'s in~$\beta(B_v)$ that correspond to parts in~$B_v$ all of the form~$(1,0)$ or all of the form~$(0,1)$. Furthermore, we require that the coarsened partition thereby produced is still adapted to~$\lb_v$. These conditions together show that $P$ is an ordering of something of the form $(X,Y,\beta_{>1}(B_v))$.

Consider such a partition 
\[
P = (n_1, \dotsc, n_k).
\]
We will show that there is $\Delta \in \Delta(\pi_{0,v})$ with $\psi^\Delta_v$ corresponding to $P$. Let $I_1$ be the subset of indices such that $n_i = 1$ and $I_+$ its complement.  Let~$\lb_v$ have~$P$-parts $\lb^P_1, \dotsc, \lb^P_k$. 
Since $\lb_v$ is adapted to~$P$, for each~$i \in I_+$, there exists~$\lb_{i,v}$ such that the shape $(1, n_i, \lb_{i,v})$ has total infinitesimal character $\lb^P_i$ at~$v$. Next, let $\lb'_v$ be the concatenation of $\lb^P_i$ for $i \in I_1$. Finally, choose the other components for $w \neq v$ of $\lb_i$ and $\lb'$ arbitrarily. Consider
\[
\Delta = (|I_1|, 1, \lb', \eta), ((1, n_i, \lb_i, \eta_i))_{i \in I_+}.
\]
Then, by the constructions in \S\ref{cohomological A-packets and partitions}, $\psi^\Delta_v$ corresponds to $P$. Note that $\Delta \in \Delta(\pi_{0,v})$ by the form of $P$. 
\end{proof}

Next, we define
\[
\beta_{>1}(\pi_0) := \bigcup_{v \in \infty} \beta_{>1}(B_v)
\]
where the union is interpreted as of non-disjoint multisets (the multiplicity of an element is exactly equal to the maximum of its multiplicities in the 
$\beta_{>1}(B_v)$).
\begin{lem}\label{enumsl2E}
The possible Arthur-$\SL_2$'s for $\Delta \in \Delta(\pi_0)$ correspond exactly to unordered partitions $Q$ such that, for each $v \in \infty$, $Q$ can be written in the form:
\[
Q = (X_v,Y_v,\beta_{>1}(B_v))
\]
where $X_v$ refines $Q_p(\pi_{0,v})$ and $Y_v$ refines $Q_q(\pi_{0,v})$. In particular:
\begin{itemize}
    \item All Arthur-$\SL_2$'s for $\Delta \in \Delta(\pi_0)$ contain $\beta_{>1}(\pi_0)$ as a subpartition,
    \item
    if $\Delta(\pi_0)$ isn't empty, there is $\Delta \in \Delta(\pi_0)$ with Arthur-$\SL_2$
    \[
    (1, \dotsc, 1, \beta_{>1}(\pi_0)).
    \]
\end{itemize}
\end{lem}

\begin{proof}
That the Arthur-$\SL_2$'s must be contained in this set is an elementary combinatorial extension of the argument in \ref{enumsl2}. 

Existence of a $\Delta = (T_i, d_i, \lb_i)_i$ with a particular $\SL_2$ follows from fixing each component of the $\lb_i$ as in the argument of Lemma \ref{enumsl2} instead of just the $v$-component. 

Finally, the two bullet points are also elementary combinatorial properties of the set of such simultaneous $(X_v, Y_v, \beta_{>1}(B_v))$.
\end{proof}

\begin{note}\label{deltamaxalgo}
We summarize this section as a three step algorithm for computing $\Delta^{\max}(\pi_0)$:
\begin{enumerate}
    \item Find the possible Arthur-$\SL_2$'s for $\Delta \in \Delta(\pi_0)$ by Lemma \ref{enumsl2E}. 
    \item Compute $R(Q)$ for each of these Arthur-$\SL_2$'s to compute $Q^{\max}(\pi_0)$. 
    \item $\Delta^{\max}(\pi_0)$ is then partitioned into non-empty parts corresponding to the $Q \in Q^{\max}(\pi_0)$. Each part can be determined by Lemma \ref{maxoversl2}. 
\end{enumerate}
\end{note}
It turns out that the second step becomes much easier for GSK-shapes. Showing this will make up the remainder of our combinatorial work.


\subsubsection{The key bound}
We next prove 
an elementary combinatorial bound that is basically 
 a reformulation of Lemma 7.1 in \cite{MS19}. Recall the numerical invariants $\bar R(\Delta)$ and $R(\Delta)$, introduced respectively in \eqref{Equation R(Delta)=bar} and \eqref{Equation R(Delta)}; 
 $R(Q)$ is 
 the maximum of $R(\Delta)$ over the $\Delta$ such that $ \psi_\infty^\Delta$ corresponds to $Q$.
Part of the complexity of this argument is an artifact of only being able to prove the suboptimal bound $R(\Delta)$ from Corollary \ref{simpleshapesstronger} instead of the optimal $R_0(\Delta)$ from Conjecture \ref{simpleshapesconj}.

\begin{lem}\label{maxsl2}
Let $Q_0$ be an unordered partition that has distinct parts and no parts of size $1$. Then the maximum value of $R(Q)$ over all unordered partitions $Q$ of $N$ that have subpartition $Q_0$ is achieved by
\[
Q_\can = (1^{(r)}, Q_0). 
\]
Furthermore, if either $r \neq 2$ or $Q_0$ has no parts of size $2$, this is the unique such $Q$ that achieves this maximum. 
\end{lem}

\begin{proof}
Let
\[
Q_\can = (1, \dotsc, 1, Q_0) = (1^{(r)}, (a_i^{(r_i)})_i)
\]
for $a_i$ distinct and $r_i = 1$. The other possible $Q$ containing $Q_0$ are produced by decreasing the number of $1$'s and increasing one of the $r_i$'s. We will therefore show that $R(Q)$ decreases if we increase any of the $r_i$. 

Recall from Remark \ref{Remark RQ} that $R(Q)$ is obtained by starting from $\bar{R}(Q)$ (which is equal to the dimension of a certain parabolic block matrix) and then replacing the dimensions of certain blocks on the diagonal with modified counts. Along the diagonal, the blocks are indexed by the number $1$ and the distinct $a_i$. Changing an $r_i$ changes three blocks of the parabolic: the two on the diagonal associated to $1$ and $a_i$ and one corresponding to this pair above the diagonal. In particular, we can treat changes in each $r_i$ independently. Furthermore, changing an $r_i$ from $0$ (i.e. creating a new block) can be easily seen to decrease $R(Q)$.  

In general, going from $r_i = 1$ to $r_i = k$ changes the modified summand associated to these three blocks from:
\begin{multline*}
r^2 + ra_i + 1 \mapsto \\ (r - (k-1)a_i)^2 + ka_i(r - (k-1)a_i) + (\text{modified count for $a_i$ with $r_i = k$}). 
\end{multline*}
Expanding out, the change in $R(Q)$ is
\[
(\text{modified count for $a_i$ with $r_i = k$}) - a_i(a_i + r)(k-1) - 1,
\]
where the modified counts are the counts for the part of formula for $R(Q)$ associated to the $a_i$-block on the diagonal.  The modified count defining $R(Q)$ is defined differently for $r_i = 1,2,3$ versus everything else. Therefore we look at cases:

\begin{itemize}
\item
If an $r_i$ is increased to $2$, the modified count is
$
2a_i(a_i - 1) + (a_i + 3)
$, 
making the total difference equal to 
\[
2a_i(a_i - 1) + (a_i + 3) -  a_i(a_i + r) - 1.
\]
Using $r \geq a_i$, this is at most $ 
- a_i + 2
$
and is always negative unless $r = a_i = 2$. 
\item
If an $r_i$ is increased to $3$, the modified count is
$ 
\f92a_i(a_i - 1) + ((4+\eps)a_i + 5)$,
making the total difference
\[
\f92a_i(a_i - 1) + ((4+ \eps)a_i + 5) - 2a_i(a_i + r) - 1.
\]
Using $r \geq 2a_i$, this is bounded above by
\[
-\f32 a_i^2 - \lf(\f12 - \eps\ri)a_i + 4,
\]
which is always negative for in particular $\eps < 1/2$ since $a_i \geq 2$. 
\item
If an $r_i$ is increased to $k \geq 4$, then the change in modified counts is $
\f{k^2}2 a_i(a_i + 1)$,
making the total difference
\[
\f{k^2}2 a_i(a_i + 1) - a_i(a_i + r)(k-1) - 1.
\]
Using $r \geq (k-1)a_i$, this is bounded above by
\[
-\lf( \f{k^2}2 - k\ri) a_i^2 + \f{k^2}2 a_i - 1
\]
and using $a_i \geq 2$ and $k^2/2 - k > 0$ gives an upper bound by
\[
-\lf( \f{k^2}2 - 2k\ri) a_i - 1,
\]
which is always negative when $k \geq 4$. 
\end{itemize}
In total, if we increase any of the $r_i$, then $R(Q)$ decreases.
\end{proof}

\subsubsection{Summary of combinatorial work}
To conclude:
\begin{dfn}\label{Qcan}
Let $\pi_0=\bigotimes_{v \mid \infty} \pi_{0,v}$,  
with $\pi_{0,v}$ from $B_v = (p_{i,v}, q_{i,v})_v \in \mc P_1(p_v, q_v)$. Let $\beta_{>1}(B_v)$ be the unordered subpartiton of pieces of size greater than $1$ in $\beta(B_v)$. Let $\beta_{>1}(\pi_0)$ be the union of all the $\beta_{>1}(B_v)$ as multisets. Finally, define
\[
Q_\can(\pi_0) := (1, \dotsc, 1, \beta_{>1}(\pi_0))
\]
as an unordered partition of $N$ (if it exists). 
\end{dfn}

Then:
\begin{prop}\label{Deltamaxenum}
Let $Q_\can(\pi_0)$ be of the form
\[
Q_\can(\pi_0) = (1^{(r)}, a_1, \dotsc, a_k)
\]
for the $a_i$ distinct (i.e. $\beta_{>1}(\pi_0)$ has distinct parts). Further assume that either $r \neq 2$ or there is no $i$ such that $a_i = 2$. Then $\Delta^{\max}(\pi_0)$ consists of all shapes of the form
\[
(r, 1, \lb), (1, a_i, \lb_i)_{1 \leq i \leq k}
\]
in $\Delta(\pi_0)$ and $\Delta^{\max}(\pi_0) \neq \emptyset$ provided that $\Delta(\pi_0) \neq \emptyset$. 
\end{prop}

\begin{proof}
This is the result of applying algorithm \ref{deltamaxalgo} keeping Lemma \ref{maxsl2} in mind for step (2) to get that $Q^{\max}(\pi_0) = \{Q_\can(\pi_0)\}$. 
\end{proof}

We warn that Proposition \ref{Deltamaxenum} does not neccesarily hold if $\beta_{>1}(\pi_0)$ has repeated elements. For example, consider $\beta_{>1}(\pi_0) = (2,2,2,2)$. Then we have 
\[
R(2^{(4)},1,1) = 67 <  R(2^{(5)}) = 74.
\]
In fact, this is a counterexample to even the analogous statement with the conjectural optimal bound $R_0$. Whether any $\Delta \in \Delta(\pi_0)$ can have Arthur-$\SL_2$ given by $(2^{(5)})$ depends on what exactly the $Q_p(\pi_{0,v})$ and $Q_q(\pi_{0,v})$ are. Therefore, a general description of $Q^{\max}(\pi_0)$ is much more complicated.

\subsubsection{GSK-maxed representations}
Now we restrict to the special class of representations we can study:

\begin{dfn}\label{GSKm}
Let $\pi_0$ be a cohomological representation of some $U(p,q)$ corresponding to bipartition $(p_i, q_i)_i \in \mc P_1(p+q)$. We say $\pi_0$ is GSK-maxed if the only value of $p_i + q_i$ that appears with multiplicity is $1$.

We say $\pi_0$ is odd GSK-maxed if in addition all the $p_i + q_i$ are odd. 
\end{dfn}

\begin{dfn}\label{GSKmE}
Let $\pi_0$ factor into $\pi_{0,v}$ that each corresponding to bipartition $(p_{i,v}, q_{i,v})$. Then we say that $\pi_0$ is GSK-maxed if $\Delta(\pi_0) \neq \emptyset$ and the only number that appears with multiplicity among the $p_{i,v} + q_{i,v}$ is $1$.

We say $\pi_0$ is odd GSK-maxed if in addition all the $p_{i,v} + q_{i,v}$ are odd. Equivalently, $\beta_{>1}(\pi_0)$ is a partition into (odd) distinct parts. 
\end{dfn}

These definitions are justified by the following corollary of Proposition \ref{Deltamaxenum}: 
\begin{cor}\label{maxedshapeE}
Fix an extended pure inner form $G$ of $G^* \in \wtd {\mc E}_\el(N)$ and let $\pi_0$ be a cohomological representation of $G_\infty$ that is (odd) GSK-maxed. Further assume that if $\beta(\pi_0) = (1^{(r)}, a_1, \dotsc, a_k)$, then either $r \neq 2$ or none of the $a_i = 2$ (this is automatically satisfied if $\pi_0$ is odd GSK-maxed). 

Then all elements of $\Delta^{\max}(\pi_0)$ are (odd) GSK. 
\end{cor}

We warn that in general, $\Delta^{\max}(\pi_0)$ isn't a singleton. The different possibilities differ by different assignments of infinitesimal characters $\lb_{i,v}$ to each block $(T_i, d_i)$. 
\begin{ex}
    Consider $F$ with two infinite places $v,w$ and $G_v \cong G_w \cong U(6,1)$. Let
\begin{equation*}
\pi_{0,v} = (1,1), (1,0)^{(5)}, \quad \text{and} \quad \pi_{0,w} = (2,1), (1,0)^{(4)}
\end{equation*}
have the infinitesimal character of the trivial representation:
\[
\lb = (3, 2, 1, 0, -1,-2,-3).
\]
If $\Delta^{\max}(\pi_0) \neq \emptyset$, any $\Delta \in \Delta^{\max}(\pi_0)$ is of the form
\[
\Delta = (2,1, (\lb^1_v, \lb^1_w)), (1,2, (\lb^2_v, \lb^2_w)), (1,3, (\lb^3_v, \lb^3_w)),
\]
and the unordered partition $Q^{\max}(\pi_0)$ is $(3,2,1,1)$.
We are forced to choose
\[
\lb^2_v[2] = (3, 2), \qquad \lb^3_w[3] = (3, 2, 1).  
\]
However, we still need to pick $\lb^3_v$ and $\lb^2_w$. This will correspond to a choice of ordering of $Q^{\max}(\pi_0)$ at each of $v$ and $w$.

There are three choices
\[
\lb^3_v[3] = (1,0,-1), (0, -1, -2), \text{ or }  (-1,-2,-3)
\]
corresponding to three choices
\[
\lb^1_v = (-2,-3), (1, -3), \text{ or } (1, 0)
\]
and three orderings of $Q^{\max}(\pi_0)$:
\[
(2,3,1,1), (2,1,3,1), \text{ or } (2,1,1,3). 
\]

Similarly, there are three choices
\[
\lb^2_w[2] = (0,-1), (-1, -2), \text{ or } (-2, -3)
\]
corresponding to three choices
\[
\lb^1_w = (-2, -3), (0, -3), \text{ or } (0,-1)
\]
and three orderings of $Q^{\max}(\pi_0)$:
\[
(3,2,1,1), (3,1,2,1), \text{ or } (3,1,1,2). 
\]
Thus in total $\Delta^{\max}(\pi_0)$ can contain up to nine elements. Note that the different possibilities for $\lb^1_v$ (resp. $\lb^1_w$) aren't necessarily even character twists of each other.

\end{ex}

\begin{ex}As an simple example where this difficulty doesn't appear, assume
\[
G_\infty = U(N-1,1)^r
\]
and let $\pi_0 \cong \pi_{0,v}^r$ be diagonal. There is only one non-$1$ entry in any element of~$\mc P_1(N-1,1)$ so  $\pi_0$ is GSK-maxed. The arguments in Lemmas \ref{enumsl2} and \ref{enumsl2E} show that since $\beta_{>1}(\pi_0) = \beta_{>1}(B_v)$ 
is a singleton, there is exactly one way to assign infinitesimal characters $\lb_{i,v}$ to the blocks $(T_i, d_i)$, so that $\Delta^{\max}(\pi_0)$ is a singleton. 
\end{ex}

\subsection{Characters of the Component Group} 
\label{SectionMoeglinCharacterComputation}
As a last piece of the puzzle to derive explicit limit multiplicities for individual representations from Theorem \ref{mpi0delform}, we attach  to elements of the AJ-packet $\Pi_{\psi_\infty}$ characters of the group $S^\natural_{\psi_\infty}$ introduced in \S\ref{AClocalSpsi}. Let $P = (a_1, \dotsc, a_k) \in \mc P(N)$ correspond to an $A$-parameter~$\psi_\infty$ at infinity for $U(p,q)$ considered as an extended pure inner form. 
As mentioned after \cite[(1.3)]{MR19}, the equation \eqref{SpsiAJ} reduces to a canonical isomorphism
\[
S^\natural_{\psi_\infty} = \bigoplus_{1 \leq i \leq k} \Z/2 = \langle \eps_i \rangle_{1 \leq i \leq k}. 
\]
In addition, the subgroup~$\Z(\wh G)^\Gamma$ is the diagonally embedded $\Z/2$.

Though the character attached to a representation in a local $A$-packet is not explicit in general, M{\oe}glin-Renard in \cite{MR19} make it so for AJ-packets on unitary groups. Let
\[
\pi_0 = (p_i, q_i)_i \in \beta^{-1}(P) \subseteq \mc P(p,q) 
\]
and define $a_{<i} = \sum_{j=1}^{i-1} a_i$. Then,
\[
\eta^{\psi_\infty}_{\pi_0}(\eps_i) = (-1)^{p_i a_{<i} + q_i(a_{<i}+1) + a_i(a_i-1)/2}.
\]
We can simplify this to
\begin{equation}\label{etachar}
\eta^{\psi_\infty}_{\pi_0}(\eps_i) = (-1)^{a_i a_{<i} + q_i + \chi_4(a_i)},
\end{equation}
where
\[
\chi_4(a_i) := \begin{cases} 0 & a_i \equiv 0,1 \pmod 4 \\ 1 & a_i \equiv 2,3 \pmod 4 \end{cases}.
\]
If the $a_i$ are all odd, this further simplifies to
\begin{equation}\label{etacharodd}
\eta^{\psi_\infty}_{\pi_0}(\eps_i) = (-1)^{(i-1) + q_i + \chi_4(a_i)}.
\end{equation}
Finally, if $\psi_\infty = \psi_\infty^\Delta$ for $\Delta$ of the type in Lemma \ref{maxoversl2}, then
\[
\mc S^\natural_\Delta = \lf\langle s_d := \sum_{i : a_i = d} \eps_i \ri\rangle_{d \in \Z^+}
\]
is a subgroup of $S^\natural_{\psi_\infty}$. 
As a result, we can characterize representations such that $\eta^{\psi_\infty^\Delta}_{\pi_0}(S^\natural_\Delta) = 1$; this will be used to give asymptotics for multiplicities of individual representations in the odd GSK-maxed case.
\begin{lem}\label{etacharoddGSK}
Let $G$ be an extended pure inner form of $G^* \in \wtd {\mc E}_\el(N)$. Let $\pi_0 = \bigotimes_v \pi_{0,v} $ be an odd GSK-maxed representation of $G_\infty = \prod_v U(p_v, q_v)$. 

Let $Q^{\max}(\pi_0) = (1^{(r)}, d_1, \dotsc, d_k)$ with $d_i$ odd and distinct. The $\Delta \in \Delta^{\max}(\pi_0)$ each determine orderings 
\[
P^\Delta_v := (a_{1,v}, \dotsc, a_{k+r, v})
\]
of $Q^{\max}(\pi_0)$ for each $v$ such that $\psi^\Delta_v = P^\Delta_v$. Let
\[
\pi_{0,v} = (p^\Delta_{i,v}, q^\Delta_{i,v})_i \in \beta^{-1}(P^\Delta_v) \subseteq \mc P(p_v, q_v)
\]
and for each $d_j$, let $i^\Delta_v(d_j)$ be the index $i$ such that $d_j = a_{i,v}$.  

Then, for $\Delta \in \Delta^{\max}(\pi_0)$, $\eta^{\psi_\infty^\Delta}_{\pi_0}(S^\natural_\Delta) = 1$ if and only if
\[
t_j := \sum_{v \in \infty} (i^\Delta_v(d_j) - 1 + q^\Delta_{i^\Delta_v(d_j), v} + \chi_4(d_j))
\]
is even for each $1 \leq j \leq k$ and $G_\infty$ satisfies the parity conditions from $\ref{unraminnerforms}$. 
\end{lem}

\begin{proof}
The condition on $t_j$ comes from checking that $\eta_{\pi_0}^{\psi_\infty^\Delta}(s_d) = 1$ for $d > 1$. The second comes from checking that $\eta_{\pi_0}^{\psi_\infty^\Delta}$ factors through $\mc S_\Delta$. 
\end{proof}

Beware that it is important to keep track of whether $G_v = U(p,q)$ or $U(q,p)$ as an extended pure inner form to compute the characters $\eta^{\psi_\infty}_{\pi_0}$.

\subsection{Explicit Limit Multiplicities} \label{section on LM results}
Now that we understand $\Delta(\pi_0)$, we can compute our main result: exact limit multiplicities for odd GSK-maxed $\pi_0$.
\begin{thm}\label{mainexact}
Let $G$ be a pure inner form of $G^* \in \wtd {\mc E}_\el(N)$. Choose
\[
f^\infty_{\mf n} = f_{S_1} \bar \1_{K^{G,S}(\mf n)}
\]
as in Section \ref{asymptoticsetup}  with $S_0 = \emptyset$ (in particular: only split places divide $\mf n$ and $G$ and therefore $E/F$ is unramified at all finite places as in Lemma \ref{unraminnerforms}). 

Pick a cohomological representation $\pi_0$ of $G_\infty$ that is odd GSK-maxed and for 
\[
\Delta = (T_1, 1, \lb_1, \eta_1), (1, d_i, \lb_i, \eta_i)_{2 \leq i \leq k} \in \Delta^{\max}(\pi_0),
\] 
let $\lb_1(\Delta) = \lb_1$. Let $T_1$, $d_i$, $R$ and $L$ be the common values of $T_1$, $d_i$, $R(\Delta)$ and $L(\Delta)$ over $\Delta^{\max}(\pi_0)$.  Then there are $A,B,C,D,E$ with $C \geq 1$ such that if $|\mf n| \geq Dq_{S_1}^{E \kappa}$:
\begin{multline*}
|\mf n|^{-R} \Gamma_L(\mf n)^{-1} \sum_{\pi \in \mc{AR}_\disc(G)} \1_{\pi_\infty = \pi_0} \tr_{\pi^\infty}(f^\infty_{\mf n})
\\
=  \f{\vol(H(F) \bs H(\A_f))}{\vol(K^S_H)}  \lf(\sum_{\Delta \in \Delta^{\max}(\pi_0)}\1_{\eta_{\pi_0}^{\psi^\Delta_\infty}(\mc S_\Delta) = 1} \f{\dim \lb_1(\Delta)}{|\Pi_\disc(\lb_1(\Delta))|} \ri) \\
 \times \lf(f_{S_1}^{H_1}(1) \prodf_{i > 1} \int_{H_{i, \der, S_1}}  f^{H_i}_{S_1}(h) \, dh \ri)
+ O(|\mf n|^{-C} q_{S_1}^{A + B \kappa}),
\end{multline*}
where $H_i = H(T_i, d_i, \lb_i, \eta_i)$ and $H = H_1 \times \prod_{i > 1} H_i^\ab$ are both constant over $\Delta \in \Delta^{\max}(\pi_0)$.

Finally, recall that the condition on $\eta_{\pi_0}^{\psi^\Delta_\infty}$ can be checked as in Lemma \ref{etacharoddGSK}, that 
\[
R = (k-1) + \f 12 \lf(N^2 + T_1^2 - \sum_{i \geq 2} d_i^2, \ri),
\]
and that
$ 
L = T_1, 1^{(k-1)}, -1^{(k-1)}.
$ 
\end{thm}

\begin{proof}
Apply Proposition \ref{mpi0form}. Then Corollary \ref{maxedshapeE} allows us to apply Theorem \ref{mpi0delform} to compute the main terms $m(\pi_0, \Delta, f^\infty)$ for $\Delta \in \Delta^{\max}(\pi_0)$.  We compute~$R = R(\pi_0)$ by Corollary \ref{Deltamaxenum}. Theorem \ref{mpi0bound} bounds all the terms $m(\pi_0, \Delta', f^\infty)$ for $\Delta' \notin \Delta^{\max}(\pi_0)$. 
\end{proof}

\begin{note}
Since the different shapes $\Delta = (T_i, d_i, \lb_i, \eta_i)_i \in \Delta^{\max}(\pi_0)$ only differ in their $\lb_i$-coordinates, they correspond to the same map
\[
\mc S_\Delta : \wh H^{\ur, \temp}_v \to \wh G^\ur_v
\]
Therefore, as in \eqref{DeltaPlancherelEquidistribution},  we can compute
\[
f_{S_1}^{H_1}(1) \prodf_{i > 1} \int_{H_{i, \der, S_1}}  f^{H_i}_{S_1}(h) \, dh = \mu^{\pl(\pi_0), \ur}_{S_1}(\wh f_{S_1}),
\]
where $\mu^{\pl(\pi_0), \ur}$ is the common value of the $\mu^{\pi(\Delta), \ur}$. This interprets Theorem \ref{mainexact} as an unramified ``Plancherel'' equidistribution theorem for the local component $\pi_S$ as in \cite{ST16}. Beware that $\mu^{\pl(\pi_0), \ur}$ can have support on the non-tempered spectrum. 
\end{note}

For general representations, we have an upper bound:
\begin{thm}\label{mainupper}
Let $G$ be an extended pure inner form of $G^* \in \wtd {\mc E}_\el(N)$. Choose
\[
f^\infty_{\mf n} = \varphi_{S_0} f_{S_1}  \bar \1_{K^{G,S}(\mf n)}
\]
as in \S\ref{asymptoticsetup} with $f_{S_1}$ and $\varphi_{S_0}$ arbitrary (In particular, only split places divide $\mf n$).  

Pick a cohomological representation $\pi_0$ and let $R_0(\pi_0)$ be defined as in the end of \S\ref{deltamaxalgo}
. Then if $\Delta(\pi_0) = \emptyset$,
\[
\sum_{\pi \in \mc{AR}_\disc(G)} \1_{\pi_\infty = \pi_0}  = 0.
\]
Otherwise, there are $A,B,D,E$ such that as long as $|\mf n| \geq Dq_{S_1}^{E \kappa}$:
\[
\sum_{\pi \in \mc{AR}_\disc(G)} \1_{\pi_\infty = \pi_0} \tr_{\pi^\infty}(f^\infty_{\mf n}) = O(|\mf n|^{R(\pi_0)} q_{S_1}^{A + B \kappa}).
\]
\end{thm}

\begin{proof}
Apply Proposition \ref{mpi0form} and Theorem \ref{mpi0bound}. 
\end{proof}

The above result applies to any CM extension $E/F$ and any extended pure inner form~$G$.
Assuming Conjecture \ref{simpleshapesconj}, the exponent $R(\pi_0)$ can be improved to a tighter $R_0(\pi_0)$. Furthermore, in the sum from Theorem \ref{mpi0form1}, at least one of the $\Delta_1 \times \Delta_2$ for each $\Delta \in \Delta^{\max}(\pi_0)$ satisfies that the $d_i$ assigned to each $\Delta_i$ all have the same parity. Therefore, if we assume the even stronger Conjecture \ref{simpleshapesconjeps}  and don't have an obstruction from the multiplicity formula (like the $\eta^\Delta_{\pi_0}$ condition from Theorem \ref{mainupper}), we expect $R_0(\pi_0)$ to be optimal.

\section{Examples and Corollaries}\label{sectioncorollaries}
This section contains examples and corollaries of the main Theorems \ref{mainexact} and \ref{mainupper}, including applications to Sato-Tate equidistribution in families, the Sarnak-Xue density hypothesis, and the cohomology of locally symmetric spaces. We first recall all our various growth rates for reader's convenience:

Let $\pi_0$ be on a rank-$N$ group. Each $\star(\pi_0)$ is a maximum of $\star(\Delta) = \star((T_i, d_i, \lb_i, \eta_i)_i)$ for $\Delta \in \Delta(\pi_0)$. Then there are two formulas:

First, the provable growth rate  from Proposition \ref{simpleshapesstronger}, which will be used in all theorem statements is:
\begin{multline}\label{eqn Rpi that we prove}
R(\pi_0) = \max_{\Delta \in \Delta(\pi_0)} \f 12 \lf(N^2 + \sum_i T_i^2 d_i \ri) - \sum_{i : T_i = 1} \lf(\f12 (d_i^2 + d_i) - 1 \ri) - \\ \sum_{i : T_i = 2} \lf(3 d_i - 3 \ri) - \sum_{\substack{i : T_i = 3 \\ d_i > 1}} \lf((5 - 10^{-100}) d_i +5 \ri).
\end{multline}
It will be compared to the conjecturally sharp growth rate from \ref{simpleshapesconj}:
\[
R_0(\pi_0) = \max_{\Delta \in \Delta(\pi_0)} \f 12 \lf(N^2 - \sum_i T_i^2 d_i^2 \ri) + \sum_i \lf(T_i^2 + \f12 T_i (T_i - 1)(d_i^2 - 1) \ri).
\]

\subsection{Examples}
First, we work out what Theorem \ref{mainexact} says in some simple cases. To discuss infinitesimal characters, let $\lb_1, \dotsc, \lb_{n-1}$ be the fundamental weights of $\GL_n$:
\[
\lb_i = (1^{(i)}, 0^{(N-i)}),
\]
in the standard basis of $X^*(T)$ corresponding to the entries of a diagonal matrix. Define $\lb_0, \lb_n$ similarly for indexing purposes. Two weights are character twists of each other if they differ by a multiple of $\lb_n$ and the half-sum of positive roots is
\[
\rho_n = \lf(\f{n-1}2, \f{n-3}2, \dotsc, \f{1-n}2 \ri).
\]

\subsubsection{Example 1: parallel case for $U(N-1,1)$.}
For the simplest example with non-tempered representations at infinity, assume:
\begin{itemize}
    \item $\deg F/\Q = d$ is even,
    \item $G_\infty \cong U(N-1,1)^d$ as is allowed by \S\ref{ACgroups},
    \item $\pi_\infty \cong \pi_0^d$ with $\pi_0 = ((1,0)^{(r)}, (k-1,1), (1,0)^{(N - k - r)})$ for $k > 1$ odd
    , 
    \item $S_1 = \emptyset$.
\end{itemize}
Then 
$\Delta^{\max}(\pi_0)$ is a singleton 
\[
 (N-k, 1, (\lb_{1,v})_v, \eta_1), (1, k, (\lb_{2,v})_v, \eta_2), 
\]
with $\lb_{1,v}$ a character twist of $k\lb_r + \rho_{n-k}$ on $\GL_{n-k}$ and $\lb_{2,v}$ the infinitesimal character of a $1$-dimensional irrep. We recall that the discrete $L$-packet at infinitesimal character $\lb_{2,v}$ has size $N$. 
Finally, $d$ is even, so $\eta_{\pi_\infty}^{\psi_\infty} = (\eta_{\pi_0}^{\psi_v})^d$ 
is trivial.

Denoting by $\pi_{k \lb_r}$ the finite-dimensional representation of $\GL_{N-k}(\C)$ with highest weight~$k \lb_r$, we compute
\begin{multline}\label{example1}
|\mf n|^{-(N(N-k) + 1)} L_{k,1,-1}(\mf n)^{-1} \sum_{\substack{\pi \in \mc{AR}_\disc(G) \\ \pi_\infty = \pi_0}}  \dim((\pi^\infty)^{K(\mf n)}) \\
= \f1{N^d}\dim(\pi_{k \lb_r})^d \tau'(U_{E/F}(N-k) \times U(1)) + O(|\mf n|^{-C} q_{S_1}^{A + B\kappa})
\end{multline}
as an asymptotic count of automorphic \emph{forms} of level $\mf n$ corresponding to $\pi_\infty$. Recall that $\tau'$ is the modified Tamagawa number from \eqref{tamagawa}. 

Note that the ``masses'' (i.e. relative abundances in the automorphic spectrum)
\[
\f1{N^d}\dim(\pi_{k \lb_r})^d
\]
depend on $r$ even though the corresponding $\pi_0$ have the same infinitesimal character and come from the same Levi in the Langlands classification. This is because the way in which the infinitesimal character divides up between the blocks of this Levi also matters. 

More specifically, $\dim(\pi_{k \lb_r})$ is largest for $r$ close to $(N-k)/2$ and decreases to~$1$ towards $r = 0$ or $N-k$. For another prespective, all representations considered are cohomological of degree $d(N-k)$. The representations with Hodge weights closer to $\left(\frac{d}2(N-k),\frac{d}2(N-k)\right)$ have larger masses, whereas the ones whose weights are closer to $(0,d(N-k))$ and $(d(N-k),0)$ are rarer.


\subsubsection{Example 2} As a slight complication, now assume:
\begin{itemize}
    \item $\deg F/\Q = d$ is odd,
    \item $N \not \equiv 0 \pmod 4$ and, to satisfy the conditions in \S\ref{ACgroups}:
    \[
    G_\infty \cong \begin{cases} U(N-1,1)^d & N \equiv 2,3 \pmod 4 \\ U(1,N-1)^d & N \equiv 1 \pmod 4 \end{cases}
    \]
  
    \item $\pi_\infty = \pi_0^d$ with
    \[
    \pi_0 = \begin{cases}
     ((1,0)^{(r)}, (k-1,1), (1,0)^{(N - k - r)}) & N \equiv 2,3 \pmod 4 \\
     ((0,1)^{(r)}, (1,k-1), (0,1)^{(N - k - r)}) & N \equiv 1 \pmod 4
    \end{cases}
    \]
    for $k > 1$ odd and $r+ k \leq N$,
    \item $S_1 = \emptyset$. 
\end{itemize}
This situation is similar to the first example, but no longer necessarily has $\eta_{\pi_\infty}^\Delta = 1$. 

Using the test from Lemma \ref{etacharoddGSK}, $\eta_{\pi_\infty}^\Delta = 1$ if and only if
\[
r + \begin{cases} 1 & N \equiv 2,3 \pmod 4 \\ N-1 & N \equiv 1 \pmod 4 \end{cases} + \begin{cases} 1 & k \equiv 0,1 \pmod 4 \\ 0 & k \equiv 2,3 \pmod 4 \end{cases} \equiv 0 \pmod 2
\]
is even. We write this condition as
\begin{equation}\label{example2parity}
r + 1 + \chi_4(N) + \chi_4(k) \equiv 0 \pmod 2,
\end{equation}
with 
$\chi_4$ the character from Lemma \ref{etacharoddGSK} and think of it as a parity condition on $r$. 

If \eqref{example2parity} holds, then we have the same result \eqref{example1} as in the previous example. Otherwise, we have that
\begin{equation}
|\mf n|^{-(N(N-k) + 1)} L_{k,1,-1}(\mf n)^{-1} \sum_{\substack{\pi \in \mc{AR}_\disc(G) \\ \pi_\infty = \pi_0}}  \dim((\pi^\infty)^{K(\mf n)}) = O(|\mf n|^{-C} q_{S_1}^{A + B\kappa}).
\end{equation}
There are two consequences. First, we only have exact asymptotics in the case where $r$ satisfies 
\eqref{example2parity}. Second, the asymptotic growth rate of counts of forms can be different for representations coming from the same Levi in the Langlands classification.

These phenomena are caused by an obstruction from the multiplicity formula in our specific setup: the automorphic representations counted have unramified local components at all non-split finite places. In particular, they all correspond to the trivial character on the component group. Furthermore, all those representations came from parameters $\psi$ with $\eps_\psi = 1$. In total, the multiplicity formula requires $\eta^\psi_{\pi_0} = 1$ for the packet $\Pi_\psi$ to contribute to the multiplicity of $\pi_0$. 

More surprisingly, growth rates can differ even within an $L$-packet. If $\pi_0 = ((p_i, q_i))_i$ on $U(p,q)$, it can be seen from the description in \cite[\S 6]{VZ84} that some other members of its (pseudo- and therefore true) $L$-packet can be produced by reversing some pairs $(p_i, q_i)$ such that we remain in~$\mc P(p,q)$. For an $L$-packet like
\[
\{((2,1),(0,1)), ((1,2),(1,0))\}
\]
on $U(2,2)$, only one member satisfies the parity condition from Lemma \ref{etacharoddGSK}. This is starkly different from the discrete-at-infinity case in \cite{Dal22} and caused by our dominant contribution to growth rates coming from shapes $\Delta$ with non-trivial $\mc S_\Delta$.

\subsubsection{Example 3}\label{example3}
We will also consider an example where there is only one non-compact place. Assume:
\begin{itemize}
    \item $\deg F/\Q = d$ with a fixed place $v_0 \in \infty$.
    \item $G_\infty \simeq U(p,q) \times U(N,0)^{d-1}$ where
    \[
    q \equiv \begin{cases}
    0 \pmod 2& d \text{ even or } N \equiv 0,1 \pmod 4 \\
    1 \pmod 2 & d \text{ odd and } N \equiv 2,3 \pmod 4
    \end{cases} 
    \]
    to satisfy the conditions of \S\ref{ACgroups}. The $U(p,q)$ factor is at $v_0$. 
    \item $\pi_\infty = \pi_0 \times \1^{d-1}$ where $\1$ is the trivial representation and \[
    \pi_0 = ((p_i, q_i))_i, \qquad n_i = p_i + q_i.
    \] 
    is odd GSK-maxed with the same infinitesimal character as $\1$.
    \item $S_1 = \emptyset$. 
\end{itemize}
We need some more combinatorial parameters
\begin{itemize}
    \item $d_j$ for $1 \leq i \leq k$ are the distinct non-$1$ values among the $n_i$,
    \item If $d_j = n_{i(d_j)}$, let $r_j = i(d_j) - \#\{i(d_{j'}) < i(d_j)\}$
    \item $M = N - \sum_j d_j$.
\end{itemize}
The $\Delta \in \Delta^{\max}(\pi_\infty)$ are then all of the form
\[
(M, 1, (\lb_v)_v, \eta), (1, d_j, (\lb_{j,v})_j, \eta_j)_j,
\]
where all the $\lb_{j,v}$ are infinitesimal characters of $1$-dimensional irreps and $\lb_{v_0}$ is a character twist of  
\[
\rho_M + \sum_j d_j \lb_{r_j}.
\]
The possible choices for each $\lb_v$ with $v \neq v_0$ are in bijection with reorderings $(n_{v,i})_i$ of $(n_i)_i$. We can define $i_v(d_j)$ and $r_{v,j}$ from $(n_{v,i})_i$ similar to the original definitions of $i(d_j)$ and $r_j$. Then $\lb_v$ is equal to a character twist of
\[
\rho_M + \sum_j d_j \lb_{r_{v,j}}.
\]

The condition on the character from Lemma \ref{etacharoddGSK} then reduces to 
\begin{equation}\label{ex3parity}
i(d_j) + \sum_{v \neq v_0} i_v(d_j) \equiv d(\chi_4(d_j) - 1)  + q_{i(d_j)} \pmod 2
\end{equation}
for all $1 \leq j \leq k$. There is always a set of reorderings $(n_{v,i})_i$ that satisfy \eqref{ex3parity}: the only way there couldn't be is if the $d_j$ are all the $n_i$ and a parity condition from summing \eqref{ex3parity} over all $j$ fails \footnote{This statement boils down to a combinatorial puzzle about filling in a $(\#\{n_i\}-1) \times k$ grid of squares black or white such that each row has $\lfloor (\#\{n_i\}-1)/2 \rfloor$ black squares and the first $d-1$ columns have a fixed parity of black squares.}. 
However, the parity condition on $q$ implies that the condition from summing always holds.

It follows that this is analogous to Example 1 in that we do not need to check the character condition form Lemma \ref{etacharodd} to get lower bounds. In particular the asymptotic \eqref{example1} holds with a much more complicated factor replacing the $N^{-d} \dim(\pi_{r\lb})^d$.

\subsection{Growth of Cohomology} Our limit multiplicities can be translated into computations of upper and lower bounds for the growth cohomology of arithmetic lattices in $G_\infty$. Recall that for a cocompact lattice $\Gamma$ in a Lie group $G_\infty$ with maximal compact $K_\infty$ and Lie algebra $\fg$, Matsushima \cite{Ma67} computed the cohomology of~$\Gamma$ with coefficients in the finite-dimensional representation~$F$ of $G$: 
\begin{equation} \label{eq Matsushima's formula}
    H^*(\Gamma,F) = \sum_{\pi \in G_\infty^\vee} m(\pi, \Gamma) H^*(\fg,K_\infty; \pi \ten F^*).
\end{equation}
Here $G_\infty^\vee$ is the unitary dual of $G_\infty$, the integer $m(\pi,\Gamma)$ is the multiplicity of $\pi$ in~$L^2(\Gamma \dom G_\infty)$, and $H^*(\fg,K_\infty)$ is the $(\fg,K)$-cohomology of $\pi\otimes F^*$ 
; see \cite{BW00}. We say $\pi$ is \emph{cohomological with coefficients in $F$} if $H^*(\fg,K_\infty; \pi \ten F^*)$ is nontrivial.

We will restrict our groups $G$ so that $G_\infty = U(p,q) \times U(N,0)^a \times U(0,N)^{b}.$ Following Lemma \ref{unraminnerforms}, such $G$ exist for all values of $N$, though possibly not over all unramified extensions $E/F$.  We will abuse terminology and say that a degree $i$ of cohomology appears in an $A$-packet $\Pi_\psi$ if it contains a representation with non-vanishing cohomology in degree $i$.

\subsubsection{Lower bounds}
We give a sample application of Theorem \ref{mainexact} to lower bounds on growth of cohomology. First, as a direct consequence of Lemma \ref{LemmaCohomology} and of  \S\ref{cohomological A-packets and partitions}, we find: 

\begin{lem}\label{lemmaLowestDegree}
       Let $G_v = U(p,q)$ with $p+q = N$ and $\min(p,q) = r$. Let $1<d\leq N$ be odd. 
\begin{itemize}
    \item[(i)] Let $Q \in \mathcal P(N)$ be an ordered partition with one entry equal to $d$ and all others equal to $1$. 
    Then the  lowest degree of cohomology appearing in the packet $\Pi_{\psi_{Q},v}$ is:
    \[i= i(d,N,r) = \begin{cases}
        r(N-r)-\frac{d^2-1}{4} & r \geq \frac{d-1}{2} \\ 
        r(N-d) & r\leq\frac{d-1}{2}.
    \end{cases}
    \]
    \item[(ii)] If $r \leq \frac{d-1}{2}$, the degree $i$ is achieved by a unique $\pi_i \in \Pi_{\psi_Q,v}$. If \[Q = (\overbrace{1,\dotsc,1}^s,d,1,\dotsc,1) \in \mathcal P(N), \] then the Hodge weight of $\pi_i$ in degree $i = r(N-d)$ is:
    \[(a,b) = \begin{cases}
        (rs,r(N-d-s)) & G = U(N-r,r)\\
        (r(N-d-s),rs) & G = U(r,N-r).   
    \end{cases}  \]
\end{itemize}       
\end{lem}

From this we deduce the following:

\begin{thm}
\label{cohmain}
Let $E/F$ be unramified at all finite places, and let $G$ be an extended pure inner form of $G^*$ unramified at all finite places, isomorphic to $U(p,q)$ at one infinite place $v_0$, and compact at all the others. 
For $\fn$ divisible only by primes that split in $E$, let $\Gamma(\fn) = G(F) \cap K^G(\fn)$ be a lattice in $U(p,q)$ of level $\fn$. 
Let~$N =p+q$ and $r=\min(p,q)$. Let $j \not\equiv N \mod 2$ be such that $j \leq |p-q|-1$. Then 
\[ 
\dim H^{rj}(\Gamma(\fn), \BC) \gg |\fn|^{Nj}. 
\] 
The same bound holds for each $H^{rk,r(j-k)}(\Gamma(\fn), \BC)$ for $0 \leq k \leq j.$
\end{thm}
\begin{proof}
By Lemma \ref{lemmaLowestDegree}, it suffices to give lower bounds on multiplicities for the representation $\pi_{rj} = \pi_{rj,v_0}\ten {\bf 1}^{[F:\BQ]-1}$. Since $\pi_{rj}$ is odd GSK-maxed by the congruence condition on $j$, Theorem \ref{mainexact} gives exact multiplicities with $R=Nj+1$ provided that there exists $\Delta \in \Delta(\pi_0)$ with \[\eta_{\pi_{rj}}^{\psi_\infty^\Delta}(\mathcal{S}_\Delta)\equiv 1.\]  Since we are in the setup of Example \ref{example3}, there automatically exists such a~$\Delta$.

This gives a lower bound for the cohomology of the disconnected locally symmetric spaces $Y(\fn) = G(F) \dom G(\BA)/K(\fn)K_\infty$, of which $\Gamma(\fn) \dom G_\infty/K_\infty$ is one connected component. Though the different connected components of $Y(\fn)$ are not necessarily isomorphic, their cohomology is equidimensional by Lemma \ref{lemma connected components} below. It then follows from \cite[\S 2]{deligne1971travaux} that \[\dim H^i(Y(\fn), \BC) = |\pi_0(T(F)\dom T(\BA)/\nu(K_\infty K(\fn)))| \dim H^i(X(\fn), \BC),\] for $T = G/G^{\der}$. Since $|\pi_0(T(\BA)/T(F)\nu(K_\infty K(\fn)))| \gg |\fn|^{1-\epsilon}$, we conclude. 
%
%
\end{proof}

\begin{lem} \label{lemma connected components}
Let $G^*$ be an inner form of $G = U_{E/F}(N)$ and $K$ be a compact open subgroup of $G(\BA^\infty)$. Let \[ Y_K = G(F) \dom G(\BA) /K_\infty K \] be the locally symmetric space of level $K$. Let $X_{K,\gamma}$ denote the connected components of $Y_K$, indexed by representatives of $G(F) \dom G(\BA^\infty)/K$. Then for any $\gamma, \gamma'$ we have 
\[ 
\dim H^i(X_{K,\gamma},\BC) = \dim H^i(X_{K,\gamma'}, \BC). 
\]   
\end{lem}
\begin{proof}
Let $G'$ be the derived subgroup of $G$ with $T = G/G'$ and quotient map $\nu$. 
We have  
\[ 
Y_K = \bigsqcup X_{K,\gamma} := \bigsqcup_\gamma G'(F) \dom G'(\BA) / \gamma K\gamma^{-1}(K_\infty \cap G'(F_\infty)) 
\] 
for a finite set of representatives $\gamma$ of $G(F) \dom G(\BA)/K$.  Consider the space $S_G = G(F)\dom G(\BA)/K_\infty = \varprojlim_K Y_K$: it carries an action of $G(\BA_f)$ by translation which, following \cite[\S 2]{deligne1971travaux}, is transitive on connected components. Moreover, since $G'$ is simply connected, $G'(\BA_f)$ is the stabilizer of a given connected component. In fact, the space $S_G$ is the induction of the connected component 
\[
S_{G'} = G'(F)\dom G'(\BA)/(K_\infty \cap G'_\infty),
\]
in a precise sense laid out in \cite[\S 2.7]{deligne1979varietes}. This translates actual induction when passing to cohomology: 
\[ 
H^*(S_G, \BC) = \Ind_{G'(\BA_f)}^{G(\BA_f)} H^*(S_{G'}, \BC), 
\] 
see \cite{New13}, where the cohomology of $S_G$ is a smooth representation realized as a direct limit over the cohomology of the $S_K$ (viewed as $K$-fixed vectors) in the natural way.
Fixing a set of representatives $\gamma$ for $G(\BA_f)/G'(\BA_f)$, we deduce that
\[ 
H^*(S_G,\BC)\mid_{G'(\BA_f)} = \bigoplus_\gamma H^*(S_{G'}, \BC)^\gamma, 
\]
where the superscript $\gamma$ denotes conjugating the representation. In particular, for the subgroup $K$ of $G'(\BA_f)$, we have 
\[ 
\dim\left((H^*(S_{G'}, \BC)^\gamma)^K\right) = \dim(H^*(S_{G'}, \BC)^{\gamma K \gamma^{-1}}) = \dim H^*(X_{\gamma K \gamma^{-1}}, \BC), 
\] 
where the last equality follows from Matsushima's formula. 
\end{proof}

We compare these lower bounds with previous results: in \cite{MS19} Marshall-Shin showed that, when $d < N-1$, $\dim H^d(\Gamma(\fn),\BC) \ll_\epsilon |\fn|^{Nd+\epsilon}$  for split-level lattices~$\Gamma(\mf n)$ in $U(N-1,1)$ and conjectured that this bound was sharp. Theorem~\ref{cohmain} specialized to $(p,q) = (N-1,1)$ then gives:
\begin{cor}
If $d \not\equiv N \mod 2$, then Marshall-Shin's bounds are sharp.
\end{cor}

\subsubsection{Upper bounds} \label{section upper bounds algo}

As an example for the upper bounds, fix:
\begin{itemize}
\item an arbitrary CM extension $E/F$,
\item an extended pure inner form $G$ of $G^* \in \wtd {\mc E}_\el(N)$, isomorphic to $U(p,q)$ at one infinite place $v_0$ and compact at all other infinite places,
\item a finite set of finite places $S$ at which $G$ is split,
\item
an open compact $U^{S, \infty}$ away from $S$ and $\infty$,
\item
$\mf n$ an ideal supported over $S$. We compute asymptotics as $|\mf n| \to \infty$, 
\item
$\Gamma(\mf n) = G(F) \cap U^{S, \infty} K^G_S(\mf n)$ a lattice.
\end{itemize}
As before, we are interested in $H^{a,b}(\Gamma(\mathfrak n), \C)$. To compute this:
\begin{enumerate}
\item 
We use Lemma \ref{LemmaCohomology} to enumerate all $\pi_0 \in \mc P_1(p,q)$ contributing to $H^{a,b}$. 
\item
We use the algorithm at the end of \S\ref{deltamaxalgo} 
 and the formula \eqref{eqn Rpi that we prove} to compute all the $R(\pi_0)$. Let the maximum value be $R(a,b)$.
\end{enumerate}
Then:
\begin{prop}
For all $\eps > 0$,
\[
\dim H^{a,b}(\Gamma(\fn), \C) \ll_\eps |\mf n|^{R(a,b) - 1 + \eps}.
\]
\end{prop}

\begin{proof}
This follows from applying Theorem \ref{mainupper} to each of the $\pi_0$ from (1) above, applying Matsushima's formula, and then using the bounds on the number of connected components to reduce to a single connected component of the adelic quotient. 
\end{proof}

\begin{ex}[Lowest Degree]
    Let $r = \min(p,q)$. Then $r$ is the lowest degree of cohomology that is not guaranteed to vanish for local reasons. There are only two nontrivial cohomological representations in degree $r$, and they have weights $(r,0)$ and $(0,r)$ respectively. Then $R(0,r) = R(r,0) = p+q$. 
\end{ex}

\begin{ex}[Upper and Lower Bounds for $U(N-2,2)$]
Consider the case case where the non-compact factor is $U(N-2,2)$ and choose a degree of cohomology $0 < i < 2(N-2)$. We also assume $N > 6$ for simplicity. 



An analysis of possible shapes then shows that $H^i(\Gamma(\mf n), \C) \ll_\eps |\mf n|^{R_i} - 1$ where
\[
R_i = \max \begin{cases}
     Ni/2 + 1 & $i$ \text{ even}, \\
     1/2(i + 5/2)^2 + N(N - i - 3) + 23/8 & $i$ \text{ even, } i \geq N -2 \\
    (i/2 + 1)^2 + 7/4  & $i$ \text{ odd, } i \geq N -2 \\
    0  
\end{cases}.
\]
Moreover, $R_i$ gives an exact asymptotic whenever the dominant shape is odd-GSK---i.e. whenever $i \equiv 2N \pmod 4$ and the first case achieves the maximum. 
In addition, assuming Conjecture \ref{simpleshapesconj}, the exact asymptotic should be
\[
R_{0,i} = \max \begin{cases}
     Ni/2 + 1 & $i$ \text{ even}, \\
     (i/2 + 1)^2 + 3 & $i$ \text{ even, } i \geq N -2 \\
    (i/2 + 1)^2 + 7/4  & $i$ \text{ odd, } i \geq N -2 \\
    0
\end{cases}.
\]
\end{ex}

\subsubsection{A note on vanishing of cohomology}
Historically, another case of interest has been examples where certain degrees of cohomology identically vanish; see for example \cite[\S 15]{Ro90} for rank three unitary groups and \cite{Clo93Coho} for general rank. 

This can be achieved by choosing $G$ that is a division algebra at some place $v_0$ where $E/F$ is split. Then various shapes could potentially never contain parameters that are relevant on $G$ (as in \cite[\S1.3.7]{KMSW14}) and therefore never contribute to~$\mc{AR}_\disc$. In particular, certain cohomological $\pi_0$ can simply not appear at infinity. 

To reconcile with our bounds, this mechanism is hidden within the computation of endoscopic transfers in Proposition \ref{mpi0form}. It manifests through transfers at $v_0$ either vanishing, immediately zeroing out terms in \ref{mpi0form}, or having certain constant terms vanish, zeroing out terms after further split-place transfers computed by Lemma \ref{splitconstant}. We do not take into account this potential tightening in Theorem \ref{mainupper} for two reasons: first, it is already known and second, it is better thought of as a corollary of \cite{KMSW14} (after plugging some computations with the parameterization of cohomological representations on unitary groups recalled in \S\ref{upqparam}) instead of a consequence of the new techniques here.

\subsection{Sato-Tate Equidistribution in Families}\label{sectionSatoTate}
We prove an averaged Sato-Tate result similar to Theorem 9.26 in \cite{ST16} using that our main theorem \ref{mainexact} has error bounds of the same strength in $f_{S_1}$. 

We consider families of automorphic representations with infinite component equal to an odd GSK-maxed $\pi_0$. Their Satake parameters will not equidistribute with respect to the Sato-Tate measure on $G$, but rather with respect to the pushforward of the Sato-Tate measure on a smaller group related to the $\Delta \in \Delta^{\max}(\pi_0)$. Otherwise, this section will follow \cite{ST16} extremely closely.

\subsubsection{Sato-Tate measures}
First we recall the definition of Sato-Tate measures from \cite[\S\S3,5]{ST16} (the full details can be found there). Choose a place $v$ of $F$ over which $G$ is unramified. 
Let $A \subseteq T$ be a maximally split torus of $G_v$ and maximal torus containing it. Let $\wh A_c, \wh T_c$ be the maximal compact subgroups in the Langlands duals $\wh A, \wh T$. 
Via the Satake isomorphism, we have a parameterization of the tempered, unramified dual of $G_v$:
\[
\widecheck G_v^{\ur, \temp} \simeq \Om_{F_v} \bs \wh A_c \simeq \Om_{F_v} \bs \wh T_c/(\id - \Frob_v) \wh T_c.
\]
This space is also the same as $\wh G_v$-conjugacy classes in $\wh K_{\Frob_v} \rtimes \Frob_v$ where $\wh K_{\Frob_v}$ is the maximal $\Frob_v$-invariant compact subgroup of $\widecheck G_v$. 

Of course, not every $v$ produces the same Frobenius action on $\wh G$.  To deal with this, if $G$ splits over~$E$, let $\Gamma = \Gal(E/F)$. Then for each $\theta \in \Gamma$, let
\[
\wh T_{c, \theta} := \Om_{F_v} \bs \wh T_c/(\id - \theta) \wh T_c.
\]
For $\gamma \in \Gamma$, $t \mapsto \gamma t$ canonically identifies $T_{c, \theta}$ with $T_{c, \gamma \theta \gamma^{-1}}$ so $\wh T_{c, \theta}$ can be considered to only depend on the conjugacy class of $\theta$. This is therefore a uniform description of $\widecheck G_v^{\ur, \temp}$ whenever $\Frob_v = \theta$. 

\begin{dfn}
The Sato-Tate measure $\mu^\ST_\theta := \mu^\ST_\theta(G)$ on $\wh T_{c, \theta}$ is the quotient under $\wh G$-conjugation of the Haar measure on $\wh K_\theta \rtimes \theta$ with total volume $1$.
\end{dfn}
This should  be thought of as the ``most canonical''  measure to put on $T_{c, \theta}$. 

Now, let $\mc V_F(\theta)$ be the set of places $v$ of $F$ such that $\Frob_v = \theta$. For $v \in \mc V_F(\theta)$, the Plancherel measure on $\widecheck G_v^{\ur, \temp}$ (normalized so that a maximal compact of $G_v$ has volume $1$) gives another measure $\mu^{\pl, \ur}_v := \mu^{\pl, \ur}(G_v)$ on $\wh T_{c, \theta}$.
\begin{lem} \label{lem Sato-Tate convergence}
Let $v \in \mc V_F(\theta)$ be a sequence of places such that $q_v \to \infty$. Then there is weak convergence $\mu^{\pl, \ur}_v \to \mu^\ST_\theta$. 
\end{lem}

\begin{proof}
This follows by explicit formulas \cite[Prop 3.3]{ST16} and \cite[Lem 5.2]{ST16}. 
\end{proof}

\subsubsection{Normalizations}
While the unramified tempered spectrum $T_{c, \theta}$ is essentially the same across different primes, the full unramified spectrum is not. For example, the Satake parameter of the trivial representation on $\GL_n$ depends on $q_v$. We therefore need to normalize appropriately.

The full unramified spectrum can be described as
\[
\widecheck G_v^\ur \simeq \Om_{F_v} \bs \wh A \simeq \Om_{F_v} \bs \wh T / (\id - \Frob_v) \wh T.
\]
Define the $v$-normalization homomorphism
\[
\td v : \C^\times \to \C^\times : re^{i \theta} \mapsto e^{\log_{q_v}(r) + i \theta}
\]
in polar coordinates. Through the canonical isomorphism $\wh A = \Hom(X_*(A), \C^\times)$, postcomposing with $\td v$ gives $v$-normalization maps on $\wh G_v^\ur$:
\[
\td v : \Om_{F_v} \bs \wh A \to \Om_{F_v} \bs \wh A.
\]
We note two key properties (which motivated the construction):
\begin{itemize}
    \item $\td v$ is the identity on $\Om_{F_v} \bs \wh A_c = \widecheck G_v^{\ur, \temp}$,
    \item $\td v(s_v^I)$ is constant over $v \in \mc V_F(\theta)$ if $s_v^I$ is the Satake parameter of the trivial representation of $G_v$.
\end{itemize}

\subsubsection{Equidistribution}
We can now state and prove the equidistribution result. Fix $G$ an unramified extended pure inner form of $G^* \in \wtd {\mc E}_\el(N)$. Note that at all finite places, $G$ splits over $E$. Fix:
\begin{itemize}
    \item an odd GSK-maxed cohomological representation $\pi_0$ of $G_\infty$,
    \item $\theta \in \Gal(E/F)$,
    \item a sequence $v_i \in \mc V_F(\theta)$ (i.e. either all split or all non-split),
    \item a sequence of ideals $\mf n_i$ of $\mc O_F$ relatively prime to $v_i$.
\end{itemize}

The different $\Delta = (T_i, d_i, \lb_i, \eta_i)_i \in \Delta^{\max}(\pi_0)$ differ only in their $\lb_i$-coordinates and therefore correspond to the same map
\[
\mc S_{\Delta,v} : \widecheck H_v^{\ur, \temp} \into \widecheck G_v^\ur
\]
as in formula \eqref{satakepushforward}. Furthermore, the common group $H_v$ as in \eqref{shapeLembedding} is the same as from Theorem \ref{mainexact}. 

Because they involves the Satake parameter of the trivial representation, the maps $\mc S_{\Delta, v}$ depend on $v$. However,
\[
\td{\mc S}_{\Delta} := \td v \circ \mc S_{\Delta,v}
\]
is independent of $v \in \mc V_F(\theta)$ by the two key properties of $v$-normalization. 

Therefore, for $\theta \in \Gal(E/F)$, we can define the pushforward
\[
\mu^{\ST(\pi_0)}_\theta := \mu^{\ST(\pi_0)}_\theta(G) := (\td{\mc S}_\Delta)_* (\mu^\ST_\theta(H)).
\]
Beware that this is a measure on the full ($v$-normalized) unramified dual $\widecheck G_v^\ur$ for $v \in \mc V_F(\theta)$ instead of just the tempered part. 

Finally, for each $i$, define the empirical distribution on $\widecheck G_\theta^\ur$:
\[
\mu^{\pi_0}_{\mf n_i, v_i} := \sum_{\pi \in \mc{AR}_\disc(G)} \1_{\pi_\infty = \pi_0} \dim((\pi^\infty)^{K^G(\mf n_i)}) \delta(\td \sigma_{\pi_{v_i}}).
\]
Here, $\delta(\td \sigma_{\pi_{v_i}})$ is the delta-measure at the $v_i$-normalized Satake parameter 
\[
\td \sigma_{\pi_{v_i}} := \td v_i (\sigma_{\pi_{v_i}}).
\] 
Then:
\begin{thm}[Sato-Tate Equidistribution in Families]\label{SatoTate}
Recall the notation for constants in the statement of Theorem \ref{mainexact}. Assume that $|\mf n_i|$ grows faster than any power of $q_{v_i}$. Then for all continuous $\wh f$ on $\widecheck G_\theta^\ur$, 
\[
|\mf n_i|^{-R} \Gamma_L(\mf n_i)^{-1} \mu^{\pi_0}_{\mf n_i, v_i}(\wh f) \to C(\pi_0) \mu^{\ST(\pi_0)}_\theta(\wh f)
\]
as $i \to \infty$, with normalizing constant
\[
C(\pi_0) = \f{\vol(H(F) \bs H(\A_f))}{\vol(K^S_{H})} \sum_{\Delta \in \Delta^{\max}(\pi_0)}\1_{\eta_{\pi_0}^{\psi^\Delta_\infty}(\mc S_\Delta) = 1} \f{\dim \lb_1(\Delta)}{|\Pi_\disc(\lb_1(\Delta))|}.
\]
\end{thm}

\begin{proof}
By the Weierstrass approximation argument in Remark 9.5 of \cite{ST16}, it suffices to show that
\[
|\mf n_i|^{-R} \Gamma_L(\mf n_i)^{-1} \mu^{\pi_0}_{\mf n_i, v_i}(\wh f_{v_i}) \to C(\pi_0) \mu^{\ST(\pi_0)}_\theta(\wh f_{v_i})
\]
for $f_{v_i} \in \ms H^\ur(G_{v_i})$ (note that this Hecke algebra is constant on $\mc V_F(\theta)$). We do this by applying Theorem \ref{mainexact} with $S_1 = \{v_i\}$. Note that the growth condition on $|\mf n_i|$ shows that we will eventually have $|\mf n_i| \geq D q_{v_1}^{E \kappa(f_{v_i})}$. 

After using formula \eqref{Tcharactershape} and Fourier inversion to get that
\begin{multline*}
f_{S_1}^{H_1}(1) \prodf_{i > 1} \int_{H_{i, \der, S_1}}  f^{H_i}_{S_1}(h) \, dh = \mc T_\Delta f_{S_1}(1) \\
= \mu^{\pl, \ur}(H_v)(\wh{\mc T_\Delta f_{S_1}}) = \mu^{\pl, \ur}(H_v)(\wh f_{S_1} \circ \mc S_\Delta),
\end{multline*}
the argument follows exactly as that for 9.26 in \cite{ST16}, using Lemma \ref{lem Sato-Tate convergence} on each of the factors of $H_v$ and continuity of the map $\mc S_\Delta$. 
\end{proof}

We repeat an interpretation from the introduction: recall that as part of Langlands functoriality conjectures, every automorphic representation~$\pi$ on some $G/F$ should correspond to a group $H_\pi/F$ that is the smallest group it is a functorial transfer from. The Satake parameters $\sigma_{\pi_v}$ for $v$ ranging over a particular~$\mc V_F(\theta)$ are then expected to equidistribute according to a Sato-Tate distribution coming from $H_\pi$.

At the current time, actually finding $H_\pi$ appears out of reach. However, in reasonable families of automorphic representations, most $\pi$ should correspond to some fixed computable``maximal'' $H$. 
Therefore, if we look at Satake parameters over the entire family, we can hopefully prove an equidistribution-on-average result towards the Sato-Tate measure for this maximal $H$. 

This is conceptually what is happening here: most automorphic representations with $\pi_0$ at infinity come from group $H_{\pi_0} = H'$. Therefore the $\sigma_{\pi_v}$ ranging over both~$v$ and a reasonable family of such $\pi$ should equidistribute according to the Sato-Tate measures from $H'$. Unlike previous cases built off of \cite{ST16}, we are in a more complicated situation where this maximal $H_{\pi_0}$ isn't actually $G$ itself.

\subsection{Sarnak-Xue Conjecture}\label{sectionSX}
As an application of Theorem \ref{mainupper}, we prove certain cases of the Sarnak-Xue conjecture of \cite{SX91} for unitary groups. This conjecture is stated in terms of classical symmetric spaces instead of adelic quotients. Consider a reductive $G/F$, open compact $U \subseteq G^\infty$, and $\pi_0$ a unirrep of $G_\infty$. Let $\Gamma(U) = U \cap G(F)$ and let
\[
m(\pi_0, \Gamma(U)) := 
\dim \Hom(\pi_0, L^2(\Gamma(U) \bs G_\infty)).
\]
Note that $\Gamma(U) \bs G_\infty$ is a connected component of the adelic quotient
\[
Y(U) := G(F) \bs G(\A) / U.
\]
\begin{conj}[Cohomological Sarnak-Xue density hypothesis] \label{conj coho SX}
Let $U_i$ be a sequence of open compacts of $G^\infty$ decreasing to the identity. Then for all cohomological unirreps $\pi_0$ of $G_\infty$:
\[
m(\pi_0, \Gamma(U_i)) \ll_\eps \vol(\Gamma (U_i) \bs G_\infty)^{\f 2{p(\pi_0)} + \eps},
\]
where $p(\pi_0)$ is the infimum over $p \geq 2$ such that the $K$-finite matrix coefficients of $\pi_0$ are in $L^p(G_\infty)$.
\end{conj}
In this section, we will study this conjecture for $G$ unitary and~$U_i$ decreasing through increasing principal-congruence levels at split places. We will prove it for all $\pi_0$ that do not have a single particular representation on $U(2,2)$ as a factor. 

We mention that the work \cite{MS19} already has strong enough local bounds to achieve the Sarnak-Xue threshold. We do not improve upon these, only use them as input (through Lemmas \ref{T2bound} and \ref{T3bound}) for our more powerful global methods. 

In this section, we make frequent use of the parameterization introduce in \ref{section parameterization} and of the corresponding notation.






\subsubsection{Computing $p(\pi)$}
Before we can check Conjecture \ref{conj coho SX} for unitary groups, we first need to extend the computations of \cite{Ger20} in order to compute $p(\pi)$ in terms of our parameterization of cohomological representations. 

First, given a (possibly ordered)  bipartition $B = ((p_1, q_1), \dotsc, (p_r, q_r))$, let 
\[
\Xi(B) = (\chi_j(B))_j
\]
be the list of numbers obtained by setting $m_i = \min\{p_i, q_i\}$, $n_i = p_i + q_i$, concatenating the lists
\[
\bigsqcup_{i : m_i \neq 0} (n_i -1, n_i -3, \dotsc, n_i - 2m_i + 1)
\]
and reordering the result to be decreasing. For indexing purposes, we define $\chi_j(B) = 0$ for~$j$ out of bounds. Also define
\begin{equation}\label{eq combinatorial thing we want to bound}
    \sigma_j(B) = \sum_{k \leq j} \chi_k(B), \qquad j \geq 1.
\end{equation}

\begin{prop}\label{ppi}
Let $G = U(p,q)$ and $\pi_0$ be the cohomological representation of $G$ associated to $B = ((p_1,q_1),\dotsc,(p_r,q_r))$.  Then,
\[
\f 2{p(\pi_0)} \geq 1 - \max_i \lf\{\f{\sigma_i(B)}{i(N-i)} \ri\}.
\]
\end{prop}
\begin{proof}

We first recall a formula for~$p(\pi)$ based on results of \cite[\S\S 7-8]{Kna01} when $\pi=J(S,\sigma,\nu)$ is a Langlands quotient. To describe such a quotient we need: 
\begin{itemize}
    \item $S_0$, a minimal parabolic of $G$, with Langlands decomposition $S_0=M_0A_0N_0$, whose respective subgroups have Lie algebras $\fm_0$, $\fa_0$, and $\fn_0$,
    \item $\alpha_1,\dotsc,\alpha_{\dim \fa_0}$ the simple roots of $\fa_0$ in $\fg$, and $\omega_1,\dotsc,\omega_{\dim \fa_0}$ the basis of $\fa_0$ dual to the $\alpha_i$,
    \item $\rho_0$ the corresponding half-sum of positive roots of $\fa_0$ in $\fg$,
    \item $S = MAN$, a parabolic subgroup of $G$ standard with respect to $S_0$, with Lie algebras $\fm$, $\fa$, and $\fn$,
    \item a discrete series representation $\sigma$ of $M$,
    \item a weight $\nu \in \fa^*$ such that $\ip{\nu}{\alpha}>0$ for all roots $\alpha$ of $\fa$ in $\fn$.
\end{itemize}
Then the parabolic induction $I(S,\sigma,\nu)$ has a unique Langlands quotient $J(S,\sigma,\nu).$ We have a direct sum decomposition \[\fa_0 = \fa \oplus \fa_M,\] where $\fa_M$ is the Lie algebra of the maximal split torus of $M$. Define $\nu_0 \in \fa_0^*$ by extending $\nu$ by zero on $\fa_M$. 
 Proposition 5.13 of \cite{Ger20} then deduces from \cite[\S\S 7-8]{Kna01} the inequality
\begin{equation} \label{matrix coefficient ineq for langlands quotients} 
p(J(S, \sigma, \nu)) \leq \inf \left\{p \geq 2 \mid p > \frac{2 \langle \rho_0, \omega_j \rangle}{\langle \rho_0-\nu_0, \omega_j \rangle} \text{ for all }\omega_j \right\}. 
\end{equation}

Next, we write the representation $\pi_0$ corresponding to the bipartition~$B$ as a Langlands quotient following \cite[\S 6]{VZ84}.
Begin with the Levi subgroup 
\[ 
L = U(p_1,q_1) \times \cdots \times U(p_r,q_r) 
\] 
associated to $\pi_0$. Let $K$ be the maximal compact of $G$ 
and let $(K \cap L) A_L N_L$ the Iwasawa decomposition of $L$. 
Define $\nu_L$ to be the half-sum of positive roots of~$A_L$ acting on $N_L$, let $Z = \{ \alpha \text{ a root of }\fa_L \text{ in }\fg \mid \ip{\alpha}{\nu_L} = 0\}$, and let 
\[ 
A = \textstyle \bigcap_{\alpha \in Z} \ker \alpha \subseteq A_L. 
\] 
Define $M$ to be the centralizer of $A$ in $G$. Define also $M_L$ to be the centralizer of $A_L$ in $G$, and let $S_L = M_LA_LN_L$ be any parabolic subgroup of $G$ with respect to which~$\nu_L$ is dominant. Then by construction, there is a unique parabolic subgroup $S$ with Levi $MA$ and containing $S_L$. Let $\nu = \nu_L \mid_A$. Then following \cite[Thm 6.16]{VZ84}, there exists a discrete series representation $\sigma$ of $M$ such that $\pi_0 = J(S,\sigma,\nu)$.

We now use this construction to compute the pairings in \eqref{matrix coefficient ineq for langlands quotients}; to do so, we write everything in coordinates. To begin, we choose a posteriori a minimal parabolic subgroup $S_0$ for which $S$ is standard. Define numbers $m_\star = \min\{p_\star, q_\star\}$ and $n_\star = p_\star + q_\star$. Each $U(p_\star,q_\star)$ has a minimal parabolic corresponding to the partition 
\[
(1^{(m_\star)}, n_\star - 2m_\star ,1^{(m_\star)}) 
\]
and a maximal split torus isomorphic to $\R^{m_\star}$, which we embed in a maximal torus $T_\star$. Fix coordinates in the standard way: i.e. so that we can write
\[
X^*(T_\star) = \Z\langle e_1, \dotsc, e_{n_\star} \rangle
\]
with simple roots $\alpha_i = e_i - e_{i+1}$. We can therefore realize as a vector:
\begin{multline*}
\rho_0 = \frac12\lf(n-1, n-3, \dotsc, n-2m + 1, 0^{(n-2m)},  -n+2m - 1, \dotsc, -n+3, -n+1 \ri)
\end{multline*}
in $X^*(T)$. In $X^*(A_0)$, this becomes
\[
\rho_0 = \rho_{p,q} = (n-1, n-3, \dotsc, n - 2m + 1).
\]
Similarly, $\nu_L \in X^*(A_L)$ is the concatenation of sequences
\begin{equation} \label{eq nu-ell}
\bigsqcup_{k=1}^r \rho_{p_i, q_i}
\end{equation}
reordered to be decreasing (the reordering comes from choosing $S_0$ with respect to which $S_L$ is standard). The subtorus $A \subset A_L$ is then chosen so that~$\ip{\mathrm{Re}(\nu_L)}{\alpha}>0$ for all simple roots $\alpha$ of $A$. We have further direct sum decompositions
\[
A_0 =  A' \oplus A_L = A' \oplus A'' \oplus A
\]
where $A'$ is the maximal split torus of $L$ and $A' \oplus A''$ is that of $M$. The extension of $\nu$ by $0$ to $A_L$ is then just $\nu_L$ again. Let $\nu_0$ be the common extension by $0$ to $A_0$, obtained in coordinates by adding a string of zeros to \eqref{eq nu-ell}. Following \eqref{matrix coefficient ineq for langlands quotients}, we have 
\[  
\frac{2}{p(\pi_0)}\geq 1-\max_i\left\{\frac{\ip{\nu_0}{\omega_i}}{\ip{\rho_0}{\omega_i}}\right\}. 
\]

By symmetry of the $\nu_0$ and $\rho_0$, the maximum value is achieved for some $i \leq m$. In this case, we check that $2\langle \rho_0, \omega_i \rangle = i(N-i)$ and $2\langle \nu_0, \omega_i \rangle = \sigma_i(B)$. 
\end{proof}

\subsubsection{Some combinatorial lemmas}
We next need some involved but elementary combinatorial bounds, this time for the $\sigma_i(B)$ defined in \eqref{eq combinatorial thing we want to bound}. Once again, part of the complexity of this section is due to our use the suboptimal bound $R(\Delta)$ from Corollary \ref{simpleshapesstronger} instead of the conjectural optimal bound $R_0(\Delta)$ from \ref{simpleshapesconj}. 

First, for a (possibly ordered) partition $Q$, define $\sigma_i(Q) = \sigma_i(B)$ for the $B = ((p_i, q_i))_i \in \beta^{-1}(Q)$ such that $|p_i - q_i| \leq 1$. In particular, for any $B$ and for all $i$, $\sigma_i(B) \leq \sigma_i(\beta(B))$. Next:

\begin{lem}\label{QdQdprime}
Let $d < N \in \Z^+$.
\begin{enumerate}
\item
If
\[
Q_d := \lf(d^{(\lfloor N/d \rfloor)}, N - d\lfloor N/d \rfloor \ri),
\]
then for all $Q = (n_1, \dotsc, n_r)$ a partition of $N$ with each $n_j\leq d$, $\sigma_i(Q) \leq \sigma_i(Q_d)$ for all $i$. 
\item
Assume $N \geq 2d$. If
\[
Q'_d := \begin{cases}
 \lf(d^{(\lfloor N/d \rfloor - 1)},  d-1, N - d\lfloor N/d \rfloor + 1 \ri) & N \not\equiv -1 \pmod d \\
 \lf(d^{(\lfloor N/d \rfloor-1)}, d-1, d-1, 1 \ri) & N \equiv -1 \pmod d
\end{cases},
\]
then for all $Q = (n_1, \dotsc, n_r)$ a partition of $N$ with each $n_j\leq d$ and at most $\lfloor N/d \rfloor-1$ parts of size $d$ (i.e. not equal to $Q_d$), $\sigma_i(Q) \leq \sigma_i(Q'_d)$ for all $i$. 
\end{enumerate}
\end{lem}

\begin{proof}
For the first claim, choose such $Q = (n_1, \dotsc, n_r) \neq Q_d$ without loss of generality in decreasing order. Then $n_r, n_{r-1} < d$ so $Q' = (n_1, \dotsc, n_{r-2}, n_{r-1} + 1, n_r - 1)$ also satisfies the conditions.  

Next, $\Xi(Q')$ differs from $\Xi(Q)$ by replacing pairs of numbers
\[
(n_r - 1, n_r -1), (n_r - 3, n_r - 3), \dotsc
\]
with corresponding pairs
\[
(n_r, n_r - 2), (n_r - 2, n_r - 4), \dotsc
\]
Switching any pair $(a,b)$ in $\Xi(Q)$ with $a \geq b$ to $(a+r, b-r)$ will increase or keep equal every $\sigma_i(Q)$. Since $n_{r-1} \geq n_r$, this is a sequence of such pair-switches together with some strict increases of coordinates. Therefore $\sigma_i(Q') \geq \sigma_i(Q)$ for all $i$. 

Repeating this process until producing $Q_d$ proves the first part. The second part follows by similar argument.
\end{proof}

\begin{lem}\label{QdQdprimebound}
With notation from Lemma \ref{QdQdprime}: 
\[
\max_i \lf\{\f{\sigma_i(Q_d)}{i(N-i)} \ri\} = \f {d-1}{N-\lfloor N/d \rfloor}.
\]
Furthermore, if $N \geq 2d$,
\[
\max_i \lf\{\f{\sigma_i(Q'_d)}{i(N-i)} \ri\} = \f {d-1}{N-\lfloor N/d \rfloor + 1}.
\]
\end{lem}

\begin{proof}
This is by computer check.
\end{proof}

This gives our final results:
\begin{prop}\label{Sdensitygrowthrate}
Let $B \in \mc P_1(p,q)$ with $p+q = N$ and let $\pi_0$ be the cohomological representation of $U(p,q)$ corresponding to $B$. Let $Q \in Q^{\max}(B)$ and assume $\beta(B) \neq (2,2)$. Let $R(Q)$ be as in Corollary \ref{simpleshapesstronger}. Then
\[
(N^2 - 1)\lf(1 - \max_i \lf\{\f{\sigma_i(B)}{i(N-i)} \ri\} \ri) \geq R(Q) - 1
\]
with equality only if $Q$ has a single element.
\end{prop}

\begin{proof}
It suffices to prove the bound with 
\[
\max_i \lf\{\f{\sigma_i(\beta(B))}{i(N-i)}  \ri\} \geq \max_i \lf\{\f{\sigma_i(B)}{i(N-i)} \ri\}
\]
instead. Let $d$ be the maximal element of $\beta(B)$. First, if $N < 2d$, then $\beta(B) = (d, (a_i)_i)$ for $\sum_i a_i < d$. Then by Lemmas \ref{QdQdprime} and \ref{QdQdprimebound},
\[
\max_i \lf\{\f{\sigma_i(\beta(B))}{i(N-i)}  \ri\} \leq   \f{d-1}{N-1}.
\]
Furthermore, by Lemma \ref{enumsl2}, $Q$ has a part of size $d$ so by Lemma \ref{maxsl2},
\[
R(Q) \leq R(d, 1^{(N-d)}) = N(N-d) + 1.
\]
The result then follows from
\[
1 - \f {d-1}{N-1}  = \f{N-d}{N-1} > \f{N(N-d)}{N^2 - 1}.
\]

Next assume $N \geq 2d$. If $\beta(B) \neq Q_d$, then by Lemmas \ref{QdQdprime} and \ref{QdQdprimebound}, we have
\[
\max_i \lf\{\f{\sigma_i(\beta(B))}{i(N-i)}  \ri\} \leq \max_i \lf\{\f{\sigma_i(\beta(Q'_d))}{i(N-i)}  \ri\} = \f {d-1}{N-\lfloor N/d \rfloor + 1}.
\]
Using again that $R(Q) \leq N(N-d) + 1$,
the result follows from
\[
1 - \f {d-1}{N-\lfloor N/d \rfloor + 1} > \f{N(N-d)}{N^2 - 1}.
\]

If on the other hand $\beta(B) = Q_d$, then
\[
\max_i \lf\{\f{\sigma_i(\beta(B))}{i(N-i)}  \ri\} = \f {d-1}{N-\lfloor N/d \rfloor}.
\]
In addition, by Lemma \ref{enumsl2}, we have $Q=Q_d$ so
\[
R(\beta(B)) \leq \bar R(Q) = \f12 \lf(N^2 + \lf \lfloor \f Nd \ri \rfloor^2 d + N - d\lf \lfloor \f Nd \ri \rfloor \ri).
\]
By a computer check, the desired inequality 
\[
1 - \f {d-1}{N-\lfloor N/d \rfloor} >  \f1{N^2 - 1} \lf( \f12 \lf(N^2 + \lf \lfloor \f Nd \ri \rfloor^2 d + N - d\lf \lfloor \f Nd \ri \rfloor \ri) -1 \ri)
\]
is true except for the cases

\[ \beta(B) = (d,d), \quad  \beta(B) = (d,d,1), \quad \beta(B) = (2,2,2). \]
All these cases except $(2,2)$ can be checked by using $R(Q)$ instead of $\bar R(Q)$. 
\end{proof}
Extending to all number fields $F$:
\begin{cor}\label{SdensitygrowthrateE}
Let $G$ be an extended pure inner form of $G^* \in \wtd {\mc E}_\sm(N)$ and let $\pi_0 = \prod_v \pi_v$ be a cohomological representation of $G_\infty$ such that $\Delta^{\max}(\pi_0) \neq \emptyset$ and where each $\pi_v$ corresponds to bipartition $B_v \in \mc P_1(p_v, q_v)$.

Then, for all $B_v$ such that $\beta(B_v) \neq (2,2)$, 
\[
(N^2 - 1)\lf(1 - \max_i \lf\{\f{\sigma_i(B_v)}{i(N-i)} \ri\} \ri) \geq R(\pi_0) - 1
\]
with equality only if $\pi_\infty$ is a character.
\end{cor}

\begin{proof}
For all $v$, $R(\pi_v) \geq R(\pi_0)$ since it is a maximum over a larger set by Lemma \ref{enumsl2E}. 
The result then follows from Proposition \ref{Sdensitygrowthrate}.
\end{proof}

\subsubsection{Sarnak-Xue density}
We can now return to the original setup and specialize to unitary groups. Let $G$ be an extended pure inner form of $G^* \in \wtd {\mc E}_\sm(N)$. We choose:
\begin{itemize}
    \item a cohomological representation $\pi_0 = \prod_v \pi_v$ of $G_\infty$,
    \item a finite set of places $S_0$ containing all places where $G$ is ramified,
    \item an ideal $\mf n$ relatively prime to $S_0$ (we will compute asymptotics as $\mf n \to \infty$),
    \item an open compact $U_{S_0} \subseteq G_{S_0}$. 
\end{itemize}
We then define
\[
U_{\mf n} = U_{S_0} K^G(\mf n)
\]
using the principle congruence subgroups associated to $\mf n$. 

\begin{thm}[Cohomological Split-level Sarnak-Xue Density for Unitary Groups]\label{SarnakXue}
With setup as above, assume:
\begin{itemize}
    \item $\mf n$ is only divisible by places of $F$ that split in $E$.
    \item If $N=4$: for each $v$ with $G_v = U(2,2)$,  $\pi_v \neq ((1,1),(1,1))$. 
\end{itemize}
Then,
\[
m(\pi_0, \Gamma(U)) \ll_\eps \vol(\Gamma(U_{\mf n}) \bs G_\infty)^{\f 2{p(\pi)} + \eps}.
\]
(The $\eps$ may be removed if $\pi_0$ isn't a character). 
\end{thm}

\begin{proof}
As in \S1.1 of \cite{MS19}, $Y(U_{\mf n})$ contains $\gg_\eps |\mf n|^{1 - \eps}$ copies of $\Gamma(U_{\mf n}) \bs G_\infty$ and 
\[
\vol(\Gamma(U_{\mf n}) \bs G_\infty) \gg_\eps |\mf n|^{N^2 - 1 + \eps}.
\] 

The bound on connected components gives us that
\begin{multline*}
m(\pi_0, \Gamma(U_{\mf n})) \ll_\eps |\mf n|^{-1 + \eps} \dim \Hom(\pi_0, L^2(Y(U_{\mf n}))) \\
=  |\mf n|^{-1 + \eps} \sum_{\pi \in \mc{AR}_\disc(G)} \1_{\pi_\infty = \pi_0} \dim((\pi^\infty)^{U_{\mf n}}).
\end{multline*}
We can bound the sum by Theorem \ref{mainupper} with $S_1 = \emptyset$ and $\varphi_{S_0} = \bar \1_{U_{S_0}}$ to get
\[
m(\pi_0, \Gamma(U_{\mf n})) \ll_\eps |\mf n|^{R(\pi_0)-1 + \eps}.
\]
Next, Proposition \ref{ppi} computes
\[
\f 2{p(\pi)} = \min_v \f 2{p(\pi_v)} \geq \min_v \lf(1 - \max_i \lf\{\f{\sigma_i(B_v)}{i(N-i)} \ri\} \ri),
\]
where each $\pi_v$ corresponds to bipartition $B_v$ and with a strict inequality if $\pi_0$ isn't a character. The result follows from applying the bound in Corollary \ref{SdensitygrowthrateE} and the volume estimate.
\end{proof}

Of course, either by varying $\varphi_{S_0}$ beyond just the indicator function of $U_{S_0}$ or by using non-trivial $S_1$, we can prove similar results for very general weighted counts of representations. These don't have as clean a statement in terms of the classical symmetric spaces $\Gamma(U_i) \bs G_\infty$ however. 

The leftover $\pi_0 = ((1,1),(1,1))$ is likely an artifact of $R(\Delta)$ not being optimal. However, provably improving it enough seems to be hard---see \cite{Mar16}  (the conjectural optimal bound $R_0(\Delta)$ of Conjecture \ref{simpleshapesconj} would of course be enough). 

\subsubsection{Examples}
We compute some small cases where $\pi_0$ on a group of rank $N$ is the same value $\pi_v$ at all infinite places $v$ so that $R(\pi_0) = R(\pi_v)$. Let $\pi_v = B \in \mc P_1(p_v,q_v)$. Let $Q$ be the unordered partition corresponding to $\beta(B)$. If $\pi_v$ is GSK-maxed, let 
\[
Q = (d_1, \dotsc, d_r, 1^{(k)})
\]
with $d_1 > d_2 > \cdots > d_r > 1$. Then we can compute
\[
R(\pi_v)-1 = R_0(\pi_v)-1 = \f12\lf(N^2 + k^2 - \sum_i d_i^2\ri) + (r-1),
\]
so we get
\[
m(\pi_0, \Gamma(U_{\mf n})) \ll_\eps |\mf n|^{\f12\lf(N^2 + k^2 - \sum_i d_i^2\ri) + (r-1) + \eps}.
\]
This is conjecturally the exact exponent and provably so in the odd GSK-case for unramified $E/F$ and $U_i$ as in Theorem \ref{mainexact}. We also have
\[
\f 2{p(\pi_0)} \geq \f{N-d}{N-1}
\]
so the Sarnak-Xue bound asks for $m(\pi_0, \Gamma(U_{/mf n}))$ to be asymptotically less than
\[
\vol(\Gamma(U_{\mf n}) \bs G_\infty)^{\f 2{p(\pi_0)}} \gg_\eps |\mf n|^{(N+1)(N-d) - \eps},
\]
which is always true. If $r = 1$, we save a factor of $|\mf n|^{N-d}$ over the Sarnak-Xue bound. If we keep $n$ and $k$ fixed but increase $r$, the saving is even larger.

When $\pi_v$ isn't GSK-maxed, the formulas are much more complicated and $R(\pi_v) \neq R_0(\pi_v)$. Table \ref{grcomp} lists some values based on $Q = \beta(\pi_v)$. For each $Q$, we list the maximum possible values of $R(\pi_0)-1$ and $R_0(\pi_0)-1$, which in our setup only depend on $Q$. These are the provable and conjectural exponents on $|\mf n|$ in the growth rate of $m(\pi_0, \Gamma(\fn))$ respectively. We also list the target exponent from the Sarnak-Xue density bound and the ``trivial bound'' growth rate $N^2-1$ when $\pi_0$ is discrete series. Finally, we italicize cases where $Q$ not being GSK causes a failure of Lemma \ref{maxsl2}. 

Our $R(Q)$ beats target growth rate in every case except the bolded number when $Q = (2,2)$. The improvement is often large, though not in some cases like $(3,3), (4,4)$ and $(2,2,2,2)$. The conjectural optimum $R_0(Q)$ is usually much smaller still. 

We in particular want to point out the case of $(2,2,2,2,1,1)$ vs $(2,2,2,2,2)$ where \ref{maxsl2} fails even for the conjectural bound $R_0$. This suggests very strange behavior---for example, different asymptotic growth rates for $H^{a,b}$ with the same sum of $a+b$, or where the average Sato-Tate distribution for the family may differ depending on whether one weights the average by counts of automorphic forms or representations.  

\begin{table}[ht]
\centering
\caption{Comparison of Growth Rates}
\begin{tabular}{c||cccc}
Arthur-$\SL_2$ & provable: & conjectural: & SX goal: & trivial: \\
$Q$ & max $R(\pi_0)-1$ & max $R_0(\pi_0)-1$ & $2(N^2-1)p(\pi_v)^{-1}$ & $N^2-1$ \\ \midrule
(2,2) & \textbf 8 & 6 & \textbf{7.5} & 15\\
(2,2,1) & 13 & 11 & 16 & 24\\
(2,2,2) & 21 & 17 & 23.33 & 35\\
(2,2,1,1) & \textit{21} & 18 & 26.25 & 35\\
(3,3) & 17 & 11 & 17.5 & 35\\
(2,2,2,1) & 28 & 24 & 36 & 48\\
(3,3,1) & 24 & 18 & 28.8 & 48\\
(3,2,2) & 21 & 19 & 32 & 48\\
(2,2,2,2) & 47 & 33 & 47.25 & 63\\
(2,2,2,1,1) & \textit{47} & 33 & 50.4 & 63\\
(4,4) & 30 & 18 & 31.5 & 63\\
(3,3,3) & 43 & 32 & 53.33 & 80\\
(3,2,2,2) & 40 & 36 & 60 & 80\\
(5,5) & 47 & 27 & 49.5 & 99\\
(2,2,2,2,2) & 74 & 54 & 79.2 & 99\\
(2,2,2,2,1,1) & \textit{74} & \textit{54} & 82.5 & 99\\
\end{tabular}
\label{grcomp}
\end{table}

\bibliographystyle{amsalpha}
\bibliography{Tbib}

\end{document}